\pdfoutput=1 

\documentclass[10pt,letterpaper,reqno,twoside]{amsart}
\usepackage{LBT-macros,LBT-style}
\usepackage{setspace}
\calclayout

\usepackage[margin=20pt,font=footnotesize,labelfont=bf,skip=6pt]{caption}
\setitemize{topsep=0em}
\setenumerate{topsep=0em}
\titleformat{\section}
  {\scshape\large\bfseries\centering}
  {\thesection}{1em}{}

\titleformat{\subsection}
  {\normalfont\normalsize\bfseries\centering}
  {\thesubsection}{1em}{}

\titleformat{name=\subsection,numberless}[runin]
  {\normalfont\normalsize\bfseries}
  {}{0pt}{}[.]

\titleformat{\subsubsection}[runin]
  {\normalfont\normalsize\bfseries\centering}
  {\thesubsubsection}{5pt}{}[.]

\tikzcdset{
  cells={font=\everymath\expandafter{\the\everymath\displaystyle}},
  arrows={line width=0.6pt},
  diagrams={>={Straight Barb[scale=0.85]}},
  arrow style=tikz,
  every label/.append style = {font = \small},
  column sep/small=2em,
  column sep/normal=2.8em,
  column sep/large=3.5em
}

\title{A space level light bulb theorem in all dimensions}

\author{Danica Kosanovi\'c}
\address{LAGA, Université Sorbonne Paris Nord (Paris 13)}
\curraddr{Department of Mathematics, ETH Z\"urich}
\email{danica.kosanovic@math.ethz.ch}

\author{Peter Teichner}
\address{Max-Planck-Institut f\"ur Mathematik, Bonn}
\email{teichner@mac.com}

\begin{document}

\begin{abstract}
    Given a $d$-dimensional manifold $M$ and a knotted sphere $s\colon\S^{k-1}\hra\partial M$ with $1\leq k\leq d$, for which there exists a framed dual sphere $G\colon\S^{d-k}\hra\partial M$, we show that the space of neat embeddings $\D^k\hra M$ with boundary $s$ can be delooped by the space of neatly embedded $(k-1)$-disks, with a normal vector field, in the $d$-manifold obtained from $M$ by attaching a handle to $G$. This increase in  codimension significantly simplifies the homotopy type of such embedding spaces, and is of interest also in low-dimensional topology.
    In particular, we apply the work of Dax to describe the first interesting homotopy group of these embedding spaces, in degree $d-2k$. In a separate paper we use this to give a complete isotopy classification of 2-disks in a 4-manifold with such a boundary dual.
\end{abstract}

\maketitle

\begin{spacing}{0.5}
  \tableofcontents
\end{spacing}

\section{Introduction and survey of results}\label{sec:intro}

\subsection{A space level light bulb theorem for arbitrary dimensions} \label{subsec-intro:space-LBT}
Fix a $d$-dimensional manifold $M$ and an embedding $s\colon \S^{k-1}\hra \partial M$, and let $\Emb_s(\D^k,M)$ denote the space of neat embeddings of the $k$-disk into $M$ that restrict to $s$ on the boundary. In this paper {all manifolds are smooth, compact, connected, oriented, and have nonempty boundary, and all embeddings are smooth}. A smooth map $K\colon X\to Y$ of manifolds is \emph{neat} if it is transverse to $\partial Y$ and $K^{-1}(\partial Y)=\partial X$.
\begin{equation}\label{setting}
\begin{minipage}{0.9\textwidth}
    For a \emph{setting with a (framed geometric) dual (in the boundary)} we assume that there exists an embedding $G\colon\S^{d-k}\hra\partial M$ with trivialized normal bundle, such that $s\pitchfork G$ is a single positive point, and we fix such a framed sphere $G$.
\end{minipage}
\end{equation}

In this setting, we denote by $M_G$ the $d$-manifold obtained from $M$ by attaching to $\partial M$ a $(d-k+1)$-handle along the framed sphere~$G$. Then $\partial M_G$ is the surgery on $\partial M$ along $G$, and since $s$ intersects $G$, we have $s=u_-\cup_{u_0} u_+$ for $u_0\colon\S^{k-2}\hra \partial M_G$, $u_-\colon\D^{k-1} \hra \partial M_G$, and $u_+\colon\D^{k-1} \hra M_G$ neat, see Figure~\ref{fig:new-fig}.

For  $K \in\Emb_s(\D^k,M)$ denote by $\nu K$ an open tubular neighborhood. Then there is a diffeomorphism $M \sm \nu K \cong M_G$ and hence the right hand side does not depend on $G$ (nor its framing), and the left hand side is independent of $K$. This follows from handle cancellation, see Lemma~\ref{lem:MG}. 

Any triple $(M,s,G)$ as above is obtained by starting with an arbitrary $d$-manifold $X$ and removing a tubular neighborhood of a neat embedding $u\colon\D^{k-1}\hra X$ to obtain $M$. Then $G$ is a meridian sphere to $u$, $X$ is diffeomorphic to $M_G$, and $u$ corresponds to $u_+$. To obtain $s$ we have to assume that $u_0=\partial u$ is also the boundary of an embedding $u_-\colon\D^{k-1}\hra\partial X$ and finally, $\Emb_s(\D^k,M)$ is nonempty if and only if $u$ is isotopic to $u_-$ rel.\ $u_0$. In this case, $M$ is also a boundary connected sum $X \natural (\S^{d-k} \times \D^k)$, with $G$ corresponding to $\S^{d-k}\times \{p\}$ and $K=\{p\} \times \D^k$ the chosen embedding.

\begin{figure}[!htbp]
    \centering
    \includegraphics[width=0.97\linewidth]{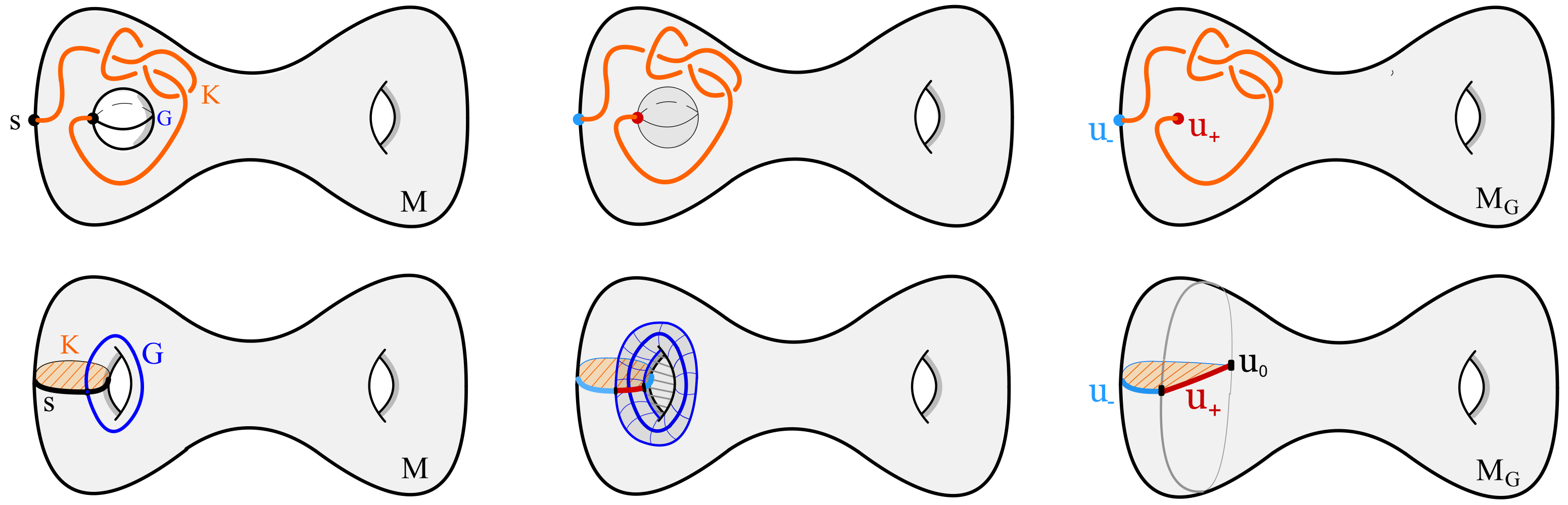}
    \caption{Correspondence of neat $k$-disks in $M$ with boundary $s$ and half-disks in $M_G$  for $k=1,2$, $d=3$.}
    \label{fig:new-fig}
\end{figure}
When $k=1$ and $d=3$ this is precisely the setting of the well-known ``light bulb trick'' in knot theory, see for example \cite[257]{Rolfsen}. Namely, $G$ is the ``light bulb'' to which a ``cable'' $\D^1\hra M$ connects on one end, while the other end is fixed in another component of $\partial M$, the ``ceiling'', see the top left of Figure~\ref{fig:new-fig}. The trick refers to the proof that such cables $K$ are isotopic if and only if they are homotopic, i.e.\ $\pi_0\Emb_s(\D^1,M)=\pi_1M$. Indeed, a homotopy from one knotted $K$ to another can be turned into an isotopy using the swinging motion around the light bulb $G$. See also Example~\ref{ex-intro:classical-LBT}.

Motivated by recent generalizations of the light bulb trick for $2$-spheres in $4$-manifolds \cite{Gabai-spheres,ST-LBT}, we prove the following theorem for disks in any dimensions $1\leq k\leq d$, and in~\cite{KT-4dLBT} use it to recover and generalize those results for $k=2,d=4$ (see also \cite{Gabai-disks} for partial results on $2$-disks).
\begin{maintheorem}\label{thm-intro:LBT-fib-seq}
    In the setting with a dual \eqref{setting}, any $\U\in\Emb_s(\D^k,M)$ leads to a fibration sequence
\[\begin{tikzcd}
    %   \Omega^{k}\S^{d-k}\arrow{r} & 
    \Emb_s(\D^k,M)\arrow{r}{\foliate_\U} & 
    \Omega\Emb_{u_0}(\D^{k-1},M_G)  \arrow{r}{\delta_{\ev_0}} &
    \Omega^{k-1}\S^{d-k}
\end{tikzcd}
\]
    where $\foliate_\U(K)$ is obtained by foliating $\D^k$ by a $1$-parameter family of $\D^{k-1}$ (rel.\ $\S^{k-2}$), applying $-\U$ and then $K$ to it and getting a loop based at $u_+\in\Emb_{u_0}(\D^{k-1},M_G)$.
\end{maintheorem}
Here $\Omega$ denotes a based loop space, and a sequence of based maps $F\overset{i}{\to} E \overset{p}{\to} B$ is a fibration sequence if $i$ factors through a weak homotopy equivalence to a homotopy fiber of $p$. In particular, we get a long exact sequence on homotopy groups.
For more details on $\foliate_\U$ we refer to Remark~\ref{rem:f-U}, and for $\delta_{\ev_0}$ to Remark~\ref{rem-intro:ThmA-from-ThmB}.

In fact, $\Emb_s(\D^k,M)$ is a loop space in general, as we explain next.

Firstly, a result of Cerf shows that this space is weakly homotopy equivalent to its subspace $\Emb_{s^\e}(\D^k,M)$, consisting of those embeddings that agree with a fixed one on an $\e$-collar of the boundary $\partial\D^k$ (see Proposition~\ref{prop:e-collar}). Secondly, if we extend the foliation of $\D^k$ from Theorem~\ref{thm-intro:LBT-fib-seq} to a $1$-parameter family of thickened disks $\D^{k-1}\times[0,\e]$, then $K\in \Emb_{s^\e}(\D^k,M)$ gives a path in the space $\Emb_{u_0^\e}^\e(\D^{k-1},M_G)$ of such \emph{$\e$-augmented $(k-1)$-disks} $\D^{k-1}\times[0,\e]\hra M_G$. We can similarly use $\U$ to complete this to a loop $\foliate_\U^\e(K)$ based at $u_+^\e$. Let us point out that the space $\Emb_{u_0^\e}^\e(\D^{k-1},M_G)$ is homotopy equivalent to $\Emb_{u_0^\e}^\uparrow(\D^{k-1},M_G)$, the space of $(k-1)$-disks equipped with a normal vector field, see Proposition~\ref{prop:normal-v-field}.
\begin{maintheorem}[Space level light bulb theorem for disks]\label{thm-intro:LBT-deloop}
    In the setting with a dual \eqref{setting}, any basepoint $\U\in\Emb_{s^\e}(\D^k,M)$ leads to a pair of inverse homotopy equivalences
    \[\begin{tikzcd}
        \foliate^\e_\U\colon\; \Emb_{s^\e}(\D^k,M) \arrow[shift left=3pt]{r}{}[swap]{\sim}
        &
        \Omega\Emb_{u_0^\e}^\e(\D^{k-1},M_G)\arrow[shift left=3pt]{l}{} \;\colon\amb_\U
      \end{tikzcd}
    \]
    where $\foliate^\e_\U$ is the $\e$-augmented foliation map, and for $n\geq 0$ the value of $\pi_n\amb_\U$ on an $n$-parameter family of isotopies $\S^1\to\Emb_{u_0^\e}^\e(\D^{k-1},M_G)$ is the $n$-parameter family of $k$-disks obtained by applying the parametrized ambient isotopy extension theorem.
\end{maintheorem}

\begin{remark}\label{rem-intro:ThmA-from-ThmB}
    Theorem~\ref{thm-intro:LBT-fib-seq} follows from Theorem~\ref{thm-intro:LBT-deloop} using Proposition~\ref{prop:forg-augm}, which says that forgetting the $\e$-augmentation is a fibration $\ev_0\colon \Emb_{u_0^\e}^\e(\D^{k-1},M_G) \sra \Emb_{u_0^\e}(\D^{k-1},M_G)$, with fiber $\Omega^{k-1}\S^{d-k}$ measuring the normal derivative in the $\e$-direction; then $\delta_{\ev_0}$ in Theorem~\ref{thm-intro:LBT-fib-seq} is the connecting map of this fibration. Here we again use Proposition~\ref{prop:e-collar} to replace $\Emb_{u_0^\e}(\D^{k-1},M_G)$ by the equivalent $\Emb_{u_0}(\D^{k-1},M_G)$.
\end{remark}

%These results could be attributed to Jean Cerf, since the method of proof appeared first in \cite{Cerf-applications}, republished as the appendix of his celebrated paper~\cite{Cerf-diffeos}, and is sometimes called the \emph{Cerf half-disk trick}. 

The main tool in these results, namely the translation to \emph{half-disks}, is due to Jean Cerf and appeared first in \cite{Cerf-applications}, republished as the appendix of his celebrated paper~\cite{Cerf-diffeos}.
Cerf did not discuss the generality in which his method applies: he used it only for $M=\D^k$, see Example~\ref{ex:Cerf} below. Our Theorems~\ref{thm-intro:LBT-fib-seq} and~\ref{thm-intro:LBT-deloop} are general results that use Cerf’s half-disk trick (and his parametrized ambient isotopy theorem) to arrive at interesting consequences for neat disks. Budney and Gabai outlined the case $M=\S^{d-k} \times \D^k$ in \cite[v1:Lem.3.4]{Budney-Gabai} and noted in \cite[v1:Rem.6.2]{Budney-Gabai} that this outline should be generalizable. For a history of related results we refer to Section~\ref{subsec:history}.

Mentioned half-disks are $k$-disks in~$M_G$ that restrict to $u_-^\e$ on the $\e$-collar of one half of the boundary, and to $u_+^\e$ on the $\e$-collar of the other half $\partial_+\D^k$, see the third column of Figure~\ref{fig:new-fig}. This space is homotopy equivalent to $\Emb_{s^\e}(\D^k,M)$, but is also the fiber of a restriction map from the space of half-disks for which $\partial_+\D^k$ is free to move in the interior of $M_G$. The latter space is contractible, so the fiber is the loop space on the base, implying Theorem~\ref{thm-intro:LBT-deloop}. The proof is in Section~\ref{sec:any-dim}, and a detailed outline in Section~\ref{sec:outline}.
\begin{remark}
    One might hope to use Theorem~\ref{thm-intro:LBT-deloop} repeatedly. However, the boundary condition for $(k-1)$-disks, $u_0\colon\S^{k-2}\hra M_G$, is null homotopic (it bounds $u_-
    \subseteq\partial M_G$), so cannot have a geometric dual.
\end{remark}
%In the reverse perspective, the theorem relates the (based loop space of the) space of embeddings $\D^{k-1}\hra X=M_G$, viewed as the core of a handle $h^{k-1}$ in some handle decomposition of $X$, with the space of cores of a cancelling handle $h^k$ (with a fixed attaching region), cf.\ the proof of Lemma~\ref{lem:MG}.
\begin{remark}
    The same method of proof applied to \emph{framed} half-disks yields a homotopy equivalence
\[
    %   \Omega^{k}\S^{d-k}\arrow{r} & 
    \Emb_{s'}^{\mathrm{fr}}(\D^k,M)\simeq
    \Omega\Emb_{u_0'}^{\mathrm{fr}}(\D^{k-1},M_G)
\]
    between spaces of framed embeddings, i.e.\ embeddings equipped with a trivialization of the normal bundle (and framed boundary conditions $s'$ respectively $u_0'$). Equivalently, we can thicken the embeddings to obtain (cf.\ Proposition~\ref{prop:normal-v-field}):
\[
    %   \Omega^{k}\S^{d-k}\arrow{r} & 
    \Emb_{\nu s}(h^{d-k},X)\simeq
    \Omega\Emb_{\nu u_0}(h^{d-k+1},X),
\]
    where $h^i=\D^i\times\D^{d-i}$ is a handle of index $i$ and $\nu s$ and $\nu u_0$ are neighborhoods of the attaching region.
    %In words, the space of handles in any handle decomposition of $X$ with the fixed attaching region $\nu s$ corresponds to the loop space of the space of their cancelling handles with attaching region $\nu u_0$.
\end{remark}

\subsection{Applications of the work of Dax}\label{subsec-intro:Dax}
The essence of Theorem~\ref{thm-intro:LBT-deloop} is that it increases the codimension for the embedding space and thus simplifies the computation of its homotopy groups. 
In particular, using Haefliger's double-point elimination~\cite{Haefliger-plong}, Jean-Pierre Dax~\cite{Dax} computed the homotopy groups $\pi_n\Emb_s(\D^k,M)$ in the ``metastable range'', i.e.\ in degrees $n< 2d-3k-3$, beyond the stable range $n< d-2k-1$ where they are as for immersions (in turn determined by Smale--Hirsch theory).  But now in the setting with a dual, Theorem~\ref{thm-intro:LBT-deloop} implies that the group
    \[\begin{tikzcd}
        \pi_n\foliate^\e_\U\colon\; \pi_n\Emb_{s^\e}(\D^k,M) \arrow[shift left=4pt]{r}{}[swap]{\scriptsize\cong}
        &
        \pi_{n+1}\Emb_{u_0^\e}^\e(\D^{k-1},M_G)\arrow[shift left=4pt]{l} \;\colon\pi_n\amb_\U
      \end{tikzcd}
    \]
can be computed using Dax's techniques for all $n< 2d-3k-1$ (equivalently, $n+1<2d-3(k-1)-3$). For example, when $d-k\leq2$ the range of Dax is empty, $n<1-k$, whereas we have $n<3-k$.

We remark that for $k=d-2$ the Goodwillie--Weiss embedding tower \cite{GKW}, which generalizes the work of Dax, converges for $\Emb_{u_0}(\D^{k-1},M_G)$, but it need not converge for $\Emb_s(\D^k,M)$.
    %  We have
    %  \[ T_\infty\Emb_s(\D^k,M)\to\Omega T_\infty\Emb_{u_0}^\e(\D^{k-1},M) \]

More precisely, Dax expresses the  homotopy group $\pi_n\big(\Imm_\partial(V,X),\Emb_\partial(V,X);u\big)$ of neat immersions, relative to neat embeddings of an $\ell$-manifold $V$ into a $d$-manifold $X$ (the boundary condition $u_0\colon\partial V\hra\partial X$ is omitted from the notation) as a certain bordism group, recalled in Theorem~\ref{thm:Dax}. However, to compute this explicitly requires more work, which to our knowledge has not been done prior to the present paper.
We identify this bordism group and the \emph{Dax invariant} for $n=d-2\ell$ and $V$ simply connected, as an isomorphism
\begin{equation}\label{eq-intro:Dax-invt}
        \Da\colon\;
        \pi_{d-2\ell}\big(\Imm_\partial(V,X),\Emb_\partial(V,X);u\big) \xrightarrow{\cong} \faktor{\Z[\pi_1X]}{\relations_{\ell,d}}
\end{equation}
where the group of relations $\relations_{\ell,d}$ is trivial for $\ell=1$, and given by $\langle g-(-1)^{d-\ell}g^{-1}: g\in\pi_1X\rangle$ for $\ell\geq2$. The map $\Da$ is the signed count of group elements at the double points of an associated generic immersion $\I^{d-2\ell}\times V\imra \I^{d-2\ell}\times X$,
see Theorem~\ref{thm:Dax-invt}. For $d-2\ell=0$ this resembles the result of Grant~\cite{Grant} that compares the invariant of Hatcher and Quinn~\cite{Hatcher-Quinn} (who study whether an immersion is regularly homotopic to an embedding for $2d-3\ell-3\geq0$) to Wall's self-intersection invariant, see Remark~\ref{rem:Grant}.

We then specialize to $V=\D^\ell$ and prove the following (restated as Theorem~\ref{thm:Dax-final}, that includes $d-2\ell=0$).

\begin{maintheorem}\label{thm-intro:Dax}
    Let $X$ be a $d$-manifold and consider the space $\Emb_\partial(\D^\ell,X)$ with a basepoint $u$. Assume $d\geq \ell+3$ and $d-2\ell\geq1$. Then the map $p_u\colon\pi_n(\Emb_\partial(\D^\ell,X),u)\to\pi_{n+\ell}X$, $p_u(f)=(\vec{t}\mapsto -u\cup_\partial f_{\vec{t}})$, is an isomorphism for $0\leq n\leq d-2\ell-2$, and there is a short exact sequence of groups (sets if $d-2\ell-1=0$):
\[\begin{tikzcd}
        \faktor{\Z[\pi_1X\sm1]}{\relations_{\ell,d}\oplus \md_u(\pi_{d-\ell}X)} 
        \arrow[tail,shift left]{r}{\partial\realmap} &
        \pi_{d-2\ell-1}(\Emb_\partial(\D^\ell,X),u)
        \arrow[two heads]{r}{p_u}
        \arrow[dashed,shift left]{l}{\Da} &
        \pi_{d-\ell-1}X.
\end{tikzcd}
\]
\end{maintheorem}
The map $\md_u\colon\pi_{d-\ell}M\to\Z[\pi]$ whose image appears in the kernel of the extension is a restriction of the isomorphism $\Da$ from \eqref{eq-intro:Dax-invt}. The \emph{realization map} $\partial\realmap$ is an explicit inverse of $\Da$, constructed in Section~\ref{subsec:Dax}. The case $\ell=1,d=4$ was also studied by Gabai~\cite{Gabai-disks}, see Remark~\ref{rem:Gabai}. Some properties of these invariants in that case can be found in~\cite{KT-4dLBT}, and for $\ell=1$ and $d\geq3$ in~\cite{K-Dax}.

%\begin{remark}
%    If $\U\colon\D^2\hra M^4$ has a framed geometric dual in the interior, then a concordance from $\U$ to $K\in\Emb_s(\D^2,M)$ is a neatly embedded 3-ball in $M\times\I$ which has a framed dual for its boundary. In a work in progress we define a concordance obstruction $\wt{f}(K)$ using this observation and Theorem~\ref{thm:gen-dim}.
%\end{remark}
%

Next, we extend this to $\e$-augmented $\ell$-disks in Theorem~\ref{thm:Dax-augmented}, using some results on homotopy groups of frame bundles, proven in Appendix~\ref{app:frame-bundles}. In particular, when $d-\ell$ is odd it turns out that 
\begin{equation}\label{eq-intro:Z-augmentations}
    \pi_{d-2\ell-1}(\Emb_{\partial^\e}^\e(\D^\ell,X),u^\e)\cong \Z\times\pi_{d-2\ell-1}(\Emb_\partial(\D^\ell,X),u),
\end{equation}
so any $\S^{d-2\ell-1}$-family of embedded $\ell$-disks in $X$ has $\Z$ many $\e$-augmentations (see Remark~\ref{rem:amb-Exp}).
In Section~\ref{subsec:proof-of-ThmD} we combine the mentioned Theorem~\ref{thm:Dax-augmented} for $\ell=k-1$ with Theorem~\ref{thm-intro:LBT-deloop} as follows.
\begin{maintheorem}\label{thm-intro:combined}
Assume the setting with a dual \eqref{setting}, and let $\pi\coloneqq\pi_1M$.
    Then there are isomorphisms $\pi_n\Emb_s(\D^k,M)\cong\pi_{n+k}M$ for $n\leq d-2k-1$, and if $d-k\neq1,3,7$ a group extension:
\[\begin{tikzcd}[column sep=19pt]
    \faktor{\Z[\pi]}{\relations_{k-1,d}\oplus\md^\e(\pi_{d-k+1}M)} \arrow[tail,shift left]{rr}{\amb_\U\circ\partial\realmap^\e} 
        && 
    \pi_{d-2k}(\Emb_s(\D^k,M),\U) 
    \arrow[two heads]{rrr}{\eta_{W,\U}\oplus (-\U\cup\bull)}
    \arrow[dashed,shift left]{ll}{\Da\circ\foliate^\e_\U}
        &&& 
    \Z_{k,d}\oplus\faktor{\pi_{d-k}M}{\Z[\pi]G}
\end{tikzcd}
\]
    where $\Z_{k,d}\coloneqq\Z$ for $d-k$ even, $\Z_{k,d}\coloneqq\Z/2$ for $d-k$ odd, and $\md^\e(\pi_{d-k+1}M)=\langle1\rangle\oplus\md(\pi_{d-k+1}M)$ for $d-k$ even or $k=2$.
    % and for $d\neq 5,9$ odd:
    % \[\begin{tikzcd}[column sep=small]
    %     \faktor{\Z[\pi]}{\md^\e(\pi_{d-1}M)} \arrow[tail]{rr}{\amb_\U\circ\partial\realmap^\e}
    %     && \pi_{d-4}(\Emb_s(\D^2,M),\U) \arrow[two heads]{rrr}{\eta\oplus(-\U\cup\bull)}
    %     &&& \Z/2 \oplus\faktor{\pi_{d-2}M}{\Z[\pi]\cdot G}
    % \end{tikzcd}
    % \]
    The map on the right is on the class of $K\colon\S^{d-2k}\to \Emb_s(\D^k,M)$ given as follows.
\begin{itemize}
\item 
   The map $-\U\cup \bull$ assigns to $K$ the homotopy class modulo $\Z[\pi]\cdot G$ of the sphere $-\U\cup K\in\Map_*(\S^{d-2k}, \Map_*(\S^k,M))\simeq\Map_*(\S^{d-k},M)$, obtained by gluing the oppositely oriented $\U$ with $K$ along boundaries.
\item 
    If $d-k$ is even and $d-k\neq k$, then $\eta_{W,\U}(K)$ is one half of $e(\nu K,\nu\U)$, the relative Euler number of the normal bundle of the immersion $\I^{d-2k}\times\D^k\imra\I^{d-2k}\times M$, given by $(\vec{t},x)\mapsto (\vec{t},K_{\vec{t}}(x))$, relative to the constant family $\U$ (they agree on $\partial(\I^{d-2k}\times\D^k)$). Moreover, $\Z_{k,d}=\Z$ splits back.
\item 
    If $d-k=k=2$ or $4$ or $8$, then the relative Euler number $e(\nu K,\nu\U)$ might be odd, but we have 
    \[
    \eta_{W,\U}(K)=\frac{1}{2}\big(e(\nu K,\nu\U)-W(-\U\cup K)\big)
    \]
    where the homomorphism $W \in \Hom_\pi(\pi_kM,\Z)$ with $W(G)=0$ is an integral lift of the \emph{spherical} Stiefel--Whitney class $w^s_k\in \Hom_\pi(\pi_kM,\Z/2)$ given by $w^s_k(a)=w_k(a^*(TM))$ for $a\colon\S^k\to M$. 
\end{itemize}
\end{maintheorem}
% See Section~\ref{subsec:proof-of-ThmD} for a proof of this theorem.
We show that such a $W$ always exists in Lemma~\ref{lem:TX-even}.
We use Theorem~\ref{thm-intro:combined} in \cite{KT-4dLBT} to compute the set of path components $\pi_0\Emb_s(\D^2,M)$ for $d=4$, thus classifying in the setting with a dual the set of isotopy classes of disks in a $4$-manifold. Note that this is precisely a case when $2d-3k-3\leq n< 2d-3k-1$.

\begin{remark} \label{rem:cocycle}
    It is an interesting problem to determine the equivalence class of the extension in Theorem~\ref{thm-intro:combined}. For $d-k$ even we can divide out $\Z=\Z_{k,d}$ that splits off and pick a set theoretic section $\sigma$ of the quotient extension. If $\star$ is the group structure on $\pi_{d-2k}\Emb_s(\D^k,M)\cong \pi_{d-2k+1}\Emb_{u_0^\e}^\e(\D^{k-1},M_G)$, then for $a_i\in\pi_{d-k}M_G\cong\pi_{d-k}M/\Z[\pi]G$ the element $\sigma(a_1) \star \sigma(a_2)\star \sigma(a_1+a_2)^{-1}$ is in the kernel of $-\U\cup\bull$, and on this $\Da\circ\foliate^\e_\U$ inverts $\amb_\U\circ\partial\realmap^\e$. Thus, the group 2-cocycle is given by 
\[
    (a_1,a_2)\mapsto \Da(\sigma(a_1) \star \sigma(a_2)\star \sigma(a_1+a_2)^{-1}) \in \faktor{\Z[\pi_1M\sm1]}{\relations_{k-1,d}\oplus\md (\pi_{d-k+1}M)}.
\]
    We plan to study it in future work.
    Note that $\star$ is the usual group structure for $d-2k\geq1$, whereas for $d=2k$ it is an unexpected group structure on the set $\pi_0\Emb_s(\D^k,M)$. For $k=2,d=4$ we compute the commutators of this group in \cite{KT-4dLBT} and show that it is usually nonabelian.
\end{remark}

\begin{remark}\label{rem:arcs-with-duals}
    If in the setting with a dual we moreover have $d\geq k+3$ and $d-2k\geq1$, then we can apply Theorem~\ref{thm-intro:Dax} directly to obtain a description of $\pi_{d-2k-1}(\Emb_s(\D^k,M),\U)$ as an extension of $\pi_{d-k-1}M$ by a quotient of $\Z[\pi_1X\sm1]$. However, Theorem~\ref{thm-intro:combined} says that this quotient is trivial, and one can see why explicitly: we compute $\md_\U(g\cdot[G])=g$ in Example~\ref{ex:computing-md}.
    % When $\partial u$ has a geometric dual $G\colon\S^{d-1}\hra\partial X$ (so $u$ has endpoints in different boundary components of $X$, one of which is a sphere), the homomorphism $\delta_{\Imm}$ is surjective. 
    % Indeed, by the space level light bulb theorem (see Theorem~\ref{thm-intro:LBT-deloop} and Example~\ref{ex-intro:LBT-arcs}) we have $\Emb_\partial(\D^1,X)\overset{\simeq}{\ra}\Omega\S(X_G)$ for $X_G\coloneqq X\cup_G h^d$, so the inclusion $\incl$ into $\Imm_\partial(\D^1,X)\simeq\Omega\S(X)$ is an isomorphism on $\pi_{d-3}$ (whose inverse is induced by the inclusion $X\hra X_G$).
    % 
    % Let us note that under the map $\pi_{d-1}\S^{d-1}\hra\pi_{d-1}\S(X_G)\cong\pi_{d-2}(\Emb_\partial(\D^1,X),u)$ the generator is sent to the family $u_{tw}^G$ of embedded arcs obtained from the interior twist family $\tau\in\pi_{d-2}(\Imm_\partial(\D^1,X),u)$ (from the proof of Proposition~\ref{prop:r(1)}) by tubing into $G$. See Section~\ref{sec:2-in-d} and Figure~\ref{fig:ambExp(1)}.
\end{remark}

%%%%%%%%%%%%%%%%
\subsection*{Acknowledgments} 
Both authors cordially thank the Max Planck Institute for Mathematics in Bonn. The first author was supported by the Fondation Sciences Math\'ematiques de Paris and ETH Z\"urich. We thank the anonymous referee for pointing out additional references.

%%%%%%%%%%%%%%%%%%%%%%%%%%%%%%%%
\section{Preliminaries}\label{sec:prelim}

\subsection{Outline}\label{sec:outline}
We outline the contents of the paper and the ideas of proofs of Theorems~\ref{thm-intro:LBT-deloop} and ~\ref{thm-intro:combined}.

\subsubsection{Light bulb tricks}
The first step in the proof of Theorem~\ref{thm-intro:LBT-deloop} is illustrated in Figure~\ref{fig:new-fig}: attaching a handle to $G$ transforms a neat $k$-disk in $M$ into a ``half-disk'' $\HD^k\hra M_G$, namely, a $k$-disk whose boundary $\partial\HD^k=\S^{k-1}=\D_-^{k-1}\cup\D_+^{k-1}$ has one half $u_-=s|_{\D_-}$ embedded in $\partial M_G$ and the other $u_+=s|_{\D_+}$ neatly embedded in $M_G$, with $\partial u_-=\partial u_+$. Half-disks are not smooth embeddings in the classical sense, but in the sense of manifolds with corners, which we review in Section~\ref{subsec:manifolds-with-corners}. 
Moreover, this correspondence gives a homotopy equivalence
    \begin{equation}\label{eq:bijection}
        \Emb_{s^\e}(\D^k,M) \simeq \Emb_{s^\e}(\HD^k,M_G),
    \end{equation}
since removing a tubular neighborhood $\nu u_+$ turns half-disks back into neat disks; we first ensure that half-disks intersect $\nu u_+$ only along the collar $u_+^\e$. This is done in Section~\ref{subsec:neat-to-half}.

\begin{lemma}\label{lem:MG}
    The existence of a disk $\U\colon\D^k\hra M$ ensures that the diffeomorphism type of $M_G$ is independent of the choice of a framing $\psi$ of $G$, and in fact, of the choice of a dual $G$ to $s$ all together. Moreover, there is a diffeomorphism $M_G\cong M\sm\nu\U$.
\end{lemma}
\begin{proof}
    First note that $M$ is obtained from the complement $M \sm \nu \U$ of an open tubular neighborhood of $\U$ by attaching a $(d-k)$-handle with cocore $\U$. As $G$ is dual to $\partial \U$ and has trivial normal bundle, the standard handle cancellation of this $(d-k)$- and $(d-k+1)$-handle pair gives a diffeomorphism
    \[
    M_G\cong M\cup_{(G,\psi)}h^{d-k+1} \cong \big(M \sm \nu \U\big)\cup h^{d-k}\cup h^{d-k+1} \cong M \sm \nu \U.\qedhere
    \]
\end{proof}
The second ingredient for Theorem~\ref{thm-intro:LBT-deloop} is Cerf--Palais ``family version'' of ambient isotopy extension Theorem~\ref{thm:Cerf-all}, saying that a map restricting embeddings to a fixed submanifold is a locally trivial fibration. This is used in Section~\ref{subsec:half-disks} as follows. Consider the space $\Emb_{\D_-^\e}(\HD^k,M_G;\U)$ of half-disks that agree with our basepoint $\U$ only on the collar $\D_-^\e$ of $\D_-\subseteq\partial\HD^k$, while $\D_+$ moves freely and neatly in $M_G$. Then $\Emb_{s^\e}(\HD^k,M_G)=\Emb_{\partial\HD^k}(\HD^k,M_G;\U)$ is by definition the fiber over $u_+^\e$ of the restriction map $\ev_{\D^\e_+}$.
This gives rise to a fibration sequence
\[\begin{tikzcd}
        \Omega\Emb_{u_0}^\e(\D^{k-1},M_G) \arrow[dashed]{r}{\amb_\U} &
        \Emb_{s^\e}(\HD^k,M_G) \arrow{r} &
        \Emb_{\D_-^\e}(\HD^k,M_G;\U)\arrow{r}{\ev_{\D^\e_+}} &
        \Emb_{u_0}^\e(\D^{k-1},M_G).
      \end{tikzcd}
\]
in which the total space $\Emb_{\D_-^\e}(\HD^k,M_G;\U)$ is contractible: in this space half-disks are allowed to shrink arbitrarily close to their $\D_-^\e$-collar, where they are fixed. This implies that the \emph{connecting map $\amb_\U$ is a homotopy equivalence} with an explicit inverse $\foliate^\e_\U$ (given in general by Lemma~\ref{lem:Hurewicz}); see also Example~\ref{ex:maps-3d}.
\begin{remark}\label{rem:f-U}
    One can try to define the map $\foliate_\U\colon \Emb_s(\D^k,M)\to\Omega\Emb_{u_0}(\D^{k-1},M_G)$ from Theorem~\ref{thm-intro:LBT-fib-seq} directly as follows. Consider the foliation of $K\in\Emb_s(\D^k,M)$ by $K(\bull,t)\colon\D^{k-1}\hra M\subseteq M_G$, using a parametrization $\D^k\cong\D^{k-1}\times\I$; this is a path from $u_-$ to $u_+$, so to get a closed loop use the inverse of such a foliation of $\U$. However, $u_-$ is not neat so does not lie in the space $\Emb_{u_0}(\D^{k-1},M_G)$. One way around this would be to enlarge this space to also include embeddings that lie in $\partial M_G$; after all, they are limits of neat embeddings and we believe the homotopy type of the space does not change.
    
    We opted for a second way, making $\Emb_s(\D^k,M)$ smaller by considering its  homotopy equivalent (by Cerf's Proposition~\ref{prop:e-collar}) subspace $\Emb_{s^\e}(\D^k,M)$.
    Then the mentioned explicit inverse $\foliate^\e_\U$ induces the following $f_\U\coloneqq\ev_0\circ\foliate_\U^\e$: pick a neat $\wt{u}_-\colon\D^{k-1} \hra \U(\S^{k-1} \times [1-\e,1])$ that is close to $u_-$, and define $\foliate_\U(K)=\foliate(-\U)\cdot\foliate(K)$ as the path $\U(\bull,-t)$ from $u_+$ to $\wt{u}_-$, followed by $K(\bull,t)$ from $\wt{u}_-$ back to $u_+$.
\end{remark}
\begin{example}\label{ex:maps-3d}
    For $k=1,d=3$ think of $\gamma\in\Omega\Emb^\e_\emptyset(\D^0,M_G)$ as an isotopy of intervals $\D^0\times[0,\e]$. Extend this to an ambient isotopy $\Phi_t$ of $M_G$, so diffeomorphism $\Phi_1$ takes $u_+^\e$ to itself. Then $\amb_\U(\gamma)\coloneqq\Phi_1(\U)\in\Emb_{s^\e}(\HD^1,M_G)$ is the half-arc obtained by dragging the endpoint of $\U$ along $\gamma$, as in Figure~\ref{fig:amb-iso-3D}.
    The homotopy inverse $\foliate^\e_\U$ of $\amb_\U$ glues the given embedded arc $K$ to $\U$ along their boundaries and foliates the resulting loop $-\U\cup K$ by small intervals $\D^0\times[0,\e]$, to obtain $\foliate^\e_\U(K)\in\Omega\Emb^\e_\emptyset(\D^0,M_G)$.
% \begin{remark}[LBT for knots]\label{rem:lbt-knots}
%     If $K\colon\S^1\hra\S^1\times\S^2$ is freely homotopic to $S\colon\S^1\times\{y\}\subseteq\S^1\times\S^2$ and has $G\coloneqq \{e\}\times\S^2$ as a geometric dual, then $K$ isotopic to $S$. This can be lifted to spaces as well: cutting $\S^1\times\S^2$ open along $G$ turns $S$ into $\U$, and gives a homotopy equivalence $\Emb_\partial(\D^1,\D^1\times\S^2)
%     \simeq\{K:K\pitchfork G= \{K(e_1)=G(e_2)\}\}
%     \subseteq\Emb(\S^1,\S^1\times\S^2)$, see Section~\ref{subsec:spheres}.
% \end{remark}
\end{example}
\begin{figure}[!htbp]
        \centering
        \begin{tikzpicture}[scale=0.68,every node/.style={scale=0.9},even odd rule]
    \clip (-5.1,-0.1) rectangle (1.9,3.8);
    % bdry
    \draw 
        (-5,2.8) to[out=0,in=90,distance=0.4cm] (-3.7,2.4)
        -- (-3.7,0.4) to[out=90,in=0,distance=0.4cm] 
        (-5,0.8) -- (-5,2.8)
    ;
    \fill[black!20,draw=black,fill opacity=0.8]
        (-3.7,2.4) to[out=-90,in=0,distance=0.3cm] (-4.5,2) -- (-4.5,0) to[out=0,in=-90,distance=0.3cm] (-3.7,0.4) -- (-3.7,2.4)
    ;

    % the grey hole
    \fill[color=black!20!white] 
        (0.7,2.4) circle (0.2cm);
        \draw[black,postaction = decorate,decoration = {markings, mark = at position 0.76 with {\arrowreversed{>}} }] 
        (0.7,2.4) circle (0.25cm) node[left=0.17cm]{\small$g$}
    ;

    % gamma
    \draw[postaction = decorate,decoration = {markings, mark = at position 0.15 with {\arrow{stealth}} }]
        (-0.3,1.2) to[out=0,in=0,distance=1.8cm] (0.8,3.7)
        to[out=180,in=90,distance=1.8cm] (-1.9,1.2);
    \fill[white]
        (1.51,2) rectangle (2.2,2.7);
    \draw
        (-1.9,1.2) to[out=-90,in=-90,distance=1.7cm]  (1.8,2)
        to[out=90,in=100,distance=1.5cm]  (-1,2)
        to[out=-80,in=180,distance=0.4cm]  (-0.3,1.2) 
        ;

    % U
    \draw[orange!90!purple,thick]
        (-3.7,1.2) -- (-0.3,1.2);
    \draw[orange!90!purple]
        (-3.7,1.2) circle (1pt) node[below right]{\small$u_-$}
        (-0.25,1.2) circle (1pt) node[below]{\small$u_+$}
        (-2.6,1.2)  node[above]{$\U$};
    \draw  
        (-4.55,0)  node[above right]{\tiny $\partial M$}
        (0.5,1.1) node[]{$\gamma$};
\end{tikzpicture}~\begin{tikzpicture}[scale=0.68,every node/.style={scale=0.9},even odd rule]
    \clip (-5.1,-0.1) rectangle (1.9,3.8);
    % bdry
    \draw 
        (-5,2.8) to[out=0,in=90,distance=0.4cm] (-3.7,2.4)
        -- (-3.7,0.4) to[out=90,in=0,distance=0.4cm] 
        (-5,0.8) -- (-5,2.8)
    ;
    \fill[black!20,draw=black,fill opacity=0.8]
        (-3.7,2.4) to[out=-90,in=0,distance=0.3cm] (-4.5,2) -- (-4.5,0) to[out=0,in=-90,distance=0.3cm] (-3.7,0.4) -- (-3.7,2.4)
        ;
        
    % the grey hole
    \fill[color=black!20!white] 
        (0.7,2.4) circle (0.2cm);
        \draw[black,postaction = decorate,decoration = {markings, mark = at position 0.76 with {\arrowreversed{>}} }] 
        (0.7,2.4) circle (0.25cm) node[left=0.17cm]{\small$g$};

    % gamma
    \draw[dashed]
        (-1.9,1.2) to[out=-90,in=-90,distance=1.7cm]  (1.8,2)
        to[out=90,in=100,distance=1.5cm]  (-1,2)
        to[out=-80,in=180,distance=0.4cm]  (-0.3,1.2)
        ;

    % a bit isotoped U
    \draw[orange!90!purple,thick]
        (-3.7,1.2) -- (-2.1,1.2) 
        (-1.7,1.2) -- (-0.3,1.2);
    \draw[orange!90!purple,thick]
        (-2.1,1.2) to[out=-90,in=160,distance=0.4cm] (-1.5,0.2) to[out=-20,in=-20,distance=0.5cm] 
        (-1,0.4)  to[out=160,in=-90,distance=0.5cm] 
        (-1.7,1.2) 
        ;
    \draw[orange!90!purple,thick,postaction = decorate,decoration = {markings, mark = at position 0.7 with {\arrow{stealth}} }]
        (-0.3,1.2) to[out=0,in=0,distance=1.8cm] (0.8,3.7)
        to[out=180,in=90,distance=1.8cm] (-1.9,1.2)
        to[out=-90,in=160,distance=0.5cm] (-1.15,0.3) circle (0.5pt)
    ;
    
    % labels
    \draw[orange!90!purple]
        (-3.7,1.2) circle (1pt) node[below right]{\small$u_-$}
        (-0.25,1.2) circle (1pt) node[below]{\small$u_+$}
        (-2.6,1.2)  node[above]{$\U$};
    \draw  
        (-4.55,0)  node[above right]{\tiny$\partial M$}
        % (0.5,1.2) node[]{$\gamma$}
        ;
\end{tikzpicture}~\begin{tikzpicture}[scale=0.68,every node/.style={scale=0.9},even odd rule]
    \clip (-5.1,-0.1) rectangle (2.4,3.8);
    % bdry
    \draw 
        (-5,2.8) to[out=0,in=90,distance=0.4cm] (-3.7,2.4)
        -- (-3.7,0.4) to[out=90,in=0,distance=0.4cm] 
        (-5,0.8) -- (-5,2.8)
    ;
    
    \fill[black!20,draw=black,fill opacity=0.8]
        (-3.7,2.4) to[out=-90,in=0,distance=0.3cm] (-4.5,2) -- (-4.5,0) to[out=0,in=-90,distance=0.3cm] (-3.7,0.4) -- (-3.7,2.4)
    ;

    % the grey hole
    \fill[color=black!20!white] 
        (0.7,2.4) circle (0.2cm);
        \draw[black,postaction = decorate,decoration = {markings, mark = at position 0.76 with {\arrowreversed{>}} }] 
        (0.7,2.4) circle (0.25cm) node[left=0.17cm]{\small$g$}
    ;

    % isotoped U
    \draw[orange!90!purple,thick,postaction = decorate,decoration = {markings, mark = at position 0.85 with {\arrow{stealth}} }]
        (0.2,1.1) to[out=40,in=180,distance=0.1cm] (0.5,1.2)
        to[out=0,in=-90,distance=0.3cm] (1.5,2.1)
        to[out=110,in=130,distance=1.2cm] (2,2.6)
        to[out=60,in=0,distance=0.8cm] (1,3.6)
        to[out=180,in=90,distance=1.8cm] (-1.9,1.2)
        ;
    \fill[white]
        (1.25,2.5) rectangle (1.38,3.12);

    \draw[orange!90!purple,thick]
        (-3.7,1.2) -- (-2.1,1.2) 
        (-1.7,1.2) 
        to[out=0,in=-140,distance=1.4cm] (0.2,1.1);
    \draw[orange!90!purple,thick]% the finger
        (-2.1,1.2) to[out=-90,in=160,distance=0.4cm] (-1.4,0.2) to[out=-20,in=-90,distance=1.4cm]  (2,2) to[out=90,in=100,distance=1.7cm]  (-1.2,2)
        to[out=-80,in=180,distance=0.5cm]  (-0.4,1)
        to[out=0,in=0,distance=0.5cm] 
        (-0.4,1.4)  
        to[in=-80,out=180,distance=0.3cm] (-0.8,2.1)
        to[in=90,out=100,distance=1.2cm]  (1.65,2)
        to[in=-20,out=-90,distance=1.1cm] (-1.4,0.6)
        to[in=-90,out=160,distance=0.3cm] (-1.7,1.2) 
        ;
    \draw[orange!90!purple]
        (-3.7,1.2) circle (1pt) node[below right]{\small$u_-$}
        (-0.25,1.2) circle (1pt) %node[below]{$u_+$}
        (-2.6,1.2)  node[above]{$\amb_\U(\gamma)$};
    \draw  
        (-4.55,0)  node[above right]{\tiny$\partial M$}
        % (0.5,1.2) node[]{$\gamma$}
        ;
    \draw[orange!90!purple, thick]
        (-1.9,1.2) to[out=-90,in=-90,distance=1.7cm]  (1.8,2)
        to[out=90,in=100,distance=1.5cm]  (-1,2)
        to[out=-80,in=180,distance=0.5cm]  (-0.3,1.2)
        ;
\end{tikzpicture}
        \caption{Moments $t=0$, $0<t<1$, $t=1$ of an ambient isotopy $\Phi_t$ defining $\amb_\U(\gamma)$. A neighborhood of the intersection $\U\cap\gamma$ has to be ``dragged along'' all the way during the isotopy, whereas a neighborhood of the undercrossing in $\gamma$ is dragged along only for a while.}
        \label{fig:amb-iso-3D}
\end{figure}
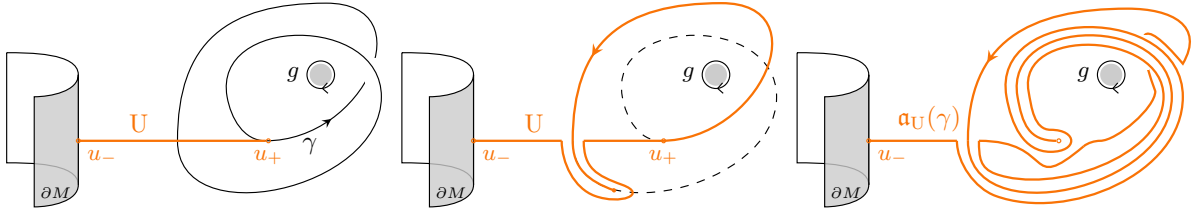

\subsubsection{The Dax invariant}
In Section~\ref{sec:disks} we discuss the work of Dax and prove Theorem~\ref{thm-intro:Dax}. The goal is to compute the homotopy groups $\pi_n(\Emb_{u_0}(\D^{k-1},M_G),u_+)$ for $n\leq d-2k$, but we will work in a general setting, with no restrictions on $M_G$ and $u_0$. Instead consider an arbitrary $d$-manifold $X$ (with adjectives as in the first paragraph of the paper), and write $\Emb_\partial(\D^\ell,X)$ and $\Imm_\partial(\D^\ell,X)$ for the spaces of neatly embedded and immersed disks with boundary $u_0\colon\S^{\ell-1}\hra \partial X$ and based at an arbitrary $u\colon\D^\ell\hra X$.

\begin{remark}\label{rem:Smale}
    There is a homotopy equivalence $\deriv(\U)^{-1}\cdot\deriv(\bull)\colon\Imm_{\partial^\e}(\D^1,X)\to\Omega V_1(X)$, where $V_1(X)$ is the unit sphere subbundle of the tangent bundle of $X$, due to Smale~\cite{Smale}. This sends $K\colon\D^1\imra X$ to the concatenation of the inverse of the path $\deriv(\U)$ followed by the path $\deriv(K)$, where $\deriv(\bull)$ is the unit derivative of an immersed arc, see also Section~\ref{subsec:immersions}. Interestingly, the proof idea is essentially the same Cerf's trick outlined above, and the key ingredient is to show that the restriction map for immersions is also a fibration. Note that $\Omega V_1(X)\simeq \Omega\S^{d-1}\times \Omega X$, see Lemma~\ref{lem:V_1}.
\end{remark}

Firstly, we will see that $\pi_n(\Emb_\partial(\D^\ell,X),u)\cong\pi_n(\Imm_\partial(\D^\ell,X),u)\cong\pi_{n+\ell}X$ for every $n\leq d-2\ell-2$ and that the inclusion induces a surjection
\[
    \pi_{d-2\ell-1}(\Emb_\partial(\D^\ell,X),u)\sra\pi_{d-2\ell-1}(\Imm_\partial(\D^\ell,X),u)\cong\pi_{d-\ell-1}X.
\]
We will compute its kernel in Section~\ref{subsec:immersions-conn-map}: it is the quotient of the relative homotopy group, computed in Section~\ref{subsec:Dax} as
\[
    \Da\colon\pi_{d-2\ell}(\Imm_\partial(\D^\ell,X),\Emb_\partial(\D^\ell,X),u)\overset{\cong}{\ra} \faktor{\Z[\pi_1X]}{\relations_{\ell,d}}
\]
by the image of $\pi_{d-2\ell}(\Imm_\partial(\D^\ell,X),u)\cong\Z/\relations_{\ell,d}\oplus\pi_{d-\ell}X$, as computed in Proposition~\ref{prop:r(1)}. 
Recall from \eqref{eq-intro:Dax-invt} that $\relations_{\ell,d}\coloneqq\langle g-(-1)^{d-\ell}g^{-1}: g\in\pi_1X\rangle$ except that $\relations_{1,d}\coloneqq0$, and that the Dax invariant $\Da([H])$ of the relative homotopy class of $H\colon (\I^{d-2\ell},\partial\I^{d-2\ell})\to (\Imm_\partial(\D^\ell,X),\Emb_\partial(\D^\ell,X))$ is the signed count of double point loops of the track $\wt{H}\colon\I^{d-2\ell}\times\D^\ell\to \I^{d-2\ell}\times X$. Note that when $\ell=1$ we can use the natural orientation of $\D^1$ to order the two sheets intersecting at a double point, and hence lend in $\Z[\pi_1X]$. However, for $\ell\geq2$ we have to mod out the ambiguous choice of sheets: the double point loop gets reversed and the sign changes exactly if the reflection on $\R^{d-\ell}\times\R^{d-\ell}$ changes the orientation.

In Theorem~\ref{thm:realization-map} (in Section~\ref{subsec:Dax}) we describe the inverse $\realmap$ of $\Da$, schematically depicted in Figure~\ref{fig:outline}. For $g\in\pi_1X$ the class $\realmap(g)$ is represented by the $(d-2\ell-1)$-family of embedded disks that swing around a meridian sphere $\S^{d-\ell-1}$ to $u$ at a point $x$, together with a path through immersed disks from this to the trivial family using the meridian ball at $x$. In fact, discussion of Section~\ref{subsec:Dax} applies to any simply connected manifold $V$ in place of $\D^l$. 

\begin{figure}[!htbp]
    \centering
     \begin{tikzpicture}[scale=0.6,every node/.style={scale=1},even odd rule,baseline]
        \clip (0.25,-1.45) rectangle (11.7,5.7);
        % the leftmost
        \draw[red!40!yellow, very thick]
                (7.6,1.8) to[out=-140,in=-130, distance=1.8cm] (8.1,2);
        %the second from left
        \draw[red!60!yellow, very thick]
                (7.6,1.8) to[out=-120,in=-110, distance=3cm] (8.1,2);
        \fill[white] (8,0) circle (1cm);
        
        \draw[thick] (1,0) circle (1.1pt) node[below]{\small$u(-1)$} -- (2,0)
            (4,0) -- (11,0) circle (1.1pt) node[below]{\small$u(+1)$}
            (6.3,-0.7) node{\color{black!80!white}$\mathbb{S}^{d-l-1}$};
        \draw[cyan,very thick]
            (8, 0) node[below=1pt]{$\mathbf{x}$} circle (1.5pt);

        % sphere
        \foreach \y in {8}{
            \shade[ball color=blue!5!white,opacity=0.55] (\y,0) circle (1cm);
            \draw[black!60!yellow] (\y-1,0) arc (180:360:1cm and 0.5cm);
            \draw[black!60!yellow,dotted] (\y-1,0) arc (180:0:1cm and 0.5cm);
            \draw[black!60!yellow,semithick] (\y,0) circle (1cm);
            }
        % the third from left
        \draw[red!70!yellow, very thick]
                        (7.6,1.8) to[out=-110,in=-100, distance=1.15cm] (7.6,-1.06)
                        (8,1.06) to[out=40,in=-120, distance=0.07cm] (8.1,2);
        % middle
        \draw[black, line width=1pt]
                (7.6,1.8) to[out=-90,in=-100, distance=1.15cm] (8,-1.06)
                (8.2,1.06) to[out=40,in=-120, distance=0.07cm] (8.1,2);
        % gp element
        \draw[black, thick, postaction=decorate,decoration = {markings, mark = at position 0.1 with {\arrow{latex}} , mark = at position 0.9 with {\arrowreversed{latex}}}]
                        (2,0) to[out=0,in=-100, distance=0.5cm]
                        (3,3.15) to[out=80,in=98, distance=4cm] (8.1,2)
                        (4,0) to[out=180,in=-100, distance=0.5cm]
                        (4,3) to[out=80,in=94, distance=3cm] (7.6,1.8);
        % the grey hole
        \fill[color=black!20!white] (5.6,3) circle (0.55cm);
        \draw[black,postaction = decorate,decoration = {markings, mark = at position 0.55 with {\arrowreversed{latex}} }] (5.6,3) circle (0.65cm) node[left=0.39cm]{$g$};
        % the third from right
        \draw[purple!80!blue, very thick]
                (7.6,1.8) to[out=-80,in=-80, distance=1.4cm] (8.8,-0.8)
                (8.5,0.95) to[out=100,in=-80, distance=0.1cm] (8.1,2)
        ;
        % the second from right
        \draw[purple!70!blue, very thick]
              (7.6,1.8) to[out=-55,in=-55, distance=3.5cm] (8.1,2);
        % the rightmost
        \draw[purple!60!blue, very thick]
                (7.6,1.8)to[out=-40,in=-40, distance=2.2cm] (8.1,2);
        \draw[myred, thick,-latex]
                (8.1,2) -- (7.6,1.8);
        \draw[myred, very thick]
                % (8.6,2.1) node{\textbf{$\alpha_N$}}
                % (9.19,-1.2) node[purple!70!blue]{\textbf{$\alpha_t$}}
                (8.2,4) node{\textbf{$\mathfrak{r}(g)$}};
    \end{tikzpicture}
    \caption{Samples $\realmap(g)_t\in\Emb_\partial(\D^\ell,X)$ for several $t\in \S^{1}$ and $\ell=1,d=4$.}
    \label{fig:outline}
\end{figure}
In Section~\ref{subsec:ribbons} we show an analogous result for $\e$-augmented arcs, see Theorem~\ref{thm:Dax-augmented}. The proof uses our study of the augmentation map from Section~\ref{subsec:forg-augm} and of frame bundles from Appendix~\ref{app:frame-bundles}. This almost immediately proves Theorem~\ref{thm-intro:combined}: it just remains to see that $\pi_{d-k}M_G$ is the quotient of $\pi_{d-k}M$ by the $\Z[\pi_1M]$-multiples of $G$, and to determine the splitting $\eta_{W,\U}$ as an Euler number, see Section~\ref{subsec:proof-of-ThmD}.

\subsection{Embeddings of manifolds with corners}\label{subsec:manifolds-with-corners}
Following Cerf \cite{Cerf-plongements} a $d$-dimensional \emph{manifold with corners} is a topological $d$-manifold $X$ with boundary, together with (a maximal atlas of) charts around all $x\in X$ with domain $\R^d_{(q)}\coloneqq \R^q\times[0,\infty)^{d-q}$ for some $0\leq q\leq d$ that send the origin $\vec{0}$ to $x$. One requires that each transition map, initially defined only on an open subset of $\R^d_{(q)}$, extends to a \emph{smooth} map on an open subset of $\R^d$. This gives a notion of smooth maps $Y\to X$ between manifolds with corners and a Whitney topology on the set $C^\infty(Y,X)$.

A component of $X_{(q)}\coloneqq \{x\in X\mid x \leftrightarrow \Vec{0}\in\R^d_{(q)} \text{ in some chart}\}$ is called a $q$-\emph{face} of $X$; it is a $q$-dimensional smooth manifold (without boundary), and the set $X$ is a disjoint union of its faces. The \emph{boundary} $\partial X$ is the union of codimension~$\geq 1$ faces, and a \emph{corner} of $X$ is a face of codimension~$\geq 2$ (e.g.\ one of the vertices in a square). If $q\in\{d-1,d\}$ we recover the usual structure on $X$ of a smooth manifold with boundary.
The role of a collar of $\partial X$ in $X$ is in general played by Cerf's ``prismatic neighbourhoods'' (that restrict to collars of faces). Manifolds with corners are clearly closed under cartesian product.
% a tangent bundle $TX$ as a $d$-dimensional vector bundle that comes equipped with $q$-dimensional subbundles along the $q$-faces.

For a $d$-manifold with corners $X$ a subset $X'\subseteq X$ is a $d'$-dimensional \emph{submanifold} for some $0\leq d'\leq d$ if each point of $X'$ admits a chart in $X$ for which $X'$ maps bijectively to some \emph{submodel} $\R^{d'}_{(q')}\subseteq\R^d_{(q)}$. A submodel is given by choosing $0\leq q'\leq q$, $0\leq k\leq\min\{q-q', d'-q'\}$, and inserting $d-d'$ many zeroes:
\[
    \R^{d'}_{(q')}\owns (x_1,\dots,x_{q'}, y_1, \dots, y_{d'-q'}) \mapsto (x_1,\dots,x_{q'}, y_1,\dots, y_k, 0, \dots, 0, y_{k+1}, \dots, y_{d'-q'})\in \R^d_{(q)}.
\]
% $(\R^{q'}\times[0,\infty)^{k}\times\{0\}^{q-(q'+k)})\times([0,\infty)^{d'-q'-k}\times\{0\}^{d-q-(d'-q'-k)})\subseteq\R^d_{(q)}$, for $0\leq k\leq d'-q'$.
% Submanifolds have suitable tubular neighborhoods by \cite[264]{Cerf-plongements}.
These relative charts induce the structure of a $d'$-manifold with corners on $X'$. 
For example, if $X$ is a smooth manifold with boundary, then for the corners of $X'$ of codimension $d'-q'\leq 2$, we have the following different cases, depending on the value of $q'$ listed on the left:
\begin{enumerate}[leftmargin=3em]
    \item[$d'$:]
        For a top dimensional face $F\subseteq X'$ we have either $F\subseteq 
        \partial X$ or $F\subseteq X_{(d)}=X \sm \partial X$.
    \item[$d'{-}1$:]
        For a small neighborhood $V_p\subseteq X'$ of $p\in X'_{(d'-1)}$ there are 3 possibilities: either $V_p\subseteq\partial X$ respectively $V_p\subseteq X \sm \partial X$ as above, or $(V_p, \partial V_p)\subseteq (X,\partial X)$ is a neat submanifold.
    \item[$d'{-}2$:]
        A small neighborhood $V_p\subseteq X'$ of $p\in X'_{(d'-2)}$ looks like a neighborhood of $\Vec{0}\in\R^{d'-2}\times~[0,\infty)^2$, and there are 4 possibilities for $V_p\subseteq X$. The case $(x_1,\dots,x_{q'},y_1,y_2)\mapsto (x_1,\dots,x_{q'},y_1,0,\dots,0,y_2)$ is the most interesting: exactly ``one half'' of $\partial V_p$, the one corresponding to $\vec{0}\times[0,\infty)\times\{0\}$, lies in $\partial X$ while the rest of $V_p$ lies in $X\sm \partial X$.
\end{enumerate}
Thus, a smooth submanifold with boundary is either neat or contained in the interior of $X$, while the simplest next case is the local model listed last above, which we use for \emph{half-disks} $X'=\HD^k$, see Figure~\ref{fig:setup-half-disk}.
% If $Y$ and $X$ have $n$ respectively $m$ faces then there are a priori $2^{m\cdot n}$ possible incidence relations for an embedding as above. Even though they are not independent, it's quite a strong restriction to fix one such relation (via the basepoint $y$). 

A smooth map $f$ of manifolds with corners is an \emph{embedding}, written $f\colon Y \hra X$, if $f(Y)$ is a submanifold of $X$ whose induced corner structure makes $f\colon Y\to f(Y)$ a diffeomorphism. 
An \emph{immersion} of manifolds with corners is a smooth map that is locally an embedding. Spaces of embeddings and immersions inherit the Whitney $C^\infty$-topology. 
\begin{defn}
For $y\colon Y\hra X$ let $\Emb(Y,X;y)\subseteq\Emb(Y,X)$ consist of embeddings $f\colon Y\hra X$ such that for each $p\in Y$ and each face $F_X\subseteq X$ we have $f(p)\in F_X$ if and only if $y(p)\in F_X$; we say that $y$ and $f$ have the same ``incidence relations'' \cite[281]{Cerf-plongements}. 

Furthermore, for a closed subset $Y'\subseteq Y$, let
\[
    \Emb_{Y'}(Y,X;y)\subseteq\Emb(Y,X;y)
\]
consist of those embeddings $f$ that agree with $y$ on $Y'$, that is $f|_{Y'}=y|_{Y'}$. We say that ``$y$ is the boundary condition along $Y'$'' (note that $y$ at the same time determines the incidence relations).
\end{defn}

%%%%%%%%%%%%%%%%
\subsubsection{Restriction maps for embeddings}\label{subsec:Cerf-Palais}
Consider compact manifolds with corners and embeddings
    \[\begin{tikzcd}
      Z'\arrow[tail]{r} & Z\arrow[tail]{r}\arrow[tail, bend left]{rr}[near start]{z} & Y\arrow[tail]{r}{y} & X.
    \end{tikzcd}
    \]
We say a subset $Y'\subseteq Y$ is a \emph{local normal tube to $Z\subseteq Y$ along $Z'$} if $Y'\cap Z=Z'$ and there is a tubular neighborhood $V\subseteq Y$ of $Z$ in $Y$ such that $Y'\cap V=pr^{-1}(Z')$, where $pr\colon V\to Z$ is the projection.

\begin{theorem}\label{thm:Cerf-all}
  With the above notation, the following restriction maps are both locally trivial:
 \begin{enumerate}[label={\Roman*.},ref={\thetheorem.\Roman*}]
    \item
    $\;\ev_Z\colon\Emb(Y,X;y)\ra\Emb(Z,X;z)\quad$  \cite[p.294 Cor.2, with notation $E\subseteq H\subseteq F$]{Cerf-plongements};
    \label{thm:Cerf-I}
    \item
    $\;\ev_Z\colon\Emb_{Y'}(Y,X;y)\ra\Emb_{Z'}(Z,X;z)\quad$ \cite[p.298 Cor.2]{Cerf-plongements}.
    \label{thm:Cerf-II}
\end{enumerate}
\end{theorem}
Here a map $p\colon E\to B$ is \emph{locally trivial} if for each $b\in B$ there exists a neighborhood $x\in V\subseteq B$ and a homeomorphism $p^{-1}(V)\cong V\times p^{-1}(x)$. Palais~\cite{Palais} showed \ref{thm:Cerf-I} in the case when all manifolds have empty boundary and $Y$ is compact; Cerf extended this to manifolds with corners and quite general boundary conditions as in \ref{thm:Cerf-II}. We also record the following fact (see \cite[p.331 Cor.3,p.337~Prop.9]{Cerf-plongements}), which will be used to replace neat embeddings by those fixed on a collar.
\begin{prop}\label{prop:e-collar}
    % The inclusion $\Emb_{s^\e}(D,X)\hra\Emb_s(D,X)$ is a weak homotopy equivalence for compact smooth manifolds $D$ and $X$ with boundary.
    In the same notation as above, if $Z'=Y'\cap Z$ is the closure of a codimension $1$ face, then the inclusion $\Emb_{Z\cup Y'}(Y,X;y)\hra\Emb_{Z}(Y,X;y)$ is a weak homotopy equivalence.
\end{prop}
% For proofs both used the following criterion.
% Let $G$ be a topological group and $B$ a space with a continuous $G$-action. We say that a $G$-space $B$ \emph{admits a local $G$-section at $b\in B$} if the map $G\to B$, given by $g\mapsto g.b$, has a section near $b$, that is, there is a neighborhood $b\in V\subseteq B$ and a map $s\colon V\to G$ so that $s(x).b=x$ for all $x\in V$.
% \begin{lemma}[Cerf--Palais Fibration Criterion {\cite[Thm.A]{Palais}, \cite[p.240 Lem.2]{Cerf-plongements}}]\label{Cerf-Palais-criterion}
%      If $f\colon E\to B$ is a $G$-equivariant map and $B$ admits local $G$-sections at all points, then $f$ is a locally trivial map.
% \end{lemma}
% To prove Theorem~\ref{thm:Cerf-all} one uses the natural action on spaces of embeddings (by postcomposition) of the group of diffeomorphisms $G=\Diff(X)$, or its subgroup $G=\Diff_{z(Z')}(X)$ of those that are the identity on $z(Z')$. The evaluation maps are clearly $G$-equivariant and the crucial step is to show the following.
% \begin{theorem}[Parametrized Isotopy Extension {\cite[p.293,Thm.5]{Cerf-plongements}}]
% \label{thm:Cerf-sections}
%   The space of embeddings \\
%   $\Emb_{Z'}(Z,X;z)$ admits local $\Diff_{z(Z')}(X)$-sections at all points.
% \end{theorem}

It is a standard fact that a locally trivial map over a paracompact base is a Hurewicz fibration. One can show that $\Emb_{Y'}(Y,X;y)$ are metrizable, so they are paracompact by Stone's theorem, and Theorem~\ref{thm:Cerf-all} implies that \emph{the restriction maps $\ev_Z$ are Hurewicz fibrations}. To show metrizability one can argue as follows: embed $X$ into some $\R^n$ so that $\Emb_{Y'}(Y,X;y)$ is a subspace of $C^\infty(Y,\R^n)$; the latter is metrizable by classical functional analysis, using that sup-norms on compact sets define a family of seminorms such that $d(f,g)= \sup_n \frac{\sup_x |(f-g)^{(n)}(x)| }{ 2^n (1+\sup_x |(f-g)^{(n)}(x)|) }$ is a metric. In Appendix~\ref{app:Hurewicz} we discuss some properties of fibrations and their connecting maps, needed in the proof of Theorem~\ref{thm:anyD}. See also Remark~\ref{rem:amb-directly}.

\subsection{History}\label{subsec:history}
Let us list some results and examples that precede Theorems~\ref{thm-intro:LBT-fib-seq} and~\ref{thm-intro:LBT-deloop}, as well as the method of proof, and which can be recovered from them as special cases.

In \cite{Cerf-plongements} Cerf laid foundations of differential topology, discussing in details manifolds with corners, submanifolds, and spaces of embeddings. In \cite{Cerf-applications} he surveyed these results and presented the half-disk trick, proving that $\Diff_\partial(\D^d)\simeq\Omega\Emb_\partial(\D^{d-1},\D^d)$ and that homotopy groups of $\Emb_\partial(\D^{d-1},\D^d)$ inject into (the shifted) homotopy groups of  $\Emb_\partial^\e(\D^{d-2},\D^d)$. In particular, $\Diff_\partial(\D^2)$ is contractible, reproving a theorem of Smale. This was later republished as the appendix of \cite{Cerf-diffeos}.

Gramain~\cite{Gramain} used Cerf's trick to extend this to a proof that the components of $\Diff_\partial(S)$ are contractible for most compact surfaces $S$, reproving theorems of Earle, Eells, and Schatz, by completely topological methods. Hatcher gave an exposition of this proof in \cite[App.B]{Hatcher-exposition}. He also used it in \cite[App.]{Hatcher-SmaleConj} when describing equivalent forms of the Smale conjecture.

Goodwillie in his thesis~\cite[7-8]{Goodwillie-thesis} used half-disks to outline why the based loop space of the space of concordance embeddings of $\D^{k-1}$ into $X$ is homotopy equivalent to the space of concordance embeddings of $\D^k$ into $X\cup h^{d-k} (=M)$.

Budney and Gabai~\cite{Budney-Gabai} used half-disks and global coordinates to show statements equivalent to the case $M=\S^{d-k}\times\D^k$ of our Theorems~\ref{thm-intro:LBT-fib-seq} and~\ref{thm-intro:LBT-deloop}, and remarked that this can be used to recover Gabai's 4-dimensional light bulb theorem, see \cite[v1:Rem.6.2]{Budney-Gabai}.

Knudsen and Kupers~\cite[Sec.6.2.4]{Knudsen-Kupers} consider an example related both to the classical light bulb trick and Cerf's half-disk trick.

Finally, let us also mention that Litherland~\cite{Litherland} proved a light bulb theorem for $\S^2$'s in $\S^2\times\S^{d-2}$.

\begin{example}\label{ex-intro:classical-LBT}
    For $k=1$ and $d=3$ Theorem~\ref{thm-intro:LBT-fib-seq} indeed recovers the light bulb trick $\pi_0\Emb_s(\D^1,M)\cong\pi_1M$, since $\Emb_{\emptyset}(\D^0,M_G)=M_G$ and $\pi_1\Omega^0\S^2=0$.
    For example, we have $\Emb_s(\D^1,\D^1\times\S^2)\simeq\Omega\S^2$ for the inclusion $s\colon\partial\D^1\times\{pt\}\hra\partial\D^1\times\S^2$ (since $(\D^1\times\S^2)_G\cong\D^3$), so all arcs are isotopic; however, note that $\pi_1\Omega\S^1=\Z$, so there is a nontrivial loop of arcs (given by the swing around the dual, see Section~\ref{subsec:d=3}). In fact, we more generally have $\Emb_s(\D^1,M)\simeq\Omega\S^2\times\Omega M_G$, see Example~\ref{ex-intro:LBT-arcs}.
\end{example}

\begin{example}\label{ex-intro:LBT-arcs}
    For $k=1$ and $d\geq 2$ the fibration splits by Lemma~\ref{lem:V_1}, so we get a homotopy equivalence
    \[
\Emb_s(\D^1,M)\overset{\simeq}{\ra}\Omega \S^{d-1} \times \Omega M_G=\Omega \S^{d-1} \times \Omega(M \cup_{\S^{d-1}} \D^d).
    \]
    Here a dual for the boundary condition means that the two points $s(\S^0)$ lie in distinct components of $\partial M$, one of which is diffeomorphic to $\S^{d-1}$, and along which a $d$-handle is attached, giving $M_G\coloneqq M \cup_{\S^{d-1}} \D^d$. 
    
    For $\dim(M)=d=2$ it follows that the group $\Z \times \pi_1(M \cup_{\S^1} \D^2)$ acts simply and transitively on the set $\pi_0\Emb_s(\D^1,M)$ of connected components, that are all contractible. 
    The action of $\Z$ is by Dehn twists around one boundary component, while $\pi_1(M \cup_{\S^1} \D^2)$ acts by the ``point-push'' map, well known in the surface community, see~\cite{Farb}. For $d\geq3$ we get such an action of $\pi_1(M)$ on $\pi_0\Emb_s(\D^1,M)$. For all $d\geq 4$ the connected components have nontrivial higher homotopy groups.
\end{example}

\begin{example}
    If $1\leq k=d-1$ then $\S^{d-k}=\S^1$, so for $n>0$ (and $n=0$ if $d\geq4$) Theorem~\ref{thm-intro:LBT-fib-seq} gives:
    \[
        \pi_{n}\Emb_s(\D^{d-1},M)\cong\pi_{n+1}\Emb_{u_0}(\D^{d-2},M\cup_{\S^1}h^2).
    \]
    In particular, this holds for $M=\S^1\times\D^{d-1}$, $s=p\times\partial\D^{d-1}$, $G=S^1\times q$, and $M_G=M\cup_{\S^1}h^2=\D^d$.
    % \[
    %     \pi_{n}\Emb_s(\D^{d-1},\S^1\times\D^{d-1})\cong\pi_{n+1}\Emb_{u_0}(\D^{d-2},\D^d),
    % \]
    % for all $n\geq 0$.
    This case was discussed by Budney and Gabai \cite{Budney-Gabai}, who constructed for $d=4$ and $n=0$ an infinitely generated subgroup of this group.
\end{example}

\begin{example}\label{ex:Cerf}
% true also for k=1, but needs a bit of thinking: since \Omega^0\S^0 has two components... But we should just not deloop once more in this case (just think Emb^e). Note: M' is a 1-manifold with boundary, so union of D^1 containing G(1), and some other components. Then \Omega M_G=\Omega M' but based at G(1) so this is \Omega D^1 so contractible...
    For $1\leq k=d$ any embedding $(\D^d,\partial\D^d)\hra (M,\partial M)$ is a diffeomorphism, as we assumed for simplicity that $M$ is connected. However, Theorem~\ref{thm-intro:LBT-fib-seq} also applies to the disjoint union $M=\D^d\sqcup M'$ with any $d$-manifold $M'$, for $s=\partial\D^d$ and $G\colon\S^0 \hra \partial M$ satisfying $G(-1)\in\partial \D^d$, $G(1)\in\partial M'$. Since $\S^{d-k}=\S^0$ and $M_G\coloneqq(\D^d\sqcup M')\cup_{\nu G} h^1 \cong M'$ we obtain
    \[
        \Diff_\partial(\D^d) \cong \Emb_s(\D^d,\D^d\sqcup M')\simeq \Omega\Emb_{u_0}(\D^{d-1},M').
    \]
    For $M'=\D^d$ this is the result of Cerf \cite[App.]{Cerf-diffeos} that motivated our entire approach.
\end{example}

%%%%%%%%%%%%%%%%%%%%%
\section{Spaces of disks and half-disks}\label{sec:any-dim}
In this section we work in arbitrary dimensions and prove Theorem~\ref{thm-intro:LBT-deloop}.
Disks in $M$ with duals in $\partial M$ are reduced to half-disks in $X=M_G$ in Section~\ref{subsec:neat-to-half}, while half-disks in $X$ are described as loops in $X$ of ``$\e$-augmented'' disks of lower dimension in Section~\ref{subsec:half-disks}. We first fix notation for the setting of half-disks.

%%%%%%%%%%%%%%%%%%%%%
  Given an embedding $\U\colon D\hra X$ of compact manifolds with corners, and a closed subset $b\subseteq D$, recall from Section~\ref{subsec:manifolds-with-corners} the space $\Emb_b(D,X;\U)$ consisting of those embeddings $K\colon D\hra X$ that have the same incidence relations (for faces in $D, X$) as $\U$, and that agree with $\U$ on $b$. We let $\U$ be its basepoint. 
  
  In particular, if $D$ and $X$ are manifolds with boundary, the incidence relation $\U(D)\cap\partial X=\U(\partial D)$ together with the boundary condition $s\coloneqq\U|_b$ on $b=\partial D$ reproduces our space
  \[
  \Emb_s(D,X)=\Emb_{\partial D}(D,X;\U)
  \]
  of neat embeddings (see the first paragraph of the paper). For such $\U$ we can expand the boundary condition to a closed collar $b\coloneqq (\partial D)\times[0,\e]\subseteq D$ and define
  \[
  \Emb_{s^\e}(D,X)\coloneqq \Emb_{(\partial D) \times[0,\e]}(D,X;\U).
  \]
  By Proposition~\ref{prop:e-collar} the inclusion $\Emb_{s^\e}(D,X)\hra\Emb_s(D,X)$ is a weak homotopy equivalence.

\begin{wrapfigure}[13]{r}{0.36\textwidth}
    \centering
    \vspace{5pt}
    \includegraphics[width=0.52\linewidth]{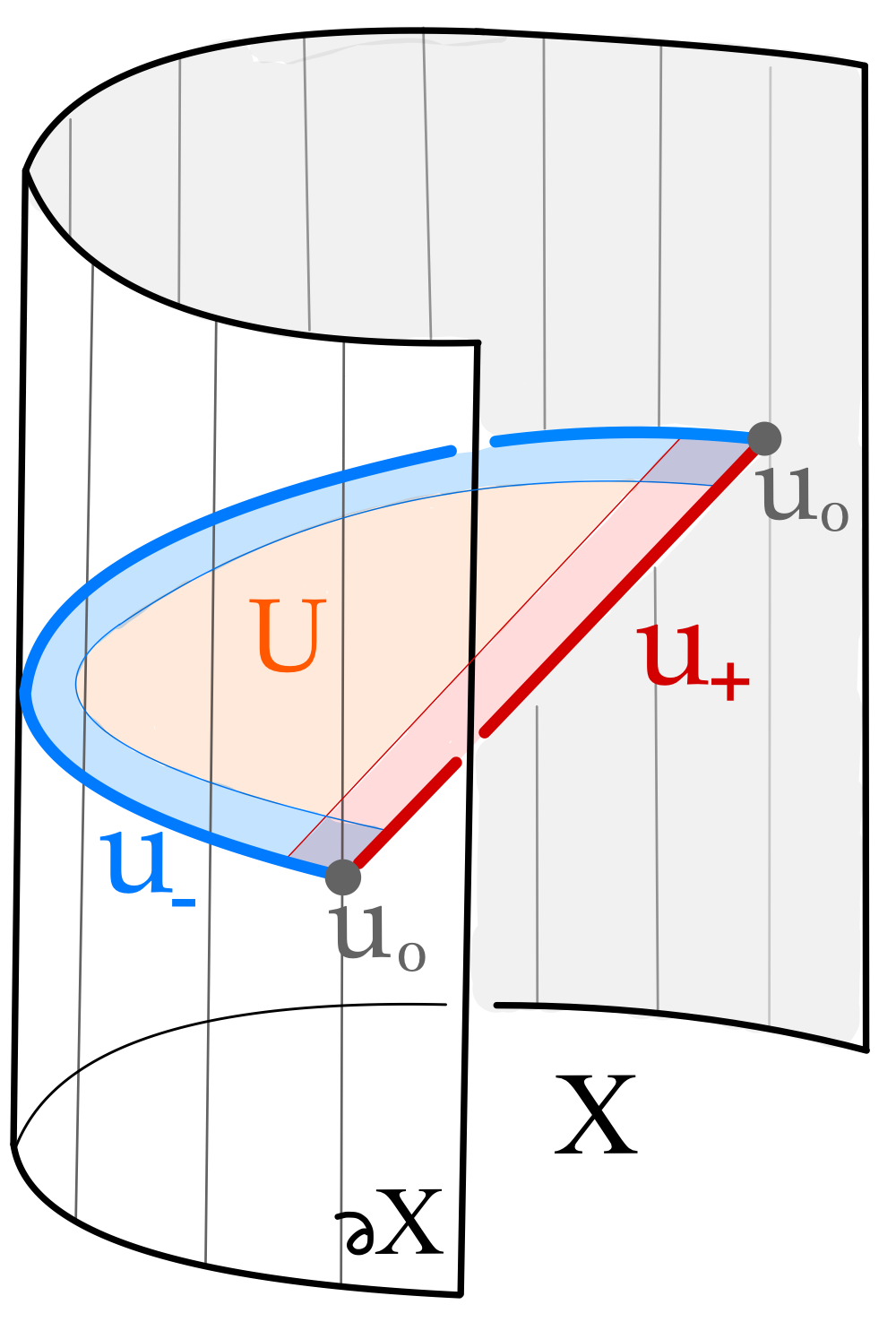}
    \caption{A half-disk for $k=2$, $d=3$.}
    \label{fig:setup-half-disk}
\end{wrapfigure}
We will need the the next simplest case of manifolds with corners, that also have codimension~2 faces. Namely, we now take the domain $D$ to be $\HD^k\coloneqq\{x\in\R^k:||x||\leq 1, x_1\leq 0\}$, that is the west half of the unit $k$-dimensional disk, and consider subsets $\D_-\coloneqq\{x\in \HD^k:||x||=1\}$ and $\D_+\coloneqq\{x\in \HD^k:x_1=0\}$, that are $(k-1)$-dimensional disks with $\partial\HD^k=\D_-\cup\D_+$ and $\S_0\coloneqq\D_-\cap\D_+$, which is a $(k-2)$-dimensional sphere. Then $\HD^k$ is a $k$-manifold with corners with one $k$-face $\mathring{\HD}^k$, two $(k-1)$-faces $\mathring{\D}_\pm$ and one $(k-2)$-face $\S_0$, the unique corner of $\HD^k$.

Moreover, consider subsets $\D_-^\e\coloneqq\{x\in\HD^k : ||x||\geq 1-\e\}$ and $\D_+^\e\coloneqq\{x\in \HD^k:x_1\geq -\e \}$ (shaded strips in Figure~\ref{fig:setup-half-disk}), both diffeomorphic to $\D^{k-1}\times[0,1]$ and with $\D_-^\e\cap\D_+^\e\cong\S_0\times[0,\e]^2$. Denote $\partial^\e\HD^k\coloneqq\D_-^\e\cup\D_+^\e$ (an example of Cerf's prismatic collar).

Next, we fix a smooth manifold with boundary $X$ and an embedding $\U\colon\HD^k\hra X$ of manifolds with corners such that $\U$ maps $\mathring{\HD}^k$ to the interior of $X$ and the other incidence relations are determined by the restrictions of $\U$ to $\D_{\pm}$ as follows (see Figure~\ref{fig:setup-half-disk}): the image of $u_-\coloneqq\U|_{\D_-}$ is contained in $\partial X$, while $u_+\coloneqq\U|_{\D_+}\colon\D_+ \hra X$ is a neat embedding, with $u_0\coloneqq \partial(u_-)=\partial(u_+)\colon\S_0\hra\partial X$.

Similarly, we write $u_\pm^\e\coloneqq \U|_{\D_\pm^\e}$ and $u_0^\e\coloneqq \U|_{\D_-^\e\cap \D_+^\e}$, and identify their domains with the corresponding products. Let us also still write $s\coloneqq \U|_{\partial\D}$ and $s^\e\coloneqq \U|_{\partial\D^\e}$

The elements of Cerf's space $\Emb(\HD^k,X;\U)$ are called \emph{half-disks} in $X$. We are interested in its subspace $\Emb_{s^\e}(\HD^k,X)\coloneqq\Emb_{\partial^\e\HD^k}(\HD^k,X;\U)$ that by definition consists of those half-disks $K\colon\HD^k\hra X$ that agree with $\U$ on the prismatic collar $\partial^\e\HD^k$. Equivalently, $K$ is a topological embedding that agrees with $\U$ on $\partial^\e\HD^k$ and restricts to an (ordinary) smooth embedding on interiors $\mathring{\HD}^k\hra X\sm \partial X$.

We saw in Proposition~\ref{prop:e-collar} that the space of neat disks $\Emb_{u_0}(\D^{k-1},X)=\Emb_{\S_0}(\D_+,X;u_+)$ has a weakly equivalent subspace $\Emb_{u_0^\e}(\D^{k-1},X)=\Emb_{\S_0\times[0,\e]}(\D_+,X;u_+)$ of disks fixed on a collar $\S_0\times[0,\e]\subseteq\D_+$. We will also need a space where each such disk is augmented with ``push-offs'', namely
\[
  \Emb_{u_0^\e}^\e(\D^{k-1},X)\coloneqq\Emb_{\D_-^\e\cap\D_+^\e}(\D_+^\e,X;u_+^\e).
\]
A point here is a topological embedding $K\colon\D^{k-1} \times [0,\e] \hra X$ that restricts to $u_0^\e$ on $(\S_0 \times [0,\e]) \times [0,\e]$ and to an (ordinary) smooth embedding $\mathring{\D}^{k-1} \times [0,\e]\hra X\sm \partial X$; we call $K$ an \emph{``$\e$-augmented'' $(k-1)$-disk}. 

\subsection{From half-disks to loops of \texorpdfstring{$\e$}{e}-augmented disks}\label{subsec:half-disks}
Our embedding spaces are always equipped with basepoints according to the convention above; in the next theorem these are respectively $u_+^\e$ and $\U$. A proof of this result was essentially given in \cite[App.]{Cerf-diffeos}, but we also identify the maps involved, needed for future geometric applications, as in \cite{KT-4dLBT}.
\begin{theorem}\label{thm:anyD}
    For all $k\geq1$ and $d\geq1$ there are inverse homotopy equivalences
    \begin{equation}
      \begin{tikzcd}[column sep=large]
        \amb_\U\colon\Omega \Emb_{u_0^\e}^\e(\D^{k-1},X) 
        \arrow[shift left=3pt]{r}{}[swap]{\sim}
        &
        \Emb_{s^\e}(\HD^k,X)\colon \foliate^\e_\U
        \arrow[shift left=3pt]{l}{}
      \end{tikzcd}
    \end{equation}
    where $\amb_\U$ is given on homotopy groups by the family ambient isotopy theorem, while $\foliate^\e_\U$ maps a half-disk $K$ to the loop of $\e$-augmented $(k-1)$-disks induced by appropriate foliation of the sphere $-\U\cup K$.
    % where $\amb$ maps a loop $\beta$ of $\e$-augmented disks based at $u_+^\e$, to the half-disk obtained from $\U$ by an isotopy extending $\beta$, while $\foliate^\e$ maps a half-disk $K$ to the foliation of the sphere $K\cup_\partial-\U$ by $\e$-augmented disks.
\end{theorem}
\begin{proof}
Consider the fibration sequence
\[\begin{tikzcd}
    \Emb_{s^\e}(\HD^k,X)\coloneqq\ev_{\D^\e_+}^{-1}(u_+^\e)\rar & %   \arrow[tail]{d}
    \Emb_{\D_-^\e}(\HD^k,X;\U)\rar{\ev_{\D^\e_+}} &    %   \arrow[equals]{d}
    \Emb_{u_0^\e}^\e(\D^{k-1},X)  %   \arrow[]{d}{\ev_0}
    %   \\ \Emb_{u_-^\e}^{u_+}(\HD^k,X)\arrow{r} & \Emb_{u_-^\e}(\HD^k,X)\arrow{r}{\ev_+} & \Emb_{u_0}(\D_+^{k-1},X)
\end{tikzcd}
\]
where $\ev_{\D^\e_+}$ restricts $K\colon\HD^k\hra X$ to the $\e$-collar $\D_+^\e\subseteq\HD^k$ of the unconstrained half of its boundary. This is a fibration by Cerf's Theorem~\ref{thm:Cerf-II} for $Y'\coloneqq\D_-^\e\subseteq\HD^k\eqqcolon Y$ and $Z'\coloneqq\D_-^\e\cap\D_+^\e\subseteq\D_+^\e\eqqcolon Z$. 
%   $\ev_Z\colon\Emb_{Y'}(Y,X;y)\to\Emb_{Z'}(Z,X;z)$
We will show that its total space $E^\e\coloneqq\Emb_{\D_-^\e}(\HD^k,X;U)$ admits an explicit contraction
\begin{equation}\label{eq:contraction}
    R\colon E^\e \times [0,1] \ra E^\e,\;
    \text{ with } R_0 =\const_\U,\;
    R_1=\Id.
\end{equation}
Then Lemma~\ref{lem:Hurewicz} implies that any connecting map
  \[
    \amb_\U\colon \Omega\Emb_{u_0^\e}^\e(\D^{k-1},X)\ra\Emb_{s^\e}(\HD^k,X)
  \]
is a homotopy equivalence, and is by definition is given by lifting the loop in the base space to a path in the total space and taking the endpoint. In our setting this amounts to extending a loop $\beta$ of $\e$-augmented $(k-1)$-disks based at $u_+^\e$ to an isotopy of half-disks starting with $\U$ and ending with the desired half-disk~$\amb_\U(\beta)$. By the same lemma, the restriction $\foliate^\e_\U(K)(t)= \ev_{\D^\e_+}\circ R_t(K)=R_t(K)|_{\D_+^\e}$ is a homotopy inverse to $\amb_\U$, and we will see below that it is indeed given by a foliation, cf.\ Section~\ref{subsec-intro:space-LBT}.

To construct the retraction $R$ we start with a path of re-embeddings $\varphi_t\colon \HD^k \hra \HD^k$, $t\in [\e,1]$, such that
\begin{enumerate}
    \item $\varphi_1 = \Id_\HD$ and $\varphi_t|_{\D_-^{\e/2}} = \Id_{\D_-^{\e/2}}$ for all $t$,
    \item $\varphi_\e(\HD^k)\subseteq\D_-^\e$,
    \item $\varphi_t(\D_-^{\e})\subseteq\D_-^{\e}$ for all $t$.
\end{enumerate}
It is not hard write down such an isotopy $\varphi_t$ using radial coordinates, see Figure~\ref{fig:push-to-minus} for $k=2$. Now consider the homotopy $E^{\e} \times [\e,1]\to E^{\e/2}\coloneqq\Emb_{\D_-^{\e/2}}(\HD^k,X;U)$, defined by $K\mapsto K\circ \varphi_t$.
By Property~(1) this indeed defines paths in the space $E^{\e/2}$ ($E^\e\subseteq E^{\e/2}$ is smaller as it has the stronger boundary condition), ending with $K\varphi_1=K$, and starting with $K\varphi_\e=\U\varphi_\e$, using Property~(2) and that $K\in E^\e$.

\begin{figure}[!htbp]
    \centering
    \includegraphics[width=0.78\linewidth]{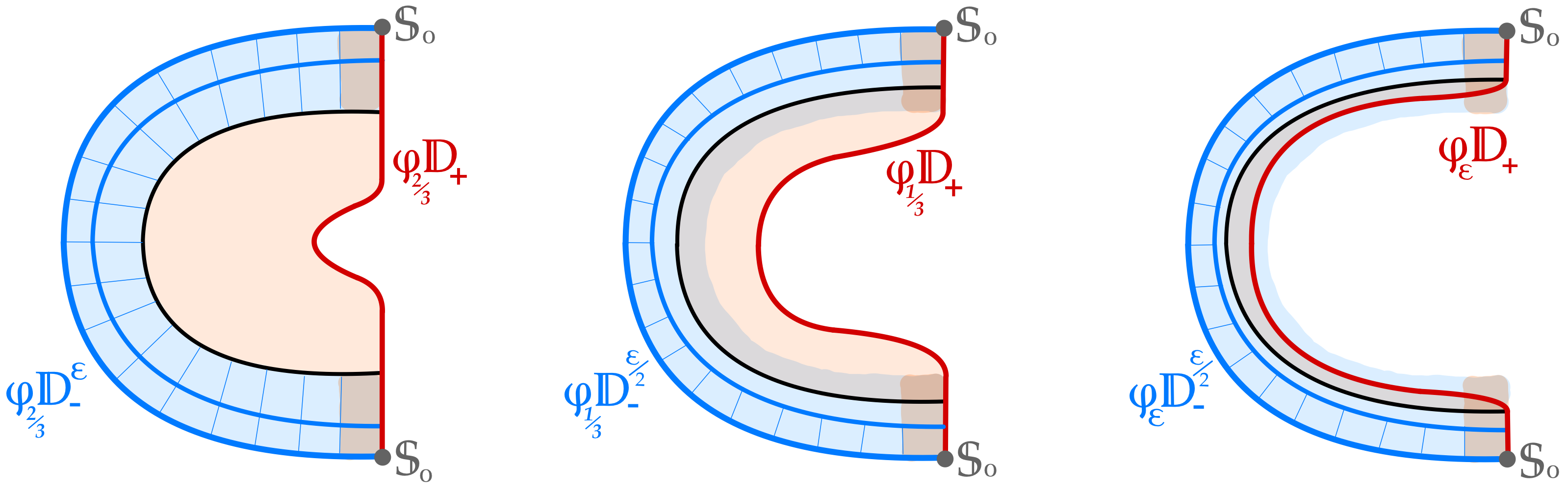}
    \caption{The image of $\varphi_t$ for $t=2/3,1/3,\e$. Dashed strips show where $\varphi_t$ is the identity; they are always contained in the blue-colored strip $\D_-^\e\subseteq\HD^k$. The black line is the image of $\D^{k-1}\times\{\e\}\subseteq\D_+^\e\subseteq\HD^k$.}
    \label{fig:push-to-minus}
\end{figure}
We next modify this homotopy to have image contained in the subspace $E^\e$. We fix an ambient isotopy $\Phi_t\colon X\xrightarrow{\cong} X$, $t\in[\e,1]$, supported in a collar of $\partial X$, such that $\Phi_1=\Id_X$ and $\U\circ\varphi_t|_{\D_-^\e}=\Phi_t^{-1}\circ\U|_{\D_-^\e}$ for $t\in[\e,1]$. This can be constructed explicitly in a collar $\partial X\times[0,\e]\hra X$ (or extend the isotopy of half-disks $\U\varphi_t$ by the usual ambient isotopy theorem). 

Then for $t\in[\e,1]$ let
\[
    R_t(K)\coloneqq \Phi_t\circ K\circ\varphi_t.
\]
This defines a path from $R_\e(K)=\Phi_\e K\varphi_\e=\Phi_\e \U\varphi_\e$ to $R_1(K)=\Id_XK\Id_\HD=K$, and now each  half-disk $R_t(K)$ for $t\in[\e,1]$ is in $E^\e$, i.e.\ it agrees with $\U$ on $\D_-^\e$. Indeed, by Property~(3) and $K\in E^\e$ we have $K\varphi_t|_{\D_-^\e}=\U\varphi_t|_{\D_-^\e}$, so $R_t(K)|_{\D_-^\e}\coloneqq \Phi_t K\varphi_t|_{\D_-^\e}=\Phi_t \U\varphi_t|_{\D_-^\e}=\U|_{\D_-^\e}$ by construction of $\Phi$.

Finally, for $t\in[0,\e]$ we let $R_t(K)\coloneqq R_{1+t-t/\e}(\U)$. This goes from $\U$ to $\Phi_\e \U\varphi_\e$, so glues with the above $R_t,t\in[\e,1]$ to a map $R\colon E^\e\times[0,1]\to E^\e$ which is the desired contraction as in \eqref{eq:contraction}. Finally, the homotopy inverse to $\amb_\U$ is defined by the formula 
\[
    \foliate^\e_\U(K)(t)=R_t(K)|_{\D_+^\e}=\begin{cases}
        \Phi_{1+t-t/\e}\circ \U\circ\varphi_{1+t-t/\e}|_{\D_+^\e},     & t\in[0,\e]\\
        \Phi_t\circ K\circ\varphi_t|_{\D_+^\e},                & t\in[\e,1].
    \end{cases}
\]
that we call the foliation: outside of a fixed collar of $X$ it agrees either with the $\e$-augmented arc $\U\varphi_t|_{\D_+^\e}$ or $K\varphi_t|_{\D_+^\e}$, and in the collar uses their modifications by $\Phi_t$, making ``a turn'' at $\wt{u}_-^\e\coloneqq\Phi_\e \U\varphi_\e|_{\D_+^\e}$.
% contained in $\U(\D_-^\e)$.
\end{proof}

\begin{remark}\label{rem:amb-directly}
    % Instead of constructing the connecting map $\amb$ using Lemma~\ref{lem:Hurewicz}, one can argue directly that the spaces in Theorem~\ref{thm:anyD} have isomorphic homotopy groups. Namely, one can define \emph{smooth homotopy groups} of $\Emb(Z,X)$, which are isomorphic to the usual homotopy groups by smooth approximation results. Given a restriction map as at the beginning of the proof of Theorem~\ref{thm:anyD}, 
    % the parametrized ambient isotopy theorem allows to lift such smooth maps and as a consequence gives a long exact sequence, where a smooth version of the connecting map is an isomorphism.
    % However, when we talk about our homotopy equivalence at the level of \emph{spaces}, this is not quite sufficient, as in fact measured by the difference between Serre and Hurewicz fibrations.
    Instead of constructing the connecting map $\amb$ using Lemma~\ref{lem:Hurewicz}, one can give the following argument that the spaces in Theorem~\ref{thm:anyD} have isomorphic homotopy groups. Using the adjunction $C^\infty(S, C^\infty(Y,X)) = C^\infty(S\times Y,X)$ for the Whitney topology, one defines \emph{smooth homotopy groups} of $\Emb(Y,X)$. They turn out to be isomorphic to the usual homotopy groups by smooth approximation results. Given a restriction map as at the beginning of the proof of Theorem~\ref{thm:anyD}, the parametrized ambient isotopy theorem allows to lift smooth maps and as a consequence gives a long exact sequence of smooth homotopy groups, where a smooth version of the connecting map is an isomorphism. However, when we talk about our homotopy equivalence at the level of \emph{spaces}, this is not quite sufficient.
\end{remark}

%%%%%%%%%%%%%%%%%%%%%
\subsection{From neat disks to half-disks}\label{subsec:neat-to-half}

Recall that the model half-disk $\HD^k\subseteq \D^k$ has boundary decomposed into two $(k-1)$-disks $\partial \HD^k = \D_- \cup_{\S_0} \D_+$ intersecting along the $(k-2)$-sphere $\S_0$, the corner of $\HD^k$. Also recall that $\U\colon\HD^k\hra X$ by definition restricts to a neat $(k-1)$-disk $u_+\colon\D_+\hra X$, while the image of $\D_-$ is contained in $\partial X$.
Using a Riemannian metric on $X$ we extend $u_+$ to an embedding $V\colon\D_+ \times \D^{d-k+1}_{\leq\e}\hra X$ onto a closed tubular neighborhood $\ol{\nu}^{\e}u_+$. We may assume that the restriction $V|\colon\D_+\times[0,\e]\cong\D_+^\e\hra X$ to the first normal vector agrees with our preferred $\e$-augmentation $u_+^\e=\U|_{\D_+^\e}$ and also, by decreasing $\e$ if necessary, that $\im(u_+^\e)=\ol{\nu}^{\e}u_+\cap\im(\U)$, i.e.\ the half-disk $\U$ does not return to this neighborhood of $u_+$.

We can view $V(\D_+\times\D^{d-k+1}_{\leq\e/2})$ as a $(d-k+1)$-handle attached to $X\sm V(\D_+\times\D^{d-k+1}_{<\e/2})$ along $V(\D_+\times\S^{d-k}_{\e/2})$, see Figure~\ref{fig:cutting-corners1}. Conversely, that complement is obtained from $X$ by removing a $(k-1)$-handle with core $u_+(\D_+)$, and is a smooth manifold with boundary only if we first \emph{smoothen the corner} $V(\S_0\times\S^{d-k}_{\e/2})$.

This is a standard procedure, used for example when attaching handles in the smooth category. In our context it amounts to picking an open subset $h_+\subseteq \ol{\nu}^{\e}u_+$ that is the union of $V(\D_+\times\mathring{\D}^{d-k+1}_{\leq\e/2})$ and a small set near the corner, so that $X\sm h_+$ is a compact smooth manifold with boundary. We make such a choice once and for all in $\D_+ \times \D^{d-k+1}_{\leq\e}$, and then let $h_+$ be its image in $X$ under $V$, see Figure~\ref{fig:cutting-corners2}. Namely, the constant radius $\e/2$ along $\D_+$ increases near $\S_0$ by a smooth function with fairly obvious properties (ensuring that the stretching function $\sigma$ in the next proof is well defined on all of $h_+$).
\begin{figure}[!htbp]
    \centering
    \begin{minipage}[c]{.5\textwidth}
      \centering
      \includegraphics[width=0.46\linewidth]{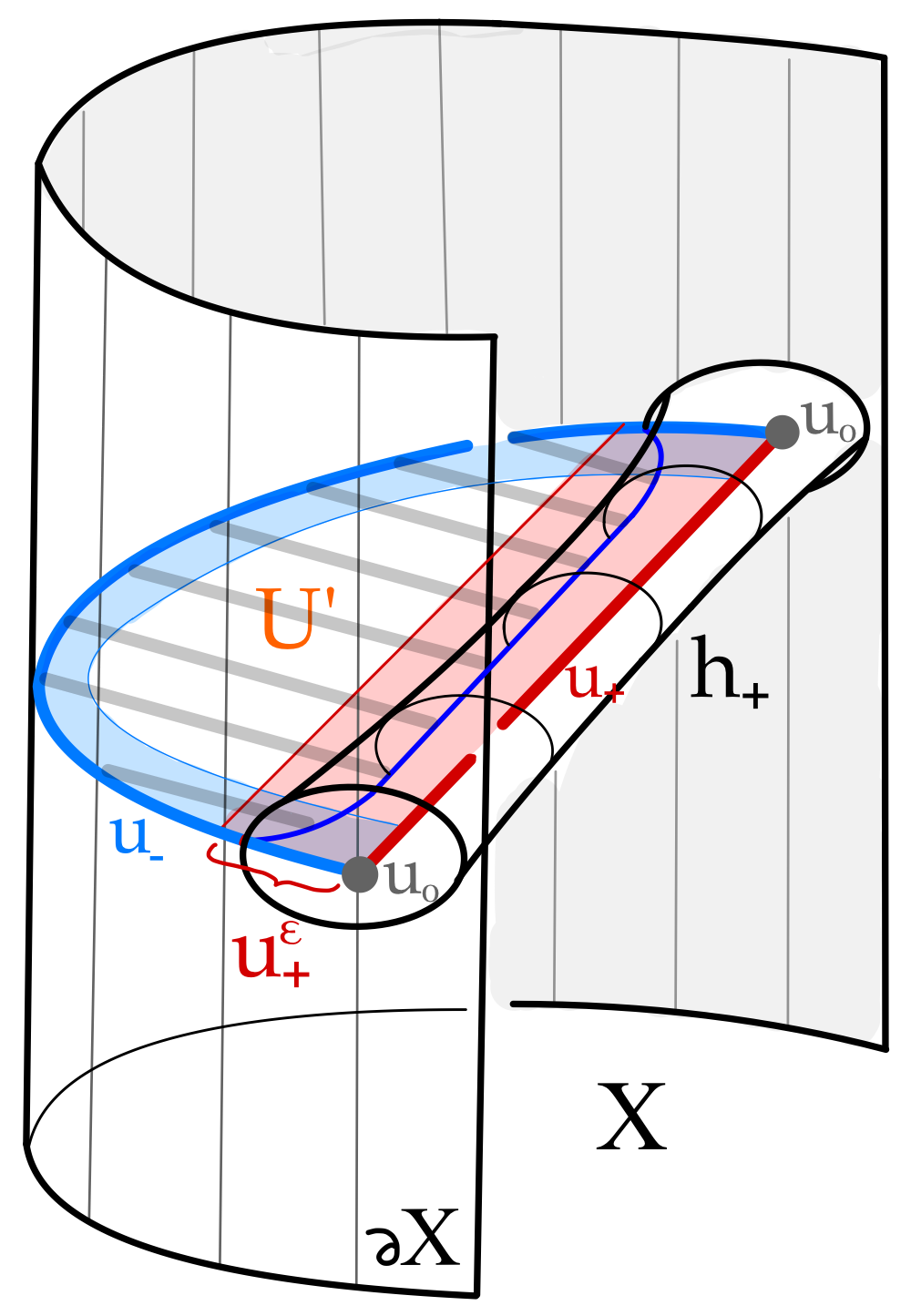}
      \captionof{figure}{Removing a handle $h_+\subseteq X$ turns the half-disk $\U$ in $X$ into a neat disk~$\U'$ in $X\sm h_+$.}
      \label{fig:cutting-corners1}
    \end{minipage}%
    \begin{minipage}[c]{.45\textwidth}
      \centering
      \vspace{0.3cm}
      \includegraphics[width=0.47\linewidth]{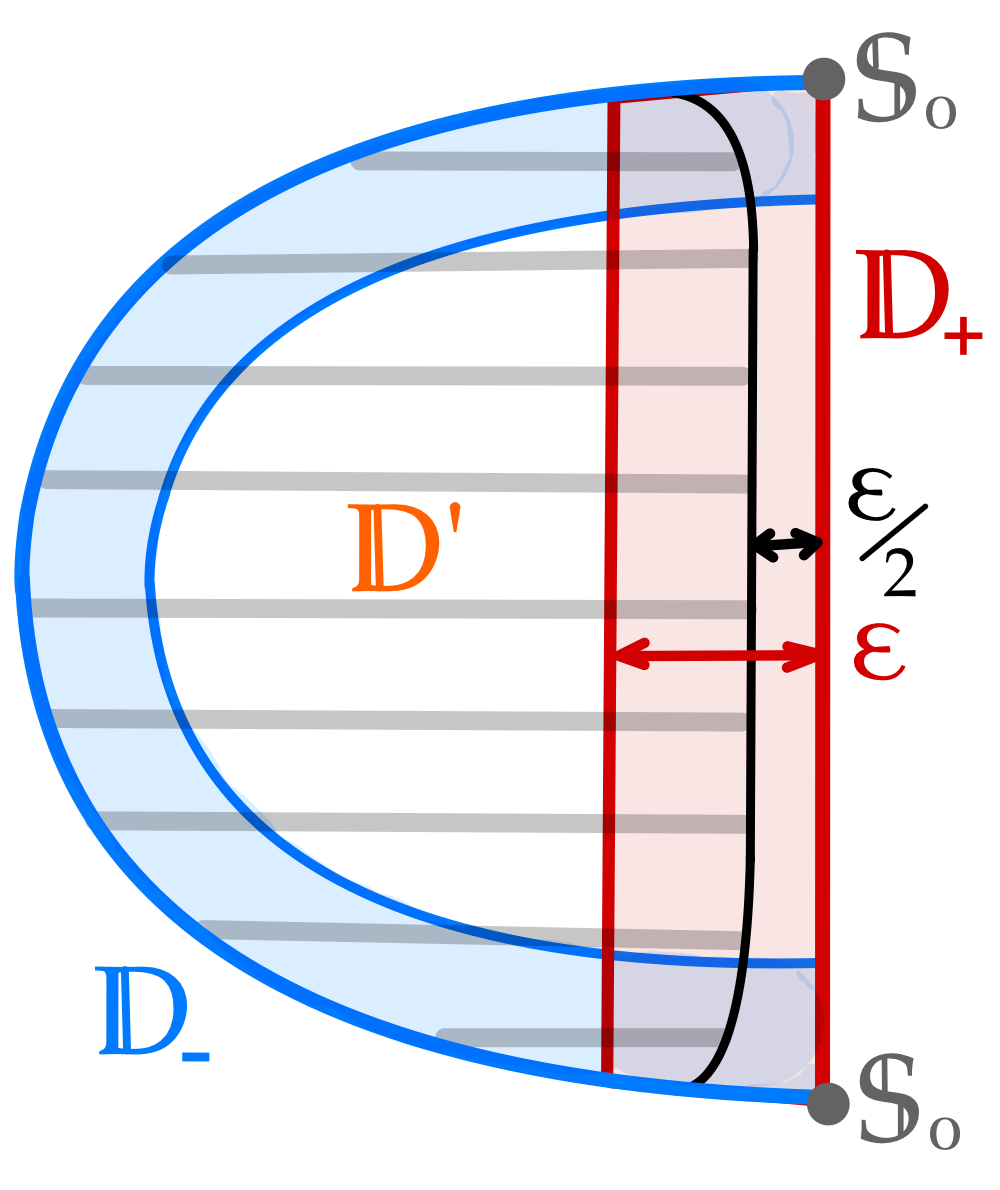}
      \vspace{0.5cm}
      \captionof{figure}{The model smoothening.}
      \label{fig:cutting-corners2}
    \end{minipage}
\end{figure}

Now let $\D'\coloneqq \HD^k \sm V^{-1}(h_+) \subseteq \HD^k$ and fix a diffeomorphism $\D^k \cong \D'$ that is the identity near $\D_-^\e\sm \D_+^\e$. Then the restriction of $\U$ to $\D'$ is a neat embedding $\U'\colon\D'\hra X\sm h_+$, by our choice of $\e$. This is an element in the space of neat embeddings $\Emb_{s^{\e/2}}(\D',X\sm h_+)\subseteq\Emb(\D^k,X\sm h_+)$, where the boundary condition is the remaining $\e/2$-part of $u_+^\e$ in $\D'$, as well as the original $u_-^{\e}$ along $u_-$. Note that we can reconstruct $\U$ from $\U'$ as $\U=\U'\cup u_+^\e$. 

\begin{lemma}\label{lem:X-M}
    The map $\bull\cup u_+^\e\colon\Emb_{s^{\e/2}}(\D',X\sm h_+)\to\Emb_{s^\e}(\HD^k,X)$ is a homotopy equivalence.
\end{lemma}
\begin{proof}
    The chosen boundary conditions $s^{\e/2}$ make this map well-defined. It is continuous with image $\mathcal{E}\subseteq\Emb_{s^\e}(\HD^k,X)$ consisting of those half-disks that meet $h_+$ only along $\im(u_+^{\e/2})$. In fact, it is a homeomorphism onto $\mathcal{E}$ whose inverse is given by restricting embeddings from $\HD^k$ to $\D'$. It thus suffices to construct a homotopy inverse from $\Emb_{s^\e}(\HD^k,X)$ back to this subspace $\mathcal{E}$.

    To this end, use the Riemannian metric on $X$ to obtain the continuous map
\[
    r\colon \Emb_{s^\e}(\HD^k,X) \ra (0,\tfrac{3}{5}\e)
\]
    so that $r(K)$ gives the minimal distance of $K(\D')$ to $K(\D_+)=\U(\D_+)$. 
    Possibly shrinking $\e$ further, we may assume that the geodesic distance to $\U(\D_+)$ on the sphere bundle in $\nu^\delta u_+$ is $\delta$ for all $\delta<\e$.
    By compactness and the injectivity of $K$, $r(K)$ is strictly positive as claimed and we now stretch it to $\frac{3}{5}\e$, in order to deform $K$ until it lies in $\mathcal{E}$. 
    We pick a smooth ``stretching'' function $\sigma\colon (0,\tfrac{1}{2}\e) \times [0,1] \times [0,\e] \to [0,\e]$
    such that $\sigma(r,t,x)=x$ whenever one of the following conditions is satisfied: $t=0$ or $x=0$ or $x\geq\frac{4}{5}\e$.
    Moreover, we require $\sigma(r,1,r)=\frac{3}{5}\e$ for all $r$ and that each $\sigma(r,t,-)$ is strictly increasing, see Figure~\ref{fig:stretch-and-s}.
 
If $(x,v) \in [0,\e] \times \S^{d-k}$ are polar coordinates, then we will refer to ``stretching by $\sigma(r,t)$'' as the self-diffeomorphism $(x,v)\mapsto (\sigma(r,t,x),v)$ of $\D^{d-k+1}_{\leq\e}$, see Figure~\ref{fig:stretch(K)}. 
The same formula applies to disk bundles in vector bundles if we stretch in a constant way along the base. 
Using the parametrization $V$, we can apply such diffeomorphisms also to our tubular neighborhood $\ol{\nu}_{\e}u_+$.
The stretching near the smoothened corners along $\S_0$ needs to be slightly modified but we leave this variation to the reader.
\begin{figure}[!htbp]
    \centering
    \begin{minipage}[b]{.45\textwidth}
      \centering
      \includegraphics[width=0.7\linewidth]{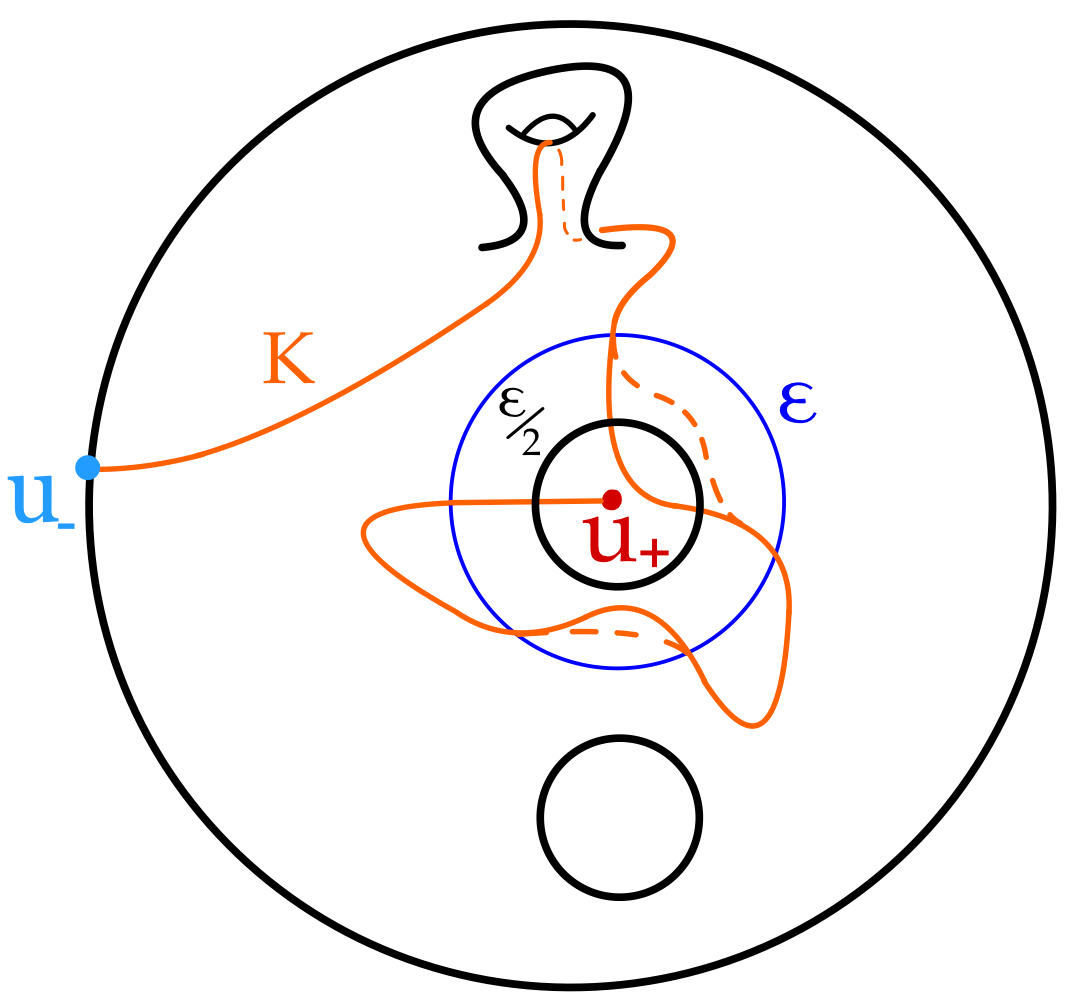}
      \captionof{figure}{Stretch $K$ towards the dashed lines to avoid the smallest central disk $h_+$.}
    \label{fig:stretch(K)}
    \end{minipage}%
    \begin{minipage}[b]{.54\textwidth}
      \centering
      \includegraphics[width=0.62\linewidth]{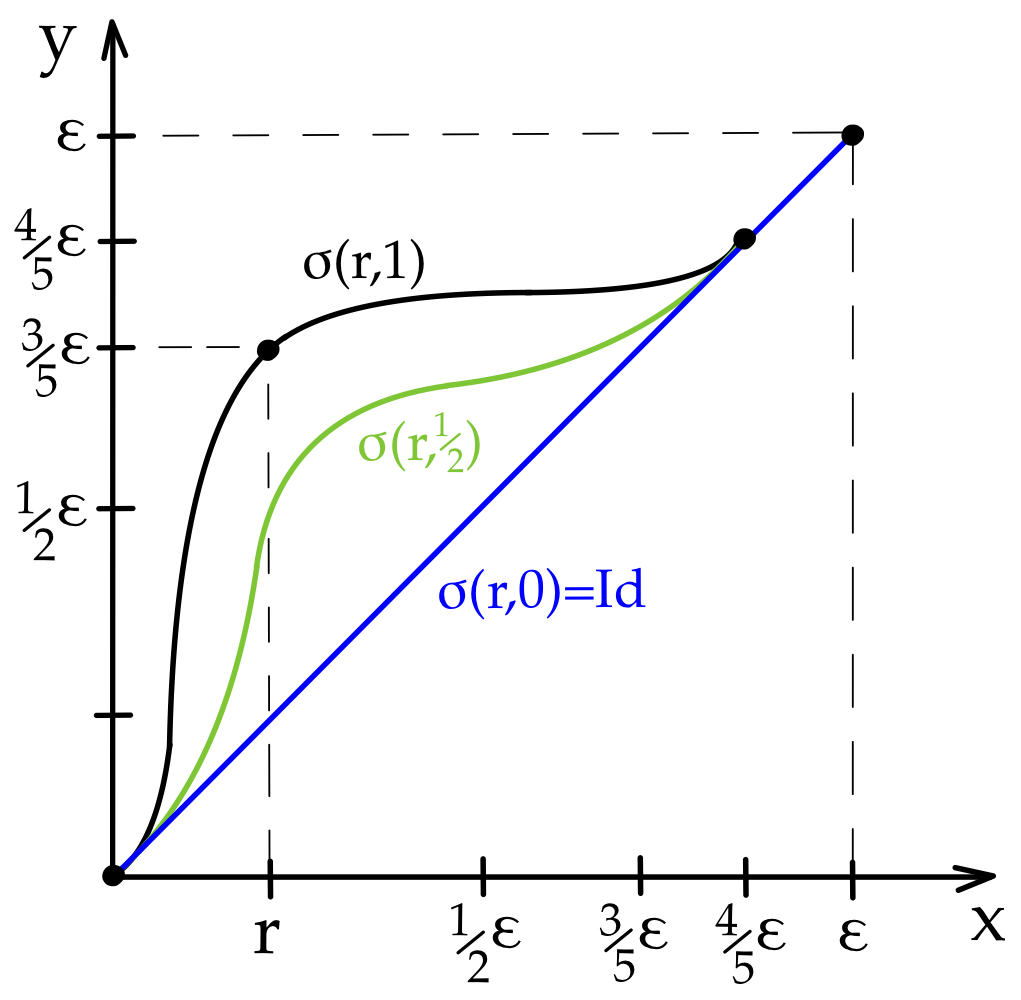}
      \captionof{figure}{Stretching functions $\sigma(r,t)\colon[0,\e]\to[0,\e]$ for fixed $r<\e/2$ and three values of $t\in[0,1]$.}
    \label{fig:stretch-and-s}
    \end{minipage}
\end{figure}

We can then construct a homotopy $H\colon \Emb_{s^\e}(\HD^k,X)\times[0,1]\to \Emb_{s^\e}(\HD^k,X)$
with $H_0=\Id$ and $\im(H_1)\subseteq \mathcal{E}$, induced by a smooth family of diffeomorphisms $\phi_{r(K),t}\colon X\to X$ that are the identity outside $\nu^{\frac{4}{5}\e}u_+$ and on $\ol{\nu}_{\e}u_+$ they stretch by $\sigma(r(K),t)$. 
More precisely, we define $H(K,t)$ to be the half-disk that equals $K$ on $\partial^{\e}\HD^k$ but away from that collar is given by the composition $\phi_{r(K),t}\circ K$. 

The properties of the stretching function $\sigma$ show that each $H_t$ sends $\mathcal{E}$ to itself and that $H_1$ is the required homotopy inverse. In fact, homotopies for both compositions to the identity are constructed from $H_t$ as follows: If $j\colon\mathcal{E} \subseteq \Emb_{s^\e}(\HD^k,X)$ is the inclusion then $H_1\circ j\colon\mathcal{E}\to \mathcal{E}$ is homotopic via $H_t\circ j$ to $H_0\circ j= \Id_{\mathcal{E}}$. Similarly, $j \circ H_1=H_1\simeq H_0=\Id$.
\end{proof}

% \begin{remark}
%     Observe that $\partial \big(X\sm h_+\big)\cong  \partial X\,\#\, \big(\S^{k-1}\times\S^{d-k}\big)$,
%     being the result of surgery on $\partial X$ along the unknotted sphere $u_-(\S_0)\subseteq\partial X$ (that bounds the embedded disk $u_-(\D_-)\subseteq\partial X$).
% \end{remark} 

\subsubsection{The proof of Theorem B}

We can now prove Theorem~\ref{thm-intro:LBT-deloop}. This is about the space $\Emb_{s^\e}(\D^k,M)$ of neat $k$-disks in a $d$-manifold $M$, that on a collar of the boundary agree with a basepoint $\U\colon\D^k\hra M$, such that $s=\partial\U\colon\S^{k-1}\hra\partial M$ has a \emph{framed geometric dual} $G\colon\S^{d-k}\hra \partial M$ (that is, the normal bundle $\nu_{\partial M}(G)$ is trivial and $G\pitchfork s=\{p\}$). Then the theorem says $\Emb_{s^\e}(\D^k,M)\simeq\Omega\Emb_{u_0}^\e(\D^{k-1},M_G)$, where $M_G\coloneqq M\cup_{(G,\psi)}h$, is obtained by attaching a $(d-k+1)$-handle  $h=\D^{d-k+1}\times\D^{k-1}$ along any framing $\psi\colon \S^{d-k}\times\D^{k-1}\hra \nu_{\partial M}(G)$ (this choice is inessential because $M_G\cong M\sm\nu \U$, see Lemma~\ref{lem:MG}).

\begin{proof}[Proof of Theorem~\ref{thm-intro:LBT-deloop}]
    Removing from $M_G$ an open $\e$-neighborhood $h_+$ of the cocore $\{0\} \times \D^{k-1}$ of $h$ gives $M$ back because we are all together just attaching $\S^{d-k} \times [\e,1] \times\D^{k-1}$ along $\S^{d-k} \times \{1\} \times\D^{k-1}\hra \partial M$. Using this diffeomorphism $M_G\sm h_+\cong M$ and Lemma~\ref{lem:X-M} we have
      \[
        \Emb_{s^\e}(\D^k,M)\simeq\Emb_{s^\e}(\D^k,M_G\sm h_+) \simeq \Emb_{s^\e}(\HD^k,M_G).
      \]
    Applying Theorem~\ref{thm:anyD} to $X\coloneqq M_G$, and Proposition~\ref{prop:e-collar}, we obtain
      \[
        \Emb_{s^\e}(\HD^k,M_G)\simeq\Omega\Emb_{u_0^\e}^\e(\D^{k-1},M_G).
      \]
    The final statement of Theorem~\ref{thm-intro:LBT-deloop}, identifying homotopy equivalences, follows from Proposition~\ref{prop:forg-augm}.
\end{proof}

\section{On homotopy groups of spaces of neat disks}\label{sec:disks}

In this section we apply Dax's results to compute the first homotopy group of the space of embedded disks differing from that of immersed disks. In the next section we extend this to $\e$-augmented disks. 

In fact, in Section~\ref{subsec:Dax} we follow Dax and work more generally with $V,X$ smooth (oriented) manifolds with boundary, and $V$ compact (but $X$ not necessarily). We consider the space $\Emb_\partial(V,X)$ (for $V=\D^{k-1}$ this was previously denoted $\Emb_{u_0}(\D^{k-1},X)$) of neat embeddings of $V$ into $X$ that are on the boundary given by $u_0\coloneqq u|_{\partial V}$, for a fixed ``unknot'' $u\colon V\hra X$; note that no dual is assumed. 

\begin{remark}
    For our applications we will actually need the space $\Emb_{\partial^\e}(\D^\ell,X)$ (for $\ell=k-1$ previously denoted $\Emb_{u_0^\e}(\D^{k-1},X)$), that consists of embeddings $\D^\ell\hra X$ that agree on $\S^{\ell-1}\times [0,\e]\subseteq\D^\ell$ with one such $u^\e$ (previously denoted $u_+^\e$). However, $\Emb_{\partial^\e}(\D^\ell,X)\simeq \Emb_\partial(\D^\ell,X)$, see Proposition~\ref{prop:e-collar}.
\end{remark}

We study the homotopy groups of $\Emb_\partial(V,X)$ by comparing them to the space $\Imm_\partial(V,X)$ of immersions $V\imra X$ with the same boundary condition $u_0$. Denoting $\dim V=\ell\leq d=\dim X$, and assuming $d-2\ell\geq0$ and $d\geq \ell+3$, the Dax invariant will be an isomorphism (bijection for $d-2\ell=0$ or $1$):
\[
    \Da\colon\;\pi_{d-2\ell}\big(\Imm_\partial(V,X),\Emb_\partial(V,X);u\big) \ra \faktor{\Z[\pi_1X]}{\relations_{\ell,d}}
\]
given on a relative homotopy class 
\[
    F\colon\; \left(\I^{d-2\ell},\,\I^{d-2\ell-1}\times\{0\},\,\partial\I^{d-2\ell-1}\times\I \cup \I^{d-2\ell-1}\times\{1\}\right)
    \ra (\Imm_\partial(V,X),\,\Emb_\partial(V,X),\,u)
\]
as the count of double point loops of the associated track $\wt{F}\colon \I^{d-2\ell}\times V\to \I^{d-2\ell}\times X$, see Theorem~\ref{thm:Dax-invt}. The group of relations $\relations_{\ell,d}$ is trivial for $\ell=1$, whereas for $\ell\geq2$ it is given by $\langle g-(-1)^{d-\ell}g^{-1}: g\in\pi_1X\rangle$.

We will then concentrate on the case $V=\D^\ell$ and study the connecting map
\begin{equation}\label{eq:delta}
    \delta_{\Imm}\colon\pi_{d-2\ell}\Imm_\partial(\D^\ell,X)\ra\pi_{d-2\ell}(\Imm_\partial(\D^\ell,X),\Emb_\partial(\D^\ell,X),u).
\end{equation} 
in order to prove the following in Section~\ref{subsec:immersions-conn-map}.

\begin{theorem}\label{thm:Dax-final}
    The inclusion $\incl\colon\Emb_\partial(\D^\ell,X)\hra\Imm_\partial(\D^\ell,X)$ is $(d-2\ell-1)$-connected. Assume $d\geq \ell+3$. If $d-2\ell-1\geq1$ there is a short exact sequence of groups:
    \[\begin{tikzcd}[column sep=large]
          \faktor{\Z[\pi_1X\sm1]}{ \relations_{\ell,d}\oplus \md_u(\pi_{d-\ell}X) } \arrow[tail, shift left]{r}{\partial\realmap} &
          \pi_{d-2\ell-1}(\Emb_\partial(\D^\ell,X),u)\arrow[two heads]{r}{\pi_{d-\ell-1}\incl}\arrow[dashed,shift left]{l}{\Da} &
         \pi_{d-2\ell-1}(\Imm_\partial(\D^\ell,X),u),
        \end{tikzcd}
    \]
    where the homomorphism $\md_u\colon\pi_{d-\ell}X\to\Z[\pi_1X\sm1]/ \relations_{\ell,d}$ is defined in~\eqref{eq:md-def} below in terms of $\Da\circ\delta_{\Imm}$. The dashed arrow $\Da$ inverts on $\ker(\pi_{d-\ell-1}\incl)$ an explicit realization map $\partial\realmap$. 
    
    If $d-2\ell-1=0$ there is an analogous exact sequence of sets, with $u$ omitted and $\partial\realmap$ extending to an action on $\pi_0\Emb_\partial(\D^\ell,X)$. Finally, if $d-2\ell=0$ there is an exact sequence of sets
    \[\begin{tikzcd}
        \pi_0\Emb_\partial(\D^\ell,X)
        \arrow{r}{\pi_0\incl} &
        \pi_0\Imm_\partial(\D^\ell,X)
        \arrow{rr}{\Da\circ\delta_{\Imm}} &&
        \faktor{\Z[\pi_1X]}{\relations_{\ell,d}},
        \end{tikzcd}
    \]
    and $\Da\circ\delta_{\Imm}$ agrees with the Wall self-intersection invariant $\mu_\ell$ (see Remark~\ref{rem:Grant}).
\end{theorem}

This can be made more explicit by computing homotopy groups of $\Imm_\partial(\D^\ell,X)$ using Smale--Hirsch immersion theory, see Corollary~\ref{cor:imm-htpy}, which immediately gives Theorem~\ref{thm-intro:Dax}. An analogue for $\e$-augmented disks will be given as Theorem~\ref{thm:Dax-augmented} in Section~\ref{subsec:ribbons}, and its case $\ell=k-1$ and $X=M_G$ will be used to prove Theorem~\ref{thm-intro:combined} in Section~\ref{subsec:proof-of-ThmD}.

The first sentence of Theorem~\ref{thm:Dax-final} follows by general position. Namely, for a family of immersed disks $F(\vec{t})\colon \D^\ell\imra X$, $\vec{t}\in\S^n$, their double points correspond to double points of the track $\wt{F}\colon \S^n\times\D^\ell\to \S^n\times X$, defined by $(\vec{t},v)\mapsto(\vec{t},F(\vec{t})(v))$. The set of double points of $\wt{F}$ has dimension $n+d-2(d-\ell)=n-d+2l$. This is negative if $n<d-2\ell$, when a generic $n$-family is embedded, i.e.\ gives a class in $\pi_n\Emb_\partial(\D^\ell,X)$. If $n<d-2\ell-1$ then these lifts are also unique, by an analogous argument with one more parameter, implying the injectivity on $\pi_n$ as well.

Thus, it remains to determine the kernel of $\pi_{d-2\ell-1}\incl$. This amounts to computing the relative homotopy group in degree~$d-2\ell$ and the image of the connecting map $\delta_{\Imm}$.

%%%%%%%%%%%%%%%%
\subsection{The work of Dax: the relative homotopy group}\label{subsec:Dax}

\subsubsection{The original formulation}
As mentioned in the introduction (Section~\ref{subsec-intro:Dax}), Dax computes the homotopy groups of immersions relative to embeddings in the metastable range $0\leq n\leq 2d-3l-3$, in terms of certain bordism groups. His manifolds are more general than considered so far in this paper: they can be disconnected and nonorientable, and the target can be noncompact.

\begin{theorem}[{\cite[375]{Dax}}]\label{thm:Dax}
    Let $V$ be a smooth, compact $\ell$-manifold with boundary, $X$ a smooth $d$-manifold with boundary, and $u\colon V\hra X$ a smooth neat embedding. If $1\leq \ell\leq d$ are such that $d-2\ell\geq0$, then for $0\leq n\leq 2d-3l-3$ there is an explicit isomorphism (injection for $n=0$):
    \[
        \beta_n\colon
        \pi_n\big(\Imm_\partial(V,X),\Emb_\partial(V,X);u\big) 
        \ra \Omega_{n-(d-2\ell)}(\C_u;\vartheta_u).
    \]
\end{theorem}

For $n=0$ the relative homotopy group is understood as the set-theoretic quotient of $\pi_0\Imm_\partial(V,X)$ by the image of $\pi_0\Emb_\partial(V,X)$. Thus, $\beta_0$ is a complete obstruction for an immersion to be regularly homotopic to an embedding; this was also studied in \cite{Hatcher-Quinn} using similar techniques, see Remark~\ref{rem:Grant}.

As usual, the normal bordism group $\Omega_i(Y;\vartheta)$ of a space $Y$ with a stable vector bundle $\vartheta$ over it, consists of bordism classes of tuples $(D,b\colon D\to Y,B\colon b^*(\vartheta)\to\nu_D)$ where $D$ is an $i$-manifold with the stable normal bundle $\nu_D$, $b$ is a map, and $B$ is a bundle isomorphism. We do not recall here the definitions from \cite[338]{Dax} of the space $\C_u$ and the stable vector bundle $\vartheta_u$ over it, to save space as we do not need them explicitly. We only point out the following three properties they have.

Firstly, there is a fibration sequence $\Omega X\to \C_u\xrightarrow{pr_W} W$, where $W$ is a compactification (to a manifold with boundary) of the quotient of $V^2\sm \Delta_V$, where $V^2=V\times V$ and $\Delta_V\coloneqq\{(v,v)\in V^2\}$ is the diagonal, by the involution $(v,w)\mapsto(w,v)$. Thus, the interior is the space of coinvariants, $\mathrm{int}\,W= (V^2\sm \Delta_V)_{\Z/2}$.

Secondly, the subspace $E^\curlyvee_u\coloneqq pr_W^{-1}(\mathrm{int}\,W)\subseteq\C_u$ can be described as the quotient of the space
\[
    \wt{E}^\curlyvee_u\coloneqq\big\{(v,w,\rho)\in (V^2\sm\Delta_V)\times \Map([-1,1],X) : \rho(-1)=u(v),\, \rho(1)=u(w)\big\}
\]
by the free involution $(v,w,\rho)\mapsto(w,v,\rho^{-1})$. We use notation $[v,w,\rho]\in E^\curlyvee_u$. This description follows from the definition of the bundle $\omega$ over $W$, implying it is trivial over $\mathrm{int} W$, see \cite[337]{Dax}.

Thirdly, the restricted bundle $\vartheta_u|_{E^\curlyvee_u}$ is obtained as the quotient of a bundle $\wt{\vartheta}_u$ over $\wt{E}^\curlyvee_u$. Namely, for the stable normal bundle $\nu_V$ and the tangent bundle $TX$, let $\wt{\vartheta}_u$ be the pullback of $\nu_V^2\oplus TX$ under $(pr_{V^2},pr_X)\colon \wt{E}^\curlyvee_u\to V^2\times X$, $(v,w,\rho)\mapsto (v,w,\rho(0))$. This map is equivariant for the involution that switches the two $V$ factors and is identity on $X$, giving an involution on $\wt{\vartheta}_u$, whose quotient is $\vartheta_u|_{E^\curlyvee_u}$. The following is used in the survey article~{\cite{GKW}} without proof (with $E^\curlyvee_u$ denoted by $E^\curlyvee(u,u)$ there).
\begin{lemma}
    The space $\C_u$ is homotopy equivalent to its subspace $E^\curlyvee_u$. In particular,
    \[
        \Omega_{n-(d-2\ell)}(\C_u;\vartheta_u)\cong\Omega_{n-(d-2\ell)}(E^\curlyvee_u;\, \vartheta_u|_{E^\curlyvee_u}).
    \]
\end{lemma}
\begin{proof}
    As $\mathrm{int}\,W$ is the interior of the compact manifold $W$ with boundary, the inclusion $i\colon\mathrm{int}\,W\hra W$ is a homotopy equivalence. Thus, the pullback $i^*(\C_u)=p^{-1}(\mathrm{int}\,W)$ is homotopy equivalent to $\C_u$. 
\end{proof}

Let us now translate the isomorphism $\beta_n$ from Theorem~\ref{thm:Dax} to this simpler target bordism group as
\begin{equation}\label{eq:beta_n'}
    \beta_n'\colon
    \pi_n\big(\Imm_\partial(V,X),\Emb_\partial(V,X);u\big) 
    \overset{\cong}{\ra} \Omega_{n-(d-2\ell)}(E^\curlyvee_u;\, \vartheta_u|_{E^\curlyvee_u}).
\end{equation}
We also need the following standard result; see for example \cite[335]{Dax} (note that Dax considers all maps rather than immersions, so has more conditions on perfect maps, regarding singular points).
\begin{lemma}\label{lem:Dax}
   Under assumptions of Theorem~\ref{thm:Dax}, a smooth map
\[
F\colon\; 
    \left(\I^n,\, \I^{n-1}\times\{0\},\, \partial\I^{n-1}\times\I \cup \I^{n-1}\times\{1\}\right)\ra \left(\Imm_\partial(V,X),\,\Emb_\partial(V,X),\,u\right)
\] % We use this convention for the current description of realization map to have the positive sign... Dax takes 0-face to be const u.
    can be approximated, relative to the boundary, by a \emph{perfect map}, i.e.\ a smooth map $F$ whose track
\begin{equation}\label{eq:track}
\wt{F}\colon\;
    \I^n\times  V \ra \I^n\times X,\quad (\vec{t},v)
    \mapsto (\,\vec{t},\, F(\vec{t})(v)\,)
\end{equation}
  has no triple points, and double points are isolated and transverse. Equivalently, the restricted square $\wt{F}^2|\colon(\I^n \times  V)^2\sm\Delta_{\I^n \times  V}\to (\I^n \times X)^2$ is transverse to the diagonal $\Delta_{\I^n \times X}$.
\end{lemma}
For $F$ a perfect map as in the lemma, let $\beta_n'[F]$ be the bordism class of the tuple $(\Delta_{Dax},b_{Dax},B_{Dax})$ defined as follows. Firstly, the double point preimage set is defined as
\[
    \wt{\Delta}_{Dax}\coloneqq(\wt{F}^2|)^{-1}(\Delta_{\I^n \times X})
    \cong\left\{(\vec{t},v,w)\in \I^n\times(V^2\sm\Delta_V): F(\vec{t})(v)=F(\vec{t})(w)\right\},
\]
whereas $\Delta_{Dax}$ is the quotient of $\wt{\Delta}_{Dax}$ by the free involution interchanging the factors:
\[
    \Delta_{Dax}\coloneqq(\wt{\Delta}_{Dax})_{\Z/2}\cong
    \left\{(\vec{t},[v,w])\in\I^n\times\mathrm{int}\,W:F(\vec{t})(v)=F(\vec{t})(w)\right\}.
\]
Note that $F\colon\Delta_{Dax}\hra X$, $F(\vec{t},[v,w])=F(\vec{t})(v)=F(\vec{t})(w)$, embeds $\Delta_{Dax}$ as the double point set, and we have a double cover
$q\colon \wt{\Delta}_{Dax}\ra\Delta_{Dax}$.
Next, $b_{Dax}\colon\Delta_{Dax}\to E^\curlyvee_u$ is defined by
\begin{equation}\label{eq:b-Dax}
    b_{Dax}(\vec{t},[v,w])= \big[v,w,\rho\coloneqq F(\vec{1}-s(\vec{1}-\vec{t}))(v)|_{s\in[0,1]}\cdot F(\vec{1}-s(\vec{1}-\vec{t}))(w)|_{s\in[0,1]}^{-1} \big].
\end{equation}
So $\rho$ is a path in $X$ from $\rho(-1)=F(\vec{1})(v)=u(v)$ to $\rho(0)=F(\vec{t})(v)=F(\vec{t})(w)$ followed by a path to $\rho(1)=F(\vec{1})(w)=u(w)$. Note that Dax takes relative classes to instead be $u$ on the $0$-face $\I^{n-1}\times\{0\}$, so has a slightly different formula for $\rho$.

Finally, $B_{Dax}$ is an isomorphism of the stable normal bundle $\nu_{\Delta_{Dax}}$ and $b_{Dax}^*(\vartheta_u|_{E^\curlyvee_u})$, given as follows. Taking the pullback under $q\colon\wt{\Delta}_{Dax}\to\Delta_{Dax}$ we have
\begin{align}
    q^*\nu_{\Delta_{Dax}} \;\cong\; \nu_{\wt{\Delta}_{Dax}}
    & \;\cong\;
    \nu_{(\I^n \times  V)^2}|_{\wt{\Delta}_{Dax}} \oplus \nu_{\wt{\Delta}_{Dax}\subseteq (\I^n \times  V)^2}
    \;\cong_s\;
    \nu_V^2|_{\wt{\Delta}_{Dax}}\oplus (\wt{F}^2|)^*(\nu_{\Delta_{\I^n \times X}\subseteq (\I^n \times X)^2}).\label{eq:normal-bdle}
\end{align}
On the other hand, $q^*b_{Dax}^*(\vartheta_u|_{E^\curlyvee_u})\cong \wt{b}_{Dax}^*(\wt{\vartheta}_u)$ for the obvious double cover $\wt{b}_{Dax}\colon \wt{\Delta}_{Dax}\to\wt{E}^\curlyvee_u$ of $b_{Dax}$. Since $pr_X\circ\wt{b}_{Dax}(\vec{t},v,w)=F(\vec{t})(v)=pr_1\circ \wt{F}^2|(\vec{t},v,w)$ and $T(\I^n\times X)\cong\delta^*(\nu_{\Delta\subseteq(\I^n\times X)^2})$, we have
\begin{equation}\label{eq:other-bdle}
    \wt{b}_{Dax}^*(pr_X^*(TX)) \cong (pr_1\circ \wt{F}^2|)^*(TX)
    \,\cong_s\, (\wt{F}^2|)^*(pr_1^*(T(\I^n\times X))
    \,\cong\,
    (\wt{F}^2|)^*(\nu_{\Delta\subseteq(\I^n\times X)^2}),
\end{equation}
Since $\wt{\vartheta}_u\coloneqq pr_{V^2}^*\nu_V^2\oplus pr_X^*(TX)$ and $\wt{b}_{Dax}^*(pr_{V^2}^*\nu_V^2)\cong\nu_V^2|_{\wt{\Delta}_{Dax}}$, we have a stable isomorphism $\wt{B}_{Dax}$ of $q^*b_{Dax}^*(\vartheta_u|_{E^\curlyvee_u})$ and $q^*\nu_{\Delta_{Dax}}$. Moreover, this respects the involutions, so gives the desired $B_{Dax}$. Thus, we have defined a bordism class $(\Delta_{Dax},b_{Dax},B_{Dax})\in  \Omega_{n-(d-2\ell)}(E^\curlyvee_u;\, \vartheta_u|_{E^\curlyvee_u})$.

\subsubsection{The Dax bordism group for a simply connected source}

\begin{prop}\label{prop:Dax-bordism-V-1-conn}
    Let $V,X,u$ be as in the first sentence of Dax's Theorem~\ref{thm:Dax}. Additionally assume that $V$ and $X$ are oriented, and that $V$ is $1$-connected. Then
\[
    \Omega_0(E^\curlyvee_u;\vartheta_u|_{E^\curlyvee_u})\cong\faktor{\Z[\pi_1X]}{\relations_{\ell,d}}
    \quad\text{ with}\quad \relations_{\ell,d}=
    \begin{cases}
        0, & \text{if } \ell=1,\\
        \langle g-(-1)^{d-\ell}g^{-1}: g\in\pi_1X\rangle, & \text{otherwise.}
    \end{cases}
\]
\end{prop}
This identified the $0$-th bordism group with the target of Wall's self-intersection invariants. 
\begin{proof}
    It is a standard fact that $\Omega_0(E^\curlyvee_u;\vartheta_u|_{E^\curlyvee_u})\cong H_0(E^\curlyvee_u;\Z(\vartheta_u|_{E^\curlyvee_u}))$, where $\Z(\vartheta_u|_{E^\curlyvee_u})$ are the local coefficients induced by the orientation of the bundle, so over the connected component $c$ of $E^\curlyvee_u$ the coefficient group is $\Z$ or $\Z/2$ depending on whether $\vartheta_u|_c$ is orientable or not.
    
    To compute the set $\pi_0E^\curlyvee_u$ of components, we first find $\pi_i(V^2\sm\Delta_V)_{\Z/2}$ for $i=1,2$, using the exact sequence
\[\begin{tikzcd}
    \pi_1(V^2\sm\Delta_V)\rar[tail] & \pi_1(V^2\sm\Delta_V)_{\Z/2}\rar & \Z/2\rar & \pi_0(V^2\sm\Delta_V)\rar[two heads] & \pi_0(V^2\sm\Delta_V)_{\Z/2}.
\end{tikzcd}
\]
    For $\ell\geq2$ we have $\pi_0(V^2\sm\Delta_V)_{\Z/2}=\pi_0(V^2\sm\Delta_V)=1$.
    %whereas $(\D^1)^2\sm\Delta_{\D^1}$ has two components that get identified by the involutions $\Z/2\cong\pi_0((\D^1)^2\sm\Delta_{\D^1})$, so again $\pi_0(V^2\sm\Delta_Vx)_{\Z/2}=1$.
    Since $V$ is $1$-connected, from the fibration sequence $V\sm pt\to V^2\sm\Delta_V\to V$ we see that $\pi_1(V^2\sm\Delta_V)=1$, unless $\ell=1,2$. For $\ell=1,2$, the only examples are $\D^1$ and $\D^2$, for which $(\D^1)^2\sm\Delta_{\D^1}$ consists of two triangles that are permuted by the $\Z/2$, so the quotient is a single triangle. 
    For $\ell=2$, $\pi_1((\D^2)^2\sm\Delta_{\D^2})$ is the pure braid group on two strands, which is infinite cyclic and generated by a full twist.
    Similarly,  $\pi_1((\D^2)^2\sm\Delta_{\D^2})_{\Z/2}$ is the braid group on two strands which is again infinite cyclic, generated by a half twist.
    
    Next, we claim that there is a fibration sequence
\begin{equation}\label{eq:E_u-fibration}
\begin{tikzcd}
        \Omega X\rar & E^\curlyvee_u\rar{pr_{V^2}} &  (V^2\sm\Delta_V)_{\Z/2}.
\end{tikzcd}
\end{equation}
    Indeed, the fiber over $[v,w]\in(V^2\sm\Delta_V)_{\Z/2}$ consists of paths $\rho$ from $u(v)$ to $u(w)$ (or equivalently from $u(w)$ to $u(v)$), so taking $\gamma\in\Omega X$ to the path $\rho=u(\phi_v)\cdot\gamma\cdot u(\phi_w)^{-1}$ gives a homotopy equivalence, for some fixed whiskers $\phi_v$ from $v$ to $e_V\in V$ (alternatively, restrict Dax's fibration $pr_W\colon\C_u\to W$ to $\mathrm{int}\,W=(V^2\sm\Delta_V)_{\Z/2}$).
    The bottom of the induced long exact sequence of homotopy groups is:
\[\begin{tikzcd}
    \pi_1(E^\curlyvee_u,c)\rar & \{1\}\text{ or }\Z\text{ or }\Z/2 \rar & \pi_1X\rar[two heads] & \pi_0E^\curlyvee_u,
\end{tikzcd}
\]
    for $\ell=1$ or $\ell=2$ or $\ell\geq3$ respectively, and a component $c\in\pi_0E^\curlyvee_u$. A generator $\sigma$ of $\Z/2$ or $\Z$ sends $g=[\gamma]\in\pi_1X$ to $g^{-1}$: indeed, $\sigma$ is represented by the loop of two points in $V$ switching positions, so if  the lifted path starts at $[v,w,u(\phi_v)\cdot\gamma\cdot u(\phi_w)^{-1}]\in E^\curlyvee_u$ then it ends at $[w,v,u(\phi_w)\cdot\gamma\cdot u(\phi_v)^{-1}]=[v,w,u(\phi_v)\cdot\gamma^{-1}\cdot u(\phi_w)^{-1}]$. Thus,
\[
    \pi_0E^\curlyvee_u\cong\begin{cases}
        \pi_1X, & \ell=1\\
        \faktor{\pi_1X}{g\sim g^{-1}}, & \ell\geq2.
    \end{cases}
\]
    Moreover, $\sigma$ fixes an element $g\in\pi_1X$ if and only if $g=g^{-1}$. Thus, $\pi_1(E^\curlyvee_u,c)\to \Z/2\text{ or }\Z$ is surjective if and only if $c$ corresponds to an order $2$ element $g$ (i.e.\ $c=[v,w,u(\phi_v)\cdot\gamma\cdot u(\phi_w)^{-1}]$ for $[\gamma]=g$).
    
    Now, we claim that for $c=[g]\in\pi_0E^\curlyvee_u$ the bundle $\vartheta_u|_c$ is nonorientable if and only if $\ell\geq2$, $d-\ell$ is odd and $g^2=1$. Indeed, recall that $\vartheta_u|_{E^\curlyvee_u}$ is the quotient of $pr_{V^2}^*(\nu_V)\oplus pr_X^*(TX)\cong pr_{V^2}^*(\nu_u^2)\oplus pr_X^*(\nu_X)$ (using that $\nu_V\cong\nu_u\oplus\nu_X$ and $\nu_X\oplus TX$ is trivial) by the involution swapping the two $V$ factors. As $X$ is orientable, it suffices to check if there is a loop in $(V^2\sm\Delta_V)_{\Z/2}$ that is orientation-reversing for the quotient of $\nu_u^2$, and lifts to $E^\curlyvee_u$. We saw $\pi_1(V^2\sm\Delta_V)_{\Z/2}$ is $\Z$ or $\Z/2$, and along the generating loop $\sigma$ the monodromy is $(-1)^{d-\ell}$, since $(d-\ell)$ is the rank of $\nu_u$. By the previous paragraph, $\sigma$ lifts to $\pi_1(E^\curlyvee_u,[g])$ if and only if $g^2=1$.
        Therefore, $\Omega_0(E^\curlyvee_u;\vartheta_u|_{E^\curlyvee_u})$ is canonically isomorphic to
\[\begin{cases}
        \Z[\pi], & \ell=1,\\[5pt]
        \Z\left[\faktor{\pi}{g\sim g^{-1}}\right]\cong\faktor{\Z[\pi]}{\langle g=g^{-1}\rangle}, & \ell\geq2,d-\ell\;\text{even}, \\[5pt]
        \Z\left[\faktor{\{g\in\pi:g^2\neq 1\}}{g\sim g^{-1}}\right]\oplus\Z/2\left[\{g\in\pi:g^2= 1\}\right]\cong\faktor{\Z[\pi]}{\langle g=-g^{-1}\rangle} , & \ell\geq2,d-\ell\;\text{odd},
\end{cases}
\]
    where $\pi\coloneqq\pi_1X$. This was exactly denoted ${\Z[\pi_1X]}/{\relations_{\ell,d}}$ in the statement, and finishes its proof.
\end{proof}

%%%%%%%%%%%%%%%%
\subsubsection{The Dax invariant  for a simply connected source}\label{subsec:arcs}
We next simplify the Dax isomorphism $\beta'_n$ from \eqref{eq:beta_n'} for $n=d-2\ell$ using the description $\Omega_{n-(d-2\ell)=0}(\wt{E}^\curlyvee_u;\vartheta_u|_{E^\curlyvee_u})\cong\Z[\pi_1X]/\relations_{\ell,d}$ from the previous section. First note that for this to fall into the metastable range we need to have $0\leq d-2\ell\leq 2d-3l-3$, which says $d\geq \ell+3$, and $d-2\ell\geq0$. For $\ell=1$ and $d=4$ the following interpretation of Dax's work was also studied and used by Gabai in~\cite{Gabai-disks} (his spinning map is analogous to our $\partial\realmap$, see Remark~\ref{rem:Gabai}).

\begin{theorem}\label{thm:Dax-invt}
    Let $V,X,u$ be as in the first sentence of Dax's Theorem~\ref{thm:Dax}. Additionally assume that $V$ and $X$ are oriented, that $V$ is $1$-connected, and $d\geq \ell+3$ and $d-2\ell\geq0$. 
    Under the isomorphism of Proposition~\ref{prop:Dax-bordism-V-1-conn}, $\beta'_{d-2\ell}$ is equivalent to the isomorphism
\begin{align}
        & \Da\colon\;
        \pi_{d-2\ell}\big(\Imm_\partial(V,X),\Emb_\partial(V,X);u\big) \overset{\cong}{\ra} \faktor{\Z[\pi_1X]}{\relations_{\ell,d}},\nonumber\\
        & \Da[F]\coloneqq\sum_{i=1}^k\e_{(\vec{t}_i, x_i)}g_{x_i},\label{eq:Dax-def}
\end{align}
    which sends a perfect map $F$ to the sum over all double points $(\vec{t}_i, x_i)$ of its track $\wt{F}$ from \eqref{eq:track}, of the associated signed loops $\e_{(\vec{t}_i, x_i)}g_{x_i}$, where $\e_{(\vec{t}_i, x_i)}\in\{\pm1\}$ and $g_{x_i}\in\pi_1X$ are defined below.
\end{theorem}
% As usual, represent a relative homotopy class by a map $F\colon \I^{d-2}=\I^{d-3}\times\I\to \Imm_\partial(\D^1,X)$ satisfying
% \begin{equation}\label{eq:rel-class}
%     F\big( \partial\I^{d-3}\times\I \cup  \I^{d-3}\times\{1\} \big) =u
%     \quad\text{ and }\quad  F\big( \I^{d-3}\times\{0\} \big)\subseteq\Emb_\partial(\D^1,X).
% \end{equation}
% \begin{lemma}\label{lem:Dax}
%   After a small perturbation of $F$ preserving boundary conditions the associated map
%   \[\wt{F}\colon\;\I^{d-2} \times  \D^1 \ra \I^{d-2} \times X,\quad (\vec{t},\theta)
%     \mapsto (\,\vec{t},\, F(\vec{t})(\theta)\,)
%   \]
%   is an immersion with only isolated transverse double points.
% \end{lemma}
% Firstly, we pick a strict partial order on the space $\D^\ell\times\D^\ell\sm\Delta_{\D^L}$ as follows: we let $v<w$ if and only if $pr_1v<pr_1w$ or $pr_1v=pr_1w$ and $pr_2v<pr_2w$ or etc.\ or $pr_iv=pr_iw$ for $i\leq \ell-1$ and $pr_\ellv<pr_\ellw$.

Firstly, for any $v\in V$ fix a whisker $\phi_v\colon[0,1]\to V$ from the basepoint $e_\ell\in\partial V$ to $v$. For example, for $V=\D^\ell$ take the straight line $\phi_v(s)=sv+(1-s)e_\ell$ from $e_\ell=(-1,0,\dots,0)$ to $v\in\partial V$. For $\ell=1$ we simply have $e_\ell=-1$ and $\phi_v=[-1,v]$.

Secondly, the map $\wt{F}$ has finitely many double points, all of the form $(\vec{t}_i, x_i)$ with $1\leq i\leq k$, for some $\vec{t}_i\in\I^{d-2}$ and $x_i\coloneqq F(\vec{t}_i)(v_i)=F(\vec{t}_i)(w_i)\in X$ with $v_i,w_i\in V$. \textbf{Let us pick an order $(v_i,w_i)$.} In other words, for the immersed manifold $F(\vec{t}_i)$ we choose ``an order of the sheets'' at the double point $x_i$.

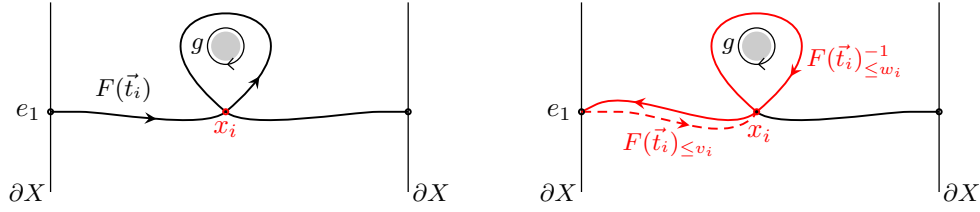
\begin{figure}[!htbp]
    \centering
        \begin{tikzpicture}[scale=0.8,every node/.style={scale=0.85},even odd rule]
        \clip (-0.7,-1.3) rectangle (5.55,1.6);
        \draw 
            (0,1.5) -- (0,-1.1) node[left=-2pt]{\small$\partial X$}  
            (4.9,1.5) -- (4.9,-1.1)  node[right=-2pt]{\small$\partial X$};
        %left
        \draw[black,semithick,postaction=decorate,decoration = {markings, mark = at position 0.6 with {\arrow{stealth}}}]
            (0,0) -- 
            (0.4,0) to[out=0, in=-140, distance=0.4cm]  (2.4,0)
            (1,0) node[above]{\small$F(\vec{t}_i)$};
        %right
        \draw[black,semithick]
            (2.4,0) to[out=-40,in=180, distance=0.4cm] 
            (4.5,0) -- (4.9,0)
            ;
        %black loop    
        \draw[black,semithick,postaction=decorate,decoration = {markings, mark = at position 0.17 with {\arrow{stealth}}}]
            (2.4,0)
            to[out=40,in=140, distance=2.8cm]
            (2.4,0) ;
        \draw[red,semithick]
            (2.4,0) circle (1pt) node[below]{$x_i$};
        \draw[black,semithick]
            (0,0) circle (1pt) node[left]{$e_1$} 
            (4.9,0)  circle (1pt) node[right]{};
            
        % the grey hole
        \fill[color=black!20!white] 
            (2.4,0.9) circle (0.2cm);
        \draw[black,postaction = decorate,decoration = {markings, mark = at position 0.76 with {\arrowreversed{>}} }] 
            (2.4,0.9) circle (0.25cm) node[left=0.155cm]{\small$g$};
    \end{tikzpicture}~\quad\quad
       \begin{tikzpicture}[scale=0.8,every node/.style={scale=0.85},even odd rule]
        \clip (-0.7,-1.3) rectangle (5.55,1.6);
        \draw 
            (0,1.5) -- (0,-1.1) node[left=-2pt]{\small$\partial X$}  
            (4.9,1.5) -- (4.9,-1.1)  node[right=-2pt]{\small$\partial X$};
        %left
        \draw[red,densely dashed,semithick,postaction=decorate,decoration = {markings, mark = at position 0.6 with {\arrow{stealth}}}]
            (0,0) -- 
            (0.5,0)
            to[out=0, in=-140, distance=0.6cm]  (2.35,-0.08)      --(2.4,0)
            (1.2,-0.08) node[below]{\small$F(\vec{t}_i)_{\leq v_i}$};
        %left again
        \draw[red,semithick,postaction=decorate,decoration = {markings, mark = at position 0.3 with {\arrowreversed{stealth}}}]
            (0,0) -- (0.2,0.1) to[out=25, in=-140, distance=0.6cm]  (2.4,0)
            (3.75,0.7) node[]{\small$F(\vec{t}_i)_{\leq w_i}^{-1}$};
        %right
        \draw[black,semithick]
            (2.4,0) to[out=-40,in=180, distance=0.4cm] 
            (4.5,0) -- (4.9,0);
        %black loop    
        \draw[red,semithick,postaction=decorate,decoration = {markings, mark = at position 0.17 with {\arrowreversed{stealth}}}]
            (2.4,0) 
            to[out=40,in=140, distance=2.8cm]
            (2.4,0) ;
        \draw[black,semithick]
            (0,0) circle (1pt) node[left]{$e_1$} 
            (4.9,0)  circle (1pt) node[right]{}
            ;
        \draw[red,semithick,fill=black]
            (2.4,0) circle (1pt) (2.44,-0.3) node[]{$x_i$};
            
        % the grey hole
        \fill[color=black!20!white] 
            (2.4,0.9) circle (0.2cm);
        \draw[black,postaction = decorate,decoration = {markings, mark = at position 0.76 with {\arrowreversed{>}} }] 
            (2.4,0.9) circle (0.25cm) node[left=0.155cm]{\small$g$};
    \end{tikzpicture}
    
% \begin{tikzpicture}[scale=0.65, every node/.style={scale=1},even odd rule]
%     \clip (-0.75,-1) rectangle (12.5,3);
%     \draw
%         (0,0) circle (1.1pt) node[below]{\small$u(-1)$}
%         (12,0) circle (1.1pt) node[below]{\small$u(1)$} ;
%     \draw
%         (0,0) 
%         to[out=0,in=160, distance=2cm]
%         (6,2) circle (1.1pt) node[above]{$x_i$}
%         to[out=-30,in=-150, distance=6cm]
%         (6,2) 
%         to[out=30,in=180, distance=2cm]
%         (12,0) ;
%     \draw[cyan, line width=5pt, draw opacity=0.15, text opacity=1, postaction=decorate, decoration =     {markings, 
%         mark = at position 0.5 with {\arrow[line width=2pt]{>}}
%         }]
%                 (0,0)  to[out=0,in=160, distance=2cm]  (6,2);
%     \draw[pink, line width=2.5pt, draw opacity=0.6, text opacity=1, postaction=decorate, decoration =     {markings, 
%         mark = at position 0.25 with {\arrowreversed[line width=2pt]{>}},
%         mark = at position 0.7 with {\arrow[line width=2pt]{>}} 
%         }]
%                 (6,2) to[out=-30,in=-150, distance=6cm] (6,2) 
%                 to[out=160,in=0, distance=2cm] (0,0) ;
%     % the grey hole
%     \fill[color=black!15!white] (6,0.9) circle (0.55cm);
%     \draw[black,postaction = decorate,decoration = {markings, mark = at position 0.55 with {\arrowreversed{latex}} }] (6,0.9) circle (0.65cm) node[]{$g$};
% \end{tikzpicture}
    \caption{The double point $x_i\in X$ of the arc $F(\vec{t}_i)\in\Imm_\partial(\D^1,X)$ has the associated loop $g_{x_i}=g$.}
    \label{fig:double-point}
\end{figure} 
Let $\e_{(\vec{t}_i, x_i)}\in \{\pm 1\}$ be the relative orientation at $(\vec{t}_i, x_i)$, obtained by comparing orientations of the tangent space $T_{(\vec{t}_i, x_i)}(\I^{d-2\ell}\times X)$ and (in this order):
\begin{equation}\label{eq:our-Dax-sign}
    d\wt{F} \big( T_{(\vec{t}_i,v_i)}( \I^{d-2\ell}\times V ) \big)
    \oplus d\wt{F} \big( T_{(\vec{t}_i,w_i)}( \I^{d-2\ell}\times V ) \big).
\end{equation}
% \[\deriv\wt{F}T_{(\vec{t}_i,\theta_i^-)}( \I^{d-2}\oplus\D^1 ) \oplus \deriv\wt{F} [T_{(\vec{t}_i,\theta_i^+)}( \I^{d-2}\oplus\D^1 )].
% or (T_{\theta_i^-}\D^1 \oplus T_{\vec{t}_i}\I^{d-2}) \oplus (T_{\theta_i^+}\D^1 \oplus T_{\vec{t}_i}\I^{d-2}),
% \]
Again using the fact that we chose an order of sheets at $x_i$, we define the group element $g_{x_i}\in\pi_1(X,u(e_\ell))$ to be represented by the following loop based at $u(e_\ell)$ (see Figure~\ref{fig:double-point} for $\ell=1$):
\begin{equation}\label{eq:our-Dax-dp-loop}
  \gamma_{x_i}\coloneqq F(\vec{t}_i)(\phi_{v_i})\cdot F(\vec{t}_i)(\phi_{w_i})^{-1}
\end{equation}
\emph{Now note that if $\ell=1$ the order $v_i<w_i$ is canonical in the interval $\D^1$.} On the other hand, if $\ell\geq2$ and we switch the order then the sign \eqref{eq:our-Dax-sign} changes by $(-1)^{d-\ell}$ and the loop \eqref{eq:our-Dax-dp-loop} becomes its inverse. But since $g_{x_i}-(-1)^{d-\ell}g_{x_i}^{-1}\in \relations_{\ell,d}$ the class $\Da(F)$ is well defined.
\begin{proof}[Proof of Theorem~\ref{thm:Dax-invt}]
    Since $\Delta_{Dax}$ is $0$-dimensional, the bordism class of $b_{Dax}\colon\Delta_{Dax}\to E^\curlyvee_u$ is by definition the sum of signed components of $E^\curlyvee_u$ containing $b_{Dax}(\Delta_{Dax})$, with the sign $\e_{Dax}([v,w],\vec{t})=+1$ if and only if $B_{Dax}$ is an orientation-preserving isomorphism $\nu_{\Delta_{Dax}}\cong b_{Dax}^*(\vartheta_u|_{E^\curlyvee_u})$. Equivalently, this is the sum of signed components of $\wt{E}^\curlyvee_u$ containing $\wt{b}_{Dax}(\wt{\Delta}_{Dax})$, modulo the involution. For $(\vec{t},v,w)\in\wt{\Delta}_{Dax}$ the component of $\wt{b}_{Dax}(\vec{t},v,w)=(v,w,\rho)$ corresponds to the class $[u(\phi_v)\cdot \rho\cdot u(\phi_w^{-1})]\in\pi_1X$ (see \eqref{eq:E_u-fibration}), with $\rho$ as in \eqref{eq:b-Dax}. We claim that this loop in $X$ is based homotopic to the loop $\gamma_{x_i}$ from \eqref{eq:our-Dax-dp-loop}. Indeed, $\I^{d-2\ell}\times V$ is simply connected, so the choice of whiskers is irrelevant: Dax chooses to go in a direction tangent to $\I^{d-2\ell}$, while we opt for a direction tangent to $V$.
    
    Finally, the sign of this component $\wt{b}_{Dax}(\vec{t},v,w)$ is positive if and only if $\wt{B}_{Dax}|_{(\vec{t},v,w)}$ preserves orientations. Tracing through isomorphisms \eqref{eq:normal-bdle} and \eqref{eq:other-bdle} we see that the source of the sign is the isomorphism $\nu_{\Delta_{Dax}\subseteq(\I^{d-2\ell}\times V)^2}\cong(\wt{F}^2|)^*(\nu_{\Delta\subseteq(\I^{d-2\ell}\times X)^2})$. Equivalently, $\e_{Dax}([v,w],\vec{t})=+1$ if and only if $d(\wt{F}^2|)$ is orientation preserving at $(v,w,\vec{t})$ if and only if the orientation of $\nu_{\Delta\subseteq(\I^{d-2\ell}\times X)^2}$ agrees with that of $d\wt{F}^2|_{(v,w,\vec{t})}T(\I^{d-2\ell}\times V)^2 = (d\wt{F}|_{(\vec{t},v)}(T(\I^{d-2\ell}\times V), d\wt{F}|_{(\vec{t},w)}(T(\I^{d-2\ell}\times V))$. This is precisely our definition of the sign $\e_{\vec{t},x}$ in~\eqref{eq:our-Dax-sign}.
\end{proof}

\begin{remark}\label{rem:Grant}
    For $d-2\ell=0$ this should be compared to Grant's result~\cite{Grant} that the Wall invariant
\[
    \mu_\ell\colon
    \faktor{\pi_0\Imm_\partial(V,X)}{\pi_0\Emb_\partial(V,X)}\ra
    \faktor{\Z[\pi_1X]}{\relations_{\ell,d}}
\]
    agrees with the Hatcher--Quinn invariant \cite{Hatcher-Quinn} for $V$ simply connected (the latter is defined for any $V,X$ with $2d-3\ell-3\geq0$). On one hand, $\mu_\ell$ is defined as the count of signed double point loops, so it clearly agrees with $\Da$. On the other hand, one can check that $\beta_0'$ agrees with the Hatcher--Quinn invariant (in fact for any $V,X$ with $2d-3\ell-3\geq0$).
    % We could give more details for this claim.
\end{remark}
\begin{example}\label{ex:computing-md}
    Let us compute the Dax invariant of the following class. Assume $\U\colon\D^k\hra M$ has a boundary dual $G\colon\S^{d-k}\hra\partial M$ and pick $g\in\pi_1M$. Push $G$ into the interior of $M$ and foliate it by a $(d-2k)$-family of $k$-spheres (see Figure~\ref{fig:realmap}), then drag a piece of $\U$ around $g$ and connect sum it into each of those $k$-spheres. This defines a class in $\pi_{d-2k}(\Imm_\partial(\D^k,X),\U)$, which clearly lifts $g\cdot[G]\in\pi_{d-k}M$. When considered as a relative class $\pi_{d-2k}(\Imm_\partial(\D^k,M),\Emb_\partial(\D^k,M),\U)$ we can compute its $\Da$ invariant: we see a single double point, namely $\U\cap G'=\{pt\}$, with the signed double point loop $g$.
\end{example}

\subsubsection{The realization map}

\begin{theorem}\label{thm:realization-map}
    Let $V,X,u$ be as in the first sentence of Dax's Theorem~\ref{thm:Dax}. Additionally assume that $V$ and $X$ are oriented.
    There is an explicit \emph{realization map}
\[
    \realmap\colon\faktor{\Z[\pi_1X]}{\relations_{\ell,d}}\ra\pi_{d-2\ell}(\Imm_\partial(V,X),\Emb_\partial(V,X),u).
\]
    If $V$ is $1$-connected, and $d\geq \ell+3$ and $d-2\ell\geq0$, then $\realmap$ is the inverse of $\Da$ given in Theorem~\ref{thm:Dax-invt}. 
\end{theorem}
\begin{proof}
For $g\in\pi_1X$ we define
\[
    \realmap(g)\colon(\I^{d-2\ell-1}\times\I,\, \I^{d-2\ell-1}\times\{0\},\, \partial\I^{d-2\ell-1}\times\I \cup \I^{d-2\ell-1}\times\{1\})\ra (\Imm_\partial(V,X),\,\Emb_\partial(V,X),\,u).
\]
Firstly, define for $\vec{t}\in\I^{d-2\ell-1}\times\{0\}$ the embedded arcs $\partial\realmap(g)(\vec{t})\coloneqq\realmap(g)_{\vec{t}}$ by dragging a piece of $u$ near $u(e_\ell)$ along the group element $g$, then ``swing a lasso'' around a meridian $\mu(\S^{d-\ell-1})$ at a point $\mathbf{x}\in u$ near $u(e_\ell)$, then drag back to $u$. More precisely, we foliate $\S^{d-\ell-1}$ by a $(d-2\ell-1)$-family of $\ell$-disks $\alpha_{\vec{t}}\colon\D^\ell\hra \S^{d-\ell-1}$ based at two fixed points, see Figure~\ref{fig:realmap}. Use the pinch map $\I^{d-2\ell-1}\to\I\vee\S^{d-2\ell-1}$ and along $\I$ apply the finger move to a small disk in $u(V)$ following $g$, ending with the connect sum of $u$ with the $\ell$-disk $\mu(\alpha_N)$. For $\vec{t}\in\S^{d-2\ell-1}$ connect sum with the disk $\mu(\alpha_{\vec{t}})$ instead.

For $(\vec{t},s)\in\I^{d-2\ell-1}\times\I$ the paths of immersions $\realmap(g)_{\vec{t},s}\colon V\imra X$ from $\realmap(g)_{\vec{t},0}=\realmap(g)_{\vec{t}}$ back to $\realmap(g)_{\vec{t},1}=u$ are defined by similarly foliating by $\ell$-disks the meridian ball $\ol{\mu}(\ball^{d-\ell})$ bounded by $\mu(\S^{d-\ell-1})$.

Define $\realmap(-g)$ analogously, but connecting into the meridian $\mu(\S^{d-\ell-1})$ from ``below''. Then extend $\realmap$ to $\Z[\pi_1X]$ linearly: for $d\geq4$ the target is an abelian group, but for $\ell=1,d=4$ we need to check that $\realmap(g)$ and $\realmap(h)$ commute. Namely, they can be constructed using \emph{disjoint} supports $J_i$ and different meridian balls $\mu_i(\ol{\ball^3})$, so there is a null homotopy $\I\times\I^2\to\Imm_\partial(\D^1,X)$ of their commutator, given at $(t_0,t_1,t_2)$ by applying the map $\realmap(g)(t_0,t_1)$ on $J_1$ and $\realmap(h)(t_0,t_2)$ on $J_2$, and $u$ otherwise.
\begin{figure}[!htbp]
    \centering
        \begin{tikzpicture}[scale=0.7,every node/.style={scale=1},even odd rule]
        \clip (0.25,-1.5) rectangle (11.45,6);
        % the leftmost
        \draw[red!40!yellow, very thick]
                (7.6,1.8) to[out=-140,in=-130, distance=1.8cm] (8.1,2);
        %the second from left
        \draw[red!60!yellow, very thick]
                (7.6,1.8) to[out=-120,in=-110, distance=3cm] (8.1,2);
        \fill[white] (8,0) circle (1cm);
        
        \draw[thick] 
            (1,0) circle (1.1pt) node[below]{\small$u(-1)$} -- (2,0)
            (4,0) -- (11,0) circle (1.1pt) node[below]{\small$u(1)$}
            (6.1,-0.7) node{\color{black!80!white}$\mu(\S^{d-l-1})$};
        \draw[cyan,very thick]
            (8, 0) node[below=1pt]{$\mathbf{x}$} circle (1.5pt);

        % sphere
        \foreach \y in {8}{
            \shade[ball color=blue!5!white,opacity=0.55] (\y,0) circle (1cm);
            \draw[black!60!yellow] (\y-1,0) arc (180:360:1cm and 0.5cm);
            \draw[black!60!yellow,dotted] (\y-1,0) arc (180:0:1cm and 0.5cm);
            \draw[black!60!yellow,semithick] (\y,0) circle (1cm);
            }
        % the third from left
        \draw[red!70!yellow, very thick]
                        (7.6,1.8) to[out=-110,in=-100, distance=1.15cm] (7.6,-1.06)
                        (8,1.06) to[out=40,in=-120, distance=0.07cm] (8.1,2);
        % middle
        \draw[black, line width=1pt]
                (7.6,1.8) to[out=-90,in=-100, distance=1.15cm] (8,-1.06)
                (8.2,1.06) to[out=40,in=-120, distance=0.07cm] (8.1,2);
        % gp element
        \draw[, thick, postaction=decorate,decoration = {markings, mark = at position 0.1 with {\arrow{latex}} , mark = at position 0.9 with {\arrowreversed{latex}}}]
                        (2,0) to[out=0,in=-100, distance=0.5cm]
                        (3,3.15) to[out=80,in=98, distance=4cm] (8.1,2)
                        (4,0) to[out=180,in=-100, distance=0.5cm]
                        (4,3) to[out=80,in=94, distance=3cm] (7.6,1.8)
                        (8.2,4) node{\textbf{$\mathfrak{r}(g)_t$}};
        % the grey hole
        \fill[color=black!20!white] (5.6,3) circle (0.55cm);
        \draw[black,postaction = decorate,decoration = {markings, mark = at position 0.55 with {\arrowreversed{latex}} }] (5.6,3) circle (0.65cm) node[left=0.39cm]{$g$};
        % the third from right
        \draw[purple!80!blue, very thick,postaction=decorate,decoration = {markings, mark = at position 0.49 with {\arrowreversed{latex}}}]
                (7.6,1.8) to[out=-80,in=-80, distance=1.4cm] (8.8,-0.8)
                (8.5,0.95) to[out=100,in=-80, distance=0.1cm] (8.1,2)
        ;
        % the second from right
        \draw[purple!70!blue, very thick]
              (7.6,1.8) to[out=-55,in=-55, distance=3.5cm] (8.1,2);
        % the rightmost
        \draw[purple!60!blue, very thick]
                (7.6,1.8)to[out=-40,in=-40, distance=2.2cm] (8.1,2);
        \draw[myred, very thick,-latex]
                (8.1,2) -- (7.6,1.8);
        \draw[myred, very thick]
                (9,2.15) node{\textbf{$\mu(\alpha_N)$}}
                (9.55,-1) node[purple!70!blue]{\textbf{$\mu(\alpha_t)$}}
                ;
    \end{tikzpicture}~\quad\quad\begin{tikzpicture}[scale=0.7,every node/.style={scale=1},even odd rule]
    \clip (0.25,-1.5) rectangle (11.45,6);
    
        \draw[thick]
            (1,0) circle (1.1pt) node[below]{\small$u(-1)$} -- (2,0)
            (4,0) -- (7.2,0) (7.65,0) -- (11,0) circle (1.1pt) node[below]{\small$u(1)$};
            
        % tip
        \draw[thick]
                (7.5,1.8) to[out=-95,in=-120, distance=1.25cm] (8,0)
                 to[out=60,in=-80, distance=0.7cm] (8.1,2);
        % gp element
        \draw[thick]
                        (2,0) to[out=0,in=-100, distance=0.5cm]
                        (3,3.15) to[out=80,in=98, distance=4cm] (8.1,2)
                        (4,0) to[out=180,in=-100, distance=0.5cm]
                        (4,3) to[out=80,in=85, distance=2.8cm] (7.5,1.8);
                        
        % the grey hole
        \fill[color=black!20!white] (5.6,3) circle (0.55cm);
        \draw[postaction = decorate,decoration = {markings, mark = at position 0.55 with {\arrowreversed{latex}} }] (5.6,3) circle (0.65cm) node[left=0.39cm]{$g$};
        
        % whiskers
        \draw[red, thick, postaction=decorate, decoration = {markings, 
        mark = at position 0.08 with {\arrow{>}},
        mark = at position 0.23 with {\arrow{>}}, 
        mark = at position 0.42 with {\arrow{>}}, 
        mark = at position 0.7 with {\arrow{>}} 
        }]
                        (8,0.08) -- (7.65,0.08) 
                        (7.2,0.08) -- (4.1,0.08) to[out=180,in=-100, distance=0.5cm]
                        (4.1,2.9) to[out=80,in=85, distance=2.8cm] (7.35,1.7) 
                        to[out=-95,in=-120, distance=1.25cm] (8.05,-0.2)
                        to[out=60,in=-80, distance=0.8cm] (8.25,2)
                        to[in=80,out=98, distance=4cm] (2.9,3.35)
                        to[in=0,out=-100, distance=0.5cm]
                        (1.9,0.08) -- (1,0.08)
                        ;
        \draw[red, thick, dashed, postaction=decorate, decoration = {markings, 
        mark = at position 0.3 with {\arrow{>}},
        mark = at position 0.808 with {\arrow{>}}
        }]
                        (1,-0.08)--(2.1,-0.08) to[out=0,in=-100, distance=0.5cm]
                        (3.1,3.05) to[out=80,in=98, distance=3.9cm] (8,2)
                        to[out=-80,in=60, distance=0.7cm] (7.9,0)
        ;
        \draw[cyan,very thick]
            (8, 0) node[below right]{$\mathbf{x}$} circle (1.5pt);
\end{tikzpicture}
    \caption{
        \emph{Left.} The samples $\realmap(g)_t\in\Emb_\partial(\D^\ell,X)$ for several $t\in \I^{1}$, $\ell=1,d=4$, with $\realmap(g)_0=\realmap(g)_1=u$ as the horizontal arc. \emph{Right.} The double point loop $g_{\mathbf{x}}=g$ is the dashed arc followed by the solid red arc.}
    \label{fig:realmap}
\end{figure}
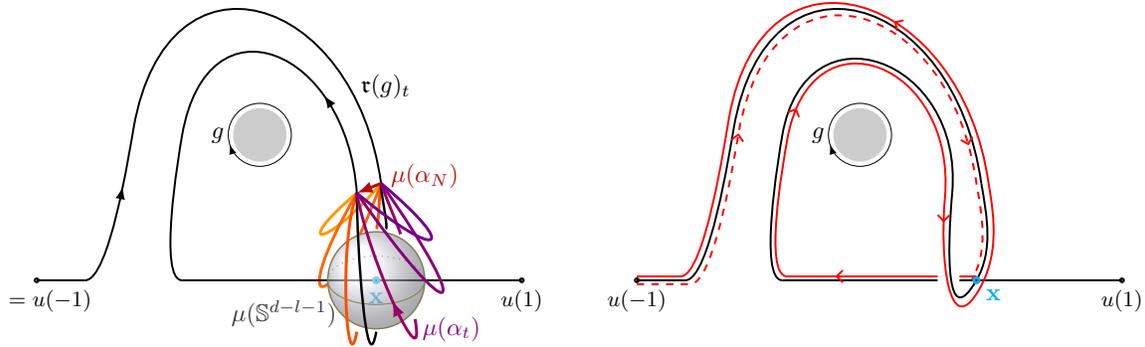

Let us show that $\Da\circ\realmap(g)=g$. In the family $\realmap(g)_{\vec{t},s}$ all disks are embedded except one, for which there is exactly one double point $\{\mathbf{x}\}=u\cap\ol{\mu}(\ball^{d-\ell})$ and the associated loop is precisely $g_{\mathbf{x}}=g\in\pi_1X$, see the right part of Figure~\ref{fig:realmap}. To determine the sign choose coordinates $\R^d=\R^{2\ell+1}\times\R^{d-2\ell-1}$ around $\mathbf{x}\in X$ so that $\R^{2\ell+1}\times\{0\}$ contains $\R^\ell\times\{\vec{0}\}\times\{0\}$ as the $w$-sheet and $\{\vec{0}\}\times\R^\ell\times\{0\}$ as the $v$-sheet (for $\ell=1$ this is depicted in the figure). The derivative at $v$-sheet of $\realmap(g)$ applied to $\I^{d-2\ell-1}$ gives the positive basis of $\I^{d-2\ell-1}$ and the positive basis of $\R^{d-2\ell-1}$, applied to $\I$ it is the sum of the positive $\I$ direction and the upward pointing vector in our $\R^{2l}$-chart, i.e.\ in $\{\vec{0}\}\times\{\vec{0}\}\times\R\subseteq\R^{2\ell+1}$, while applied to $\D^\ell$ it gives the positive basis of $\{\vec{0}\}\times\R^\ell\times\{0\}$. At $w$-sheet we see the vector in the positive $\I^{d-2\ell-1}$-direction and the positive basis of $\R^\ell\times\{\vec{0}\}\times\{0\}$. Comparing to the canonical basis of $\I^{d-2\ell-1}\times\I\times \R^{d-2\ell-1}\times\R^{2\ell+1}\subseteq\I^{d-2\ell-1}\times\I\times X$ we use $2(d-\ell-1)$ transpositions, so $\varepsilon_{\mathbf{x}}=+1$. We also clearly have $\Da\circ\realmap(-g)=-g$.
% dim 3 (t+up,F, t,R)=(up,F,t,R) vs (t,up,F,R)
% dim 4 (t$^{d-3}$+t+future$_X^{d-3}$+up, F,
% t$^{d-3}$+t, R)
% =(future$_X^{d-3}$+up,F,t$^{d-3}$+t,R) 
% vs (t$^{d-3}$,t, future$_X^{d-3}$, up,F,R)

Finally, we check that $\realmap(g)=(-1)^{d-\ell}\realmap(g^{-1})$ if $\ell\geq2$. Firstly, for any $\ell$ we can perform a homotopy of $\realmap(g)$ by making the first sheet stand still whereas the second sheet moves, and then pushing the neighborhood of $x$ back around $g$, so that the root and the tip of the finger switch positions and the finger follows $g^{-1}$. However, when $\ell\geq2$ the root of the finger can also be freely moved around $\D^\ell$, so that we obtain the class $(-1)^{d-\ell}\realmap(g^{-1})$. Namely, the meridian sphere $\mu(\S^{d-\ell-1})$ to the second sheet became the meridian sphere to the first sheet but with the sign $(-1)^{d-\ell}$, since it got inverted (cf.\ (anti)symmetry of the linking number). Thus, $\Da\circ\realmap=\Id_{\Z[\pi_1X]/\relations_{\ell,d}}$ by construction.
\end{proof}

\begin{remark}\label{rem:Gabai}
    This proves that $\Da$ is surjective, without using Dax's Theorem~\ref{thm:Dax}. For $\ell=1,d=4$, Gabai proves in \cite[Step 4]{Gabai-disks} that $\Da$ is injective also directly, avoiding a parametrized double-point elimination argument of Dax. Namely, Gabai shows that if $\tilde{F}$ has double points $x_i$ with signed loops $\e_{x_i}g_{x_i}$ then $F$ is homotopic to $\realmap(r)$, $r=\sum_i\e_{x_i}g_{x_i}$. But if $r=0\in\Z[\pi_1X]$ then $\realmap(r)$ is clearly null homotopic.
\end{remark}
\begin{remark}\label{rem:GW}
    To show $\realmap\circ\Da=\Id$ one could instead of Dax's theorem use the fundamental theorem of embedding calculus \cite{GKW}, which implies that the evaluation map $ev_2$ from $\Emb_\partial(\D^\ell,X)$ to the second Taylor stage $T_2$ is $(2d-3l-3)$-connected. Since $T_1\simeq\Imm_\partial(\D^\ell,X)$ and $d-\ell\geq 3$ we have
\[
    \pi_{d-2\ell} ev_2\colon
    \pi_{d-2\ell}(\Imm_\partial(\D^\ell,X),\Emb_\partial(\D^\ell,X),u)
    \overset{\cong}{\ra}
    \pi_{d-2\ell}(T_1,T_2,u).
\]
    For example, for $\ell=1$ the last group is isomorphic to $\Z[\pi_1X]$ via an isomorphism $\chi$, see e.g.\ \cite{K-thesis}. Moreover, by a slight generalization of the results there, we have $\chi\circ\pi_{d-3}ev_2\circ\realmap=\Id_{\Z[\pi_1X]}$, see \cite[Rem.~1.10]{K-thesis}. Thus, $\realmap$ is an isomorphism for $\ell=1$, so its unique left inverse $\Da$ is as well.
\end{remark}

%%%%%%%%%%%%%%%%
\subsection{Proof of Theorem~\ref{thm:Dax-final}}\label{subsec:immersions-conn-map}

\subsubsection{On homotopy groups of spaces of immersions}\label{subsec:immersions}
For a Riemannian manifold $X$, let $V_\ell(X)$ denote the $\ell$-frame bundle of the tangent bundle $TX$, where an $\ell$-frame is an ordered set of $\ell$ orthonormal vectors in $TX$. Recall that $\Imm_\partial(\D^\ell,X)$ denotes the space of immersions $\D^\ell\imra X$ that restrict to $u_0\colon\S^{\ell-1}\hra\partial X$ on the boundary $\S^{\ell-1}=\partial\D^\ell\subseteq\D^\ell$. As in Proposition~\ref{prop:e-collar} this space is homotopy equivalent to the subspace $\Imm_{\partial^\e}(\D^\ell,X)$ of those immersions that restrict to a fixed embedding $u_0^\e$ of a collar $\partial\D^\ell\times[0,\e]\subseteq\D^\ell$.

\begin{theorem}[Smale--Hirsch~\cite{Smale,Hirsch}]
    Taking the unit derivatives in all $\ell$ tangent directions at each point of $\D^\ell$ gives a homotopy equivalence 
    \[
\deriv\colon{\Imm}_{\partial^\e}(\D^\ell,X)\ra \Map_{\partial}(\D^\ell, V_\ell(X);\deriv u)
    \]
    to the space of maps $\D^\ell\to V_\ell(X)$, that along $\partial\D^\ell$ agree with the unit derivative of $u$. 
\end{theorem}
    We combine this with the following lemma.
\begin{lemma}\label{lem:U-cup-is-he}
    For a space $Y$ let $f\colon \D^\ell\to Y$ be a based map, $f(e_\ell)=e_Y$ for basepoints $e_\ell\in\partial\D^\ell$ and $e_Y\in Y$. Then there are inverse homotopy equivalences (based for $f$ and $-f\cup_\partial f$) 
    \[\begin{tikzcd}[column sep=large]
        -f\cup_\partial\bull\colon\quad \Map_{\partial}(\D^\ell,Y;f)\arrow[shift left=3pt]{r}[swap]{\sim} & \Map_*(\S^\ell,Y)\eqqcolon\Omega^\ell Y\quad \colon f\vee\bull\arrow[shift left=3pt]{l}
    \end{tikzcd}
    \]
    where $-f\cup_\partial K$ glues two disks along the boundary, while $f\vee S\colon\D^\ell\to\D^\ell\vee\S^\ell\to Y$ is the wedge sum (pinch off a sphere from a neighborhood of the point opposite to the basepoint $e_\ell\in\partial\D^\ell$ in $\partial\D^\ell$).
\end{lemma}
\begin{proof}
    For a homotopy from $-f\cup_\partial(f\vee\bull)$ to $\Id_{\Omega^\ell Y}$ use the obvious null homotopy $-f\cup_\partial f\simeq\const_{e_Y}$. 
    
    Similarly, for the homotopy from $f\vee(-f\cup_\partial\bull)=( f\vee- f)\cup_\partial\bull$ to the identity collapse the part $f\vee-f$ to $e_X$: use the foliation $\phi_v\subseteq\D^\ell$ by the straight lines from $v\in\partial\D^\ell$ to $e_\ell\in\partial\D^\ell$ and then for each $v$ the obvious null homotopy of loop $f(\phi_v) f(\phi_v)^{-1}$ through loops based at $f(v)$.
\end{proof}

\begin{cor}\label{cor:immersions}
    There is a homotopy equivalence 
    \[
    \deriv_{u}(\bull)\coloneqq-\deriv(u)\cup_\partial\deriv(\bull)\colon\Imm_{\partial^\e}(\D^\ell,X)\xrightarrow{\sim}\Omega^\ell V_\ell(X).
    \]
\end{cor}

Thus, by definition, $\deriv_{u}$ sends $K\colon\D^\ell\imra X$ to the map $\deriv_{u}(K)\colon\S^\ell\to V_\ell(X)$ given as follows: if $x\in\S^\ell$ is in the north hemisphere (canonically identified with $\D^\ell$), take the unit derivative $\deriv(K)|_x$, and if $x$ is in the south hemisphere (canonically identified with oppositely oriented $\D^\ell$) take the unit derivative $\deriv(u)|_x$.

The homotopy equivalence $\deriv_u$ induces isomorphisms $\pi_{n}\Imm_\partial(\D^\ell,X)\xrightarrow{\cong} \pi_{n+\ell}V_\ell(X)$. Using this and Proposition~\ref{prop:TX} from Appendix~\ref{app:frame-bundles} about homotopy groups of frame bundles, we obtain the following corollary. Note that it implies that Theorem~\ref{thm-intro:Dax} indeed follows from Theorem~\ref{thm:Dax-final}.
\begin{cor}\label{cor:imm-htpy}
    Assume $d-2\ell\geq0$. The homomorphism $p_u\colon\pi_{n}\Imm_\partial(\D^\ell,X)\to \pi_{n+\ell}X$,
    given by $p_u(f)=(\vec{t}\mapsto -u\cup_\partial f_{\vec{t}})$ union the canonical null homotopy of $-u\cup_\partial u$ on the boundary, is an isomorphism for all $n\leq d-2\ell-1$, and we have an exact sequence
\begin{equation}\label{eq:imm-htpy-seq}
\begin{tikzcd}
    \Z_{\ell,d}\rar[tail] & \pi_{d-2\ell}\Imm_\partial(\D^\ell,X)\rar[two heads]{p_u} &
    \pi_{d-\ell}X,
\end{tikzcd}
\end{equation}
    where $\Z_{\ell,d}\coloneqq\Z/\relations_{\ell,d}$ is isomorphic to $\Z$ for $l=1$ or $d-\ell$ even, and to $\Z/2$ for $d-\ell$ odd with $\ell\geq2$.
\end{cor}

\subsubsection{The connecting map}
Recall the connecting map $\delta_{\Imm}$ from \eqref{eq:delta}, and consider the composite
\[\begin{tikzcd}[column sep=27pt]
    \Z_{\ell,d} \arrow[tail]{r}{i_*} &
    \pi_{d-2\ell}(\Imm_\partial(\D^\ell,X),u) \arrow{r}{\delta_{\Imm}} &
    \pi_{d-2\ell}\big(\Imm_\partial(\D^\ell, X), \Emb_\partial(\D^\ell, X),u\big)
    \arrow[tail, two heads]{r}{\Da} & 
    \faktor{\Z[\pi_1X]}{\relations_{\ell,d}}
\end{tikzcd}
\]
\begin{prop}\label{prop:r(1)}
    This composite takes $1\in \Z_{\ell,d}$ to the class of the unit $1$ in $\Z[\pi_1X]/\relations_{\ell,d}$.
\end{prop}
\begin{proof}
    The map $i_*$ for any $X$ factors through the one for $X=\D^d$, so it suffices to consider that case. Then both $i_*$ and $\Da$ are isomorphisms with $\Z_{\ell,d}=\Z/\relations_{\ell,d}$. By definition, $i_*(1)$ is the class of any map $\tau\colon\S^{d-2\ell}\to\Imm_\partial(\D^\ell,\D^d)$ whose Smale--Hirsch derivative $\deriv_u\circ\tau\colon\S^{d-2\ell}\to\Omega^\ell V_\ell(\D^d)\simeq \Omega^\ell V_\ell(d)$ is a generator of $\pi_{d-\ell}V_\ell(d)\cong \Z_{\ell,d}$ of the Stiefel manifold (see Corollary~\ref{cor:Stiefel}).
    
    Let us describe one such $\tau$. Firstly, for parameters $(\vec{t},s)\in \I^{d-2\ell-1}\times \I\cong\D^{d-2\ell}$ of the upper hemisphere of $\S^{d-2\ell}$ let $\tau(\vec{t},s)\coloneqq \realmap(1)_{\vec{t},1-s}$ be the time-reversed path of immersed disks from the previous proof. Recall that this drags a piece of $u$ to the position $\alpha_N$, and then uses $(\ell+1)$-disks foliating the meridian ball $\ol{\mu}(\ball^{d-\ell})$ to slide $\alpha_N$ into an embedded $\ell$-disk $\realmap(1)_{\vec{t}}\subseteq\mu(\S^{d-\ell-1})$. For parameters in the lower hemisphere of $\S^{d-2\ell}$ we now describe how to undo disks $\realmap(1)_{\vec{t}}$ by an isotopy, so it will immediately follow that $\Da(i_*(1))=\Da([\tau])=\Da\circ\realmap(1)=1$.
    
\begin{figure}[!htbp]
    \centering
    \begin{tikzpicture}[scale=0.7,every node/.style={scale=1},even odd rule, decoration = {markings, mark = at position 0.5 with {\arrowreversed{latex}} }]
    \clip (.7,-1.05) rectangle (13.65,3.8);
    
    \draw[line width=0.8pt]
        (1.5,0) circle (1.1pt) node[below]{\small$u({-}1)$} -- (3,0) 
        (5,0) -- (9.5,0) (10,0) -- (13,0) circle (1.1pt) node[below]{\small$u(1)$}
        (11.2,0.7) 
        % node{\small$\overline{\mu}(\ball^2_{\mathbf{x}})$}
        ;
    \draw[thick,dashed,black,dash phase=3pt] 
        (3,0) -- (5,0);

    \fill[orange!40,fill opacity=0.2, draw opacity=1,line width=1pt]
        (9.6,1.8) to[out=-95,in=-98, distance=1.95cm] (10.25,-.1) --
        (10.25,.1) to[out=70,in=-70, distance=0.15cm] (10.1,2) -- cycle;
    \draw[ultra thick, cyan] 
        (10, 0) circle (2pt) node[above=1pt]{$\mathbf{x}$};
    \draw[thick,black!60] 
        (9.9, 0.6) node[above=1pt]{$C$};
    % alpha arc
    \draw[red, line width=1pt] 
        (9.6,1.8) circle (1.25pt)
        --
        (10.1,2) node[right]{$\alpha_N$} circle (1.25pt);
    % middle
    \draw[black, line width=0.8pt]  
        (5,0) to[out=90,in=90, distance=2.7cm] (9.6,1.8)
        (10.1,2) to[out=100,in=90, distance=3cm] (3,0);
    % tip
    \draw[red!50!yellow, line width=0.8pt] 
            (9.6,1.8) to[out=-95,in=-98, distance=1.95cm] (10.25,-.1)
            (10.25,.1) to[out=70,in=-70, distance=0.15cm] (10.1,2)
            (10.7,-0.7) node{$\alpha_S$};
    % % gp elt
    % \fill[color=black!20!white] (6.4,4.2) circle (0.55cm);
    % \draw[black,postaction = decorate] (6.4,4.2) circle (0.65cm) node[left=0.4cm]{$g$};
    % % \draw[dashed] (3,0) to[out=100,in=-160, distance=1.4cm] (6,3.77);
    % % \fill (6,3.77) circle (0.9pt);
\end{tikzpicture}~\begin{tikzpicture}[scale=0.7,every node/.style={scale=1},even odd rule, decoration = {markings, mark = at position 0.5 with {\arrowreversed{latex}} }]
    \clip (.7,-1.05) rectangle (13.65,3.8);
    
    \draw[line width=0.8pt]
        (1.5,0) circle (1.1pt) node[below]{\small$u({-}1)$} -- (3,0) 
        (11.9,0) -- (13,0) circle (1.1pt) node[below]{\small$u(1)$} ;
    \draw[thick,dashed,black,dash phase=4pt] 
        (3,0) -- (12,0) ;
    \fill[orange!40,fill opacity=0.2, draw opacity=1,line width=1pt]
        (9.6,1.8) to[out=-95,in=-98, distance=1.95cm] (10.25,-.1) --
        (10.25,.1) to[out=70,in=-70, distance=0.15cm] (10.1,2) -- cycle;
    
    % alpha arc
   \draw[red, line width=1pt] 
    (9.6,1.8) circle (1.25pt)
                -- (10.1,2) node[right]{$\alpha_N$} circle (1.25pt);
    % middle
    \draw[black, line width=0.8pt]               (10.1,2) to[out=100,in=90, distance=3cm] (3,0);
    % tip
    \draw[red!50!yellow, line width=0.8pt] 
            (9.6,1.8) to[out=-95,in=-98, distance=1.95cm] (10.25,-.1)
            (10.25,.1) to[out=70,in=-70, distance=0.15cm] (10.1,2)
            (10.7,-0.7) node{$\alpha_S$};
    \fill[white] 
        (10.1,0.95) -- (10.3,0.7) -- (10.5,1) -- (10.1,1.4) --cycle;
    \fill[white] 
        (9.7,1.7) -- (9.9,1.8) -- (9.9,2) -- (9.7,2) --cycle;
    % isotoped middle arc
    \draw[black, line width=0.8pt] 
        (11.8,0) to[out=180,in=91,distance=2.2cm] (9.6,1.8);
\end{tikzpicture}
    \caption{The present slice $\D^3\times\{0\}\subseteq\D^d$ contains an arc $c\subseteq\realmap(1)_{\vec{t}_{\mathbf{x}}}$, and a 2-disk $C\subseteq\ol{\mu}(\ball^{\ell+1}_{\mathbf{x}})$. In the second picture $c$ is isotoped so that now the subarc $\alpha_S$ can be slid across $C$ without creating any double points.}
    \label{fig:realmap-isotopy}
\end{figure}
    Observe that in the foliation of $\ol{\mu}(\ball^{d-\ell})$ there is a unique $(\ell+1)$-disk $\ol{\mu}(\ball^{\ell+1}_{\mathbf{x}})$ which contains $\mathbf{x}$, the only double point in the homotopy. We can pick coordinates so that the intersection of this disk with the ``present'' slice is a $2$-disk $\ol{\mu}(\ball^2_{\mathbf{x}})=\ol{\mu}(\ball^{\ell+1}_{\mathbf{x}})\cap\D^3\times\{\vec{0}\}\subseteq\D^d$ with $\partial\ol{\mu}(\ball^2_{\mathbf{x}})=\alpha_N\cup_\partial \alpha_S$ as on the left of Figure~\ref{fig:realmap-isotopy}. In particular, $\realmap(1)_{\vec{t}_{\mathbf{x}}}$ contains $\alpha_S$.

    Now, let us first isotope the arc $\realmap(1)_{\vec{t}_x}\cap\D^3\times\{\vec{0}\}$ as in Figure~\ref{fig:realmap-isotopy}: we isotope the front guiding arc and a part of $u$ in $\realmap(1)_{\vec{t}_x}$ by ``pulling them through'' $\ol{\mu}(\ball^2_{\mathbf{x}})$ (using the rotation around the vertical axis). Note that in the new position $\alpha_S$ can be slid to $\alpha_N$ across $\ol{\mu}(\ball^2_{\mathbf{x}})$ without creating any double points, and that from there we have an obvious isotopy to $u$ -- namely, ``by pulling tight''. More generally, the desired isotopy from each $\realmap(1)_{\vec{t}}\cap\D^3\times\{\vec{0}\}$ to $u$ consists of the same isotopy as in Figure~\ref{fig:realmap-isotopy}, then sliding across the corresponding $2$-disk $\ol{\mu}(\ball^{\ell+1}_{\vec{t}})\cap\D^3\times\{\vec{0}\}\subseteq\D^d$ to get to position $\alpha_N$, and then pulling it tight.
    
    Finally, we make this into an isotopy of the whole family $\realmap(1)_{\vec{t}}$. Firstly, we``taper off'' in the remaining $d-3$ dimensions the isotopy of the guiding arc performed in the present slice; this is a standard procedure using smooth bump functions. From there, we simply use the isotopies across $\ol{\mu}(\ball^{\ell+1}_{\vec{t}})$ as before.
\end{proof}

\subsubsection{Proof of Theorem~\ref{thm:Dax-final}}
    The last proposition implies $1\in\im(\Da\circ\delta_{\Imm})$. We can use any section of $p_u\colon\pi_{d-2\ell}(\Imm_\partial(\D^\ell,X),u)\sra\pi_{d-\ell}X$ to define a group homomorphism
    \begin{equation}\label{eq:md-def}
    \begin{tikzcd}[column sep=small]
        \md_u\colon\pi_{d-\ell}X\arrow{r}{} & \pi_{d-2\ell}(\Imm_\partial(\D^\ell,X),u)\arrow{rr}{\Da\circ\delta_{\Imm}} &&
        \faktor{\Z[\pi_1X]}{\langle1,\relations_{\ell,d}\rangle}\cong\faktor{\Z[\pi_1X\sm1]}{\relations_{\ell,d}}.
    \end{tikzcd}
    \end{equation}
    By construction the value $\md_u([f])$ is computed by lifting $f\colon(\I^{d-\ell},\partial \I^{d-\ell})\to (X, u(-1))$ to any family $F\colon \I^{d-2\ell}\to\Imm_\partial(\D^\ell,X)$ (in this case the entire $\partial \I^{d-\ell}$ goes to $u$), calculating its Dax invariant 
    \[
    \Da(F)=n(f)\cdot 1+\md_u(f)\in\faktor{\Z[\pi_1X]}{ \relations_{\ell,d}},
    \]
    and disregarding the trivial group elements $n(f)$.
    
    Thus $\im(\Da\circ\delta_{\Imm})=\langle1,\md_u(\pi_{d-\ell}X)\rangle$, finishing the proof.
\hfill\qed

\section{On homotopy groups of spaces of \texorpdfstring{$\e$}{e}-augmented disks, and of disks with a dual}\label{sec:e-disks}
%%%%%%%%%%%%%%%%%%%%%
Recall that the space of $\e$-augmented disks $\Emb_{\partial^\e}^\e(\D^\ell,X)$ (previously denoted $\Emb_{u_0^\e}^\e(\D^{k-1},X)$ for $\ell=k-1$) consists of embeddings $\D^\ell\times[0,\e]\hra X$ that agree on $(\S^{\ell-1}\times [0,\e])\times[0,\e]\subseteq\D^\ell\times[0,\e]$ with one such $u^\e$ (previously denoted $u_+^\e$). In other words, their boundary condition is $u_0^\e\coloneqq u^\e|_{(\S^{\ell-1}\times [0,\e])\times[0,\e]}$. Note that whereas $\partial^\e$ records, as before, a stronger boundary condition $u|_{\S^{\ell-1}\times[0,\e]}$ along a collar of $\partial\D^\ell$, the superscript $\e$ reflects additional structure, the $\e$-augmentation.

This additional structure has a fairly simple homotopy type that just reflects a normal vector field along an $\ell$-disk. In other words, the fiber of the forgetful map $\ev_0\colon\Emb_{\partial^\e}^\e(\D^\ell,X)\to \Emb_{\partial^\e}(\D^\ell,X)$
agrees with the analogous fiber for immersions, or equivalently frame bundles, see Section~\ref{subsec:forg-augm}. Using this and results about frame bundles from Appendix~\ref{app:frame-bundles}, in Section~\ref{subsec:ribbons} we extend Theorem~\ref{thm-intro:Dax} to the space $\Emb_{\partial^\e}^\e(\D^\ell,X)$. Finally, in Section~\ref{subsec:proof-of-ThmD} we combine this with our Space Level Light Bulb Theorem~\ref{thm:anyD} to prove Theorem~\ref{thm-intro:combined}.

\subsection{Forgetting augmentations}\label{subsec:forg-augm}

\begin{prop}\label{prop:normal-v-field}
    The space $\Emb_{\partial^\e}^\e(\D^\ell,X)$ is homotopy equivalent to the space $\Emb_{\partial^\e}^\uparrow(\D^\ell,X)$ of neat embeddings $\D^\ell\hra X$ equipped with a normal vector field (a nonvanishing section of the normal bundle).
\end{prop}
\begin{proof}
    Firstly, we claim that there is a commutative diagram of fibration sequences
\[\begin{tikzcd}
    \ev_0^{-1}(u) \arrow{r} \arrow{d}{} &   
    \Emb_{\partial^\e}^\e(\D^\ell,X) \arrow{r}{\ev_0} \arrow{d}{\deriv^\uparrow} & 
    \Emb_{\partial^\e}(\D^\ell,X) \arrow[equal]{d}\\
    \Gamma_{\partial^\e}\big(\S\nu u\big) \arrow{r}&   
    \Emb_{\partial^\e}^\uparrow(\D^\ell,X) \arrow{r}{pr^\uparrow} & 
    \Emb_{\partial^\e}(\D^\ell,X).
\end{tikzcd}
\]
    where $\Gamma_{\partial^\e}\big(\S\nu u\big)=(pr^\uparrow)^{-1}(u^\uparrow)$ is the space of those sections of the unit sphere bundle $\S\nu u$ of the normal bundle of our basepoint $u\colon\D^\ell\hra X$ that agree with the basepoint $u^\uparrow\coloneqq\deriv^\uparrow(u^\e)$ on a collar of $\partial\D^\ell$.
        Indeed, both $\ev_0$ and $pr^\uparrow$ are fibrations by Theorem~\ref{thm:Cerf-II}, and for a fixed Riemannian metric on $X$, the unit derivative along $\D^\ell\times\{0\}$ in the direction of $[0,\e]$ is a map $\deriv^\uparrow$ between total spaces. Once we show that its restriction to fibers  $\deriv^\uparrow\colon\ev_0^{-1}(u)\to \Gamma_{\partial^\e}\big(\S\nu u\big)$
   are homotopy equivalences, the result will follow.

    A homotopy inverse $\Exp_{u}$ of $\deriv^\uparrow$ comes from identifying the total space of $\nu u$ with a tubular neighborhood of $u$ via a scaled exponential map: for a unit normal vector field $\xi$  along $u$ define $\Exp_{u}(\xi)\colon\D^\ell\times[0,\e]\hra X$ by $\Exp_{u}(\xi)(v,s)\coloneqq \exp(s\cdot\xi(v))$. Here we may assume by compactness of $X$ that $\e$ is smaller than the injectivity radius of the chosen metric.
    We have $\deriv^\uparrow\circ \Exp_{u}=\Id$ by construction. To define a homotopy from $\Exp_{u}\circ\deriv^\uparrow$ to the identity in the space $\ev_0^{-1}(u)$ we observe that by continuously scaling the parameter and using the exponential map, it suffices to construct such a homotopy $K_t$ for $K\colon \D^\ell \times [0,\e] \hra \nu u$. This is given by $K_t(v,s)\coloneqq \frac{1}{t}K(v, t\cdot s)$ for $t\in [0,1]$, since $K_0$ is indeed the usual description of the normal derivative of $K=K_1$ at $(v,0)$.
\end{proof}

Recall that $V_\ell(X)$ denotes the $\ell$-frame bundle of the tangent bundle of $X$, and that the unit derivative defines a map $\deriv_u\colon\Emb_{\partial^\e}(\D^\ell,X)\to\Omega^lV_l(X)$, see Corollary~\ref{cor:immersions}. We similarly have a map
\[
    \deriv_{u^\uparrow}\colon\Emb_{\partial^\e}^\uparrow(\D^\ell,X)\ra \Map_{\partial}(\D^\ell, V_{\ell+1}(X);\deriv(u^\e))\ra\Omega^\ell V_{\ell+1}X,
\]
% where for an $\e$-augmented disk $K\colon\D^\ell\times[0,\e]\hra X$ we have $\deriv_{u^\e}(K)\coloneqq-\deriv(u^\e)\cup_\partial\deriv(K)$ and $\deriv(K)\colon\D^\ell\to V_{\ell+1}(X)$ is at $v\in\D^\ell$ given by derivatives of $K$ at $(v,0)$ in all $\ell$ tangent directions plus in the $\e$-direction. 
which to a $k$-disk $K$ with a normal vector field assigns $\deriv_{u^\uparrow}(K)\coloneqq-\deriv(u^\uparrow)\cup_\partial\deriv(K)$, where $\deriv(K)$ is given by derivatives of $K$ at $(v,0)$ in all $\ell$ tangent directions, followed by the normal vector.

\begin{prop}\label{prop:forg-augm}
   There is a commutative diagram of fibration sequences
    \[\begin{tikzcd}
        \Omega^\ell\S^{d-\ell-1} \arrow{r}{n} \arrow[equal]{d} &   
        \Emb_{\partial^\e}^\uparrow(\D^\ell,X) \arrow{r}{pr^\uparrow} \arrow{d}{\deriv_{u^\uparrow}} & 
        \Emb_{\partial^\e}(\D^\ell,X) \arrow{d}{\deriv_{u}}\\
        \Omega^\ell\S^{d-\ell-1} \arrow{r}{i_{\ell+1}} &   
        \Omega^\ell V_{\ell+1}(X) \arrow{r}{pr_{\ell+1}} & 
        \Omega^\ell V_\ell(X).
    \end{tikzcd}
    \]
    In particular, we can choose the connecting map $\delta_{\ev_0}$ for the top fibration sequence in the diagram of Proposition~\ref{prop:forg-augm} to be the composite of $\Omega\deriv_{u}$ with the connecting map $\delta_{pr_{\ell+1}}$ for the bottom sequence.
\end{prop}
\begin{proof}
    A trivialization of the sphere bundle $\S\nu u\cong \D^\ell \times \S^{d-\ell-1}$ induces a homeomorphism $h$ between the space $\Gamma(\S\nu u)$ of all its sections and the space $\Map(\D^\ell,\S^{d-\ell-1})$. This identifies the basepoint $u^\uparrow=\deriv^\uparrow(u^\e)$ with some $u'\colon\D^\ell\to\S^{d-\ell-1}$, so that the subspace $\Gamma_{\partial^\e}(\S\nu u)$ of sections that agree near boundary with $u^\uparrow$ is homotopy equivalent to the subspace $\partial^{-1}(\partial u')\subseteq\Map(\D^\ell,\S^{d-\ell-1})$,
    the fiber over $\partial u'$ of the restriction map $\partial\colon\Map(\D^\ell,\S^{d-\ell-1})\to\Map(\S^{\ell-1},\S^{d-\ell-1})$. Moreover, we can identify this space $\partial^{-1}(\partial u')$ of maps rel.\ boundary with a $(k-1)$-fold loop space as in Lemma~\ref{lem:U-cup-is-he}, to obtain the map
\begin{equation}\label{eq:exp-map}
    \begin{tikzcd}[column sep=large]
        n\colon\Omega^\ell\S^{d-\ell-1} \arrow{r}{u'\vee\bull}[swap]{\sim} & 
        \partial^{-1}(\partial u') \rar{h}[swap]{\cong} & \Gamma_{\partial^\e}\big(\S\nu u\big)\hra \Emb_{\partial^\e}^\uparrow(\D^\ell,X)
        % \arrow{r}{\Exp_{u}} & \ev_0^{-1}(u).
    \end{tikzcd}
\end{equation}
    % This sends $\xi\colon\S^\ell\to\S^{d-\ell-1}$ to the $\e$-augmented disk $\Exp_{u}(u'\vee\xi)$, where $\Exp_u$ is the homotopy inverse of $\deriv^\uparrow$ from the previous proof. Thus, $\Exp_{u}(u'\vee\xi)$ is given as $u^\e$ on a neighborhood $*\in\partial X$, and otherwise it comes from integrating the vector field $\xi$. More precisely, write  $\D^\ell=\D^\ell_1\cup\D^\ell_2$ as the union of two halves, with $\D^\ell_1$ the half containing the basepoint $e_{\partial\D^\ell}$, and $u_i=u|_{\D^\ell_i}$. Let $u_1^\e$ be $u^\e$ scaled down to the image of $u_1$. Then the map \eqref{eq:exp-map} is given by $\xi\mapsto\Exp_{u}(u'\vee\xi)=u_1^\e\cup\Exp_{u_2}(\xi)$.
    which we use as the top left arrow in the diagram of the statement. For basepoints at the bottom we use the images of $u^\uparrow$ and $u$ under the vertical derivative maps.
    The square on the right clearly commutes: forgetting normal vector corresponds to forgetting the last vector in an $(\ell+1)$-frame. The square on the left also commutes, since the map $S\mapsto-\deriv(u^\uparrow)\cup_\partial\deriv(h(u'\vee S))=\deriv(-u^\uparrow)\cup_\partial \deriv(u^\uparrow\cup h(S))$ is homotopic to the inclusion $S\mapsto i_{\ell+1}(S)$ of the fiber of $pr_{\ell+1}$ over the basepoint $\deriv_{u}(u)=-\deriv(u)\cup_\partial\deriv(u)$.
\end{proof}
\begin{remark}\label{rem:proof-ThmA}
    Combining Theorem~\ref{thm-intro:LBT-deloop} with Proposition~\ref{prop:forg-augm} (for $\ell=k-1$) we obtain a proof of Theorem~\ref{thm-intro:LBT-fib-seq}: in the setting with a dual there is a fibration sequence
\[\begin{tikzcd}
      \Omega^{k}\S^{d-k} \arrow{rr}{\Exp\circ\amb_\U} &&
      \Emb_s(\D^k,M) \arrow{rr}{\foliate_\U\coloneqq\ev_0\circ\foliate^\e_\U} && \Omega\Emb_\partial(\D^{k-1},M_G) \arrow{r}{\delta_{\ev_0}} & \Omega^{k-1}\S^{d-k}.
      \end{tikzcd}
\]
    In particular, $\foliate_\U$ is a homotopy equivalence if $d=k$ or $d={k}{+}{1}\geq3$. If $d>2k$, then $\pi_0\foliate_\U$ is a bijection.
\end{remark}
% Note that the target of $\deriv_{u^\uparrow}$ can be viewed as the space of $\e$-augmented immersions.

%%%%%%%%%%%%%%%%%%%%%%%%%%%%%%%%%%%%%
%%%%%%%%%%%%%%%%
\subsection{Homotopy groups of spaces of \texorpdfstring{$\e$}{e}-augmented disks} 
\label{subsec:ribbons}

We have the following analogue of Theorem~\ref{thm-intro:Dax} for $\e$-augmented disks. Note that we assume $d-2\ell\geq1$ for simplicity and as the case $d-2\ell=0$ does not arise in Theorem~\ref{thm-intro:combined}.

\begin{theorem}\label{thm:Dax-augmented}
  Assume $1\leq\ell\leq d-3$ and $d-2\ell\geq1$, and $X$ is a $d$-manifold with boundary, $\pi\coloneqq\pi_1X$. For $1\leq n\leq d-2\ell-2$ there are isomorphisms $p_u\colon\pi_n(\Emb_{\partial^\e}^\e(\D^\ell,X),u^\e)\cong\pi_{n+\ell}X$, and a short exact sequences of groups:
\[\begin{tikzcd}[row sep=small,ampersand replacement=\&]
\begin{drcases}
    \ell=1, &\faktor{\Z[\pi\sm1]}{\md_u(\pi_{d-1}X)}\\
    \ell\geq2, d-\ell\text{ odd}, & \faktor{\Z[\pi\sm1]}{\langle g+g^{-1}\rangle\oplus\md_u(\pi_{d-\ell}X)}\\
    \ell\geq2, d-\ell\text{ even}, & \faktor{\Z[\pi]}{\langle g-g^{-1}\rangle\oplus\md^\e_{u^\e}(\pi_{d-\ell}X)}
\end{drcases} \arrow[tail]{r}{\partial\realmap^\e} 
        \& \pi_{d-2\ell-1}(\Emb_{\partial^\e}^\e(\D^\ell,X),u^\e)
        \arrow[two heads]{d}{(\eta_W\circ\deriv_{u^\e})\oplus p_u}\\
        \& \hspace{-1.8cm}
\begin{drcases}
     d-\ell\text{ odd}, & \Z\\
     d-\ell\neq2,4,8\text{ even}, & \Z/2
\end{drcases}\oplus\pi_{d-\ell-1}X.\hspace{1cm}
     \end{tikzcd}
\]
  Moreover, for $d-\ell$ odd, $(\eta,\pi_{d-2\ell-1}\ev_0)\colon\pi_{d-2\ell-1}\Emb_{\partial^\e}^\e(\D^\ell,X)\cong \Z\times\pi_{d-2\ell-1}\Emb_\partial(\D^\ell,X)$ is an isomorphism, so any $\S^{d-2\ell-1}$-family of embedded disks has $\Z$ many $\e$-augmentations. On the other hand, in the even case, the number of augmentations is twice the order of the element $1$ in $\Z[\pi]/\md^\e_{u^\e}(\pi_{d-\ell}X)$.
\end{theorem}
The homomorphisms $\md^\e_{u^\e}$ and  $\eta_W$ will be defined in the course of the proof, see \eqref{eq:md-e-def}. The map $\partial\realmap^\e$ is the family of disks $\partial\realmap$ with suitable $\e$-augmentations; $p_u$ is the same composite of a forgetful map and concatenation with $-u$ as in~\eqref{eq:imm-htpy-seq}.

To prove Theorem~\ref{thm:Dax-augmented} we consider the diagram of fibration sequences from Proposition~\ref{prop:forg-augm}.
% $\deriv_{u^\e}\colon\Emb_{\partial^\e}^\e(\D^1,X)\to \Omega V_2(X)$ from Proposition~\ref{prop:forg-augm}, which said that $\deriv_{u^\e}$ takes the fibration $\ev_0\colon\Emb_\partial^\e(\D^{1},X)\to\Emb_\partial(\D^1,X)$, with fiber $\Omega^{k-1}\S^{d-k}$, to the looping of the fibration $p_2\colon V_2(X)\to V_1(X)$. Recall that $V_k(X)$ denotes the space of $k$-frames of $TX$ and we have homotopy equivalences $\Imm_\partial(\D^{1},X)\simeq\Omega V_1(X)=\Omega\S(X)\simeq \Omega(\S^{d-1}\times X)$ (see Lemma~\ref{lem:immersions}). 
Taking the long exact sequences in homotopy groups of these fibrations implies that $\deriv_{u^\uparrow}$ is $(d-2\ell)$-connected (as is $\deriv_u$), so is $\deriv_{u^\e}\coloneqq\deriv_{u^\uparrow}\circ\deriv^\uparrow$, and gives the following commutative diagram with exact rows and columns:
\begin{equation}\label{diag:ribbons}
\begin{tikzcd}[column sep=0.34cm]
    &&& \pi_{d-\ell}\S^{d-\ell-1} \arrow{rr}\arrow{d} 
    &&  \pi_{d-\ell}V_{\ell+1}(X) \arrow{rr}{\pi_{d-\ell}pr_{\ell+1}}\arrow{d}{\delta_{\Imm^\e}} 
    &&  \pi_{d-\ell}V_\ell(X) \arrow{d}{\delta_{\Imm}}
\\
    &&& 0\arrow{d}\arrow[]{rr} 
    && \pi_{d-2\ell}\big(\Omega^\ell V_{\ell+1}(X),\Emb^\e\big)\arrow{rr}{\cong}\arrow{d}{\partial^\e} 
    && \pi_{d-2\ell}\big(\Omega^\ell V_\ell(X),\Emb\big)\arrow{d}{\partial} 
\\
    \pi_{d-2\ell}\Emb\arrow[]{rrr}{\delta_{\ev_0}}\arrow{d}[swap]{\pi_{d-2\ell}\deriv_u} 
    &&& \pi_{d-\ell-1}\S^{d-\ell-1} \arrow{rr}{\pi_{d-2\ell-1}\Exp_u}\arrow[equals]{d} 
    &&  \pi_{d-2\ell-1}\Emb^\e \arrow[two heads]{rr}{\pi_{d-2\ell-1}\ev_0}\arrow[two heads]{d}{\pi_{d-2\ell-1}\deriv_{u^\e}} 
    &&  \pi_{d-2\ell-1}\Emb \arrow[two heads]{d}{\pi_{d-2\ell-1}\deriv_u} 
\\
    \pi_{d-\ell}V_\ell(X)\arrow{rrr}{\delta_{pr_{\ell+1}}} 
    &&& \pi_{d-\ell-1}\S^{d-\ell-1}\arrow[]{rr}{\pi_{d-\ell-1}i_{\ell+1}} 
    && \pi_{d-\ell-1}V_{\ell+1}(X) \arrow[two heads]{rr}{\pi_{d-\ell-1}pr_{\ell+1}} 
    && \pi_{d-\ell-1}V_\ell(X)
\end{tikzcd}
\end{equation}
We abbreviate $\Emb\coloneqq\Emb_{\partial^\e}(\D^\ell,X)$ and $\Emb^\e\simeq\Emb_{\partial^\e}^\e(\D^\ell,X)$, and $\delta_f$ is the connecting map for a fibration $f$. Recall the isomorphism $\Da\colon\pi_{d-2\ell}(\Omega^\ell V_\ell(X),\Emb)\to \Z[\pi_1X]/\relations_{\ell,d}$ and its inverse $\realmap$ from Section~\ref{subsec:Dax} (using that $\Imm_{\partial^\e}(\D^\ell,X)\simeq\Omega^\ell V_\ell(X)$ by Corollary~\ref{cor:immersions}).
Thus, we need to compute $\pi_{d-\ell-1}V_{\ell+1}(X)$ and
the kernel of the surjection $\pi_{d-2\ell-1}\deriv_{u^\e}$. To this end, we study homotopy groups of frame bundles, and in Proposition~\ref{prop:TX} show that for $d-2\ell-1>0$:
\begin{equation}\label{eq:Vl+1X}
\begin{tikzcd}[ampersand replacement=\&]
    \coker(\delta_{pr_{\ell+1}})\cong\Z_{\ell+1,d}
    \arrow[tail]{rr}{i_{\ell+1}} 
    \&\& \pi_{d-\ell-1}V_{\ell+1}(X) \arrow[two heads]{rr}{pr_{\ell+1}} 
    \&\& \pi_{d-\ell-1}V_\ell(X)\cong\pi_{d-\ell-1}X
\end{tikzcd}
\end{equation}
where $\Z_{\ell+1,d}=\Z$ if $d-\ell-1$ even (i.e.\ $\im(\delta_{pr_{\ell+1}})=0$), and $\Z_{\ell+1,d}=\Z/2$ if $d-\ell-1$ odd ($\im(\delta_{pr_{\ell+1}})=2\Z$). 

Moreover, there are splittings $\eta_W$ of $\pi_{d-\ell-1}i_{\ell+1}$ for every $d-\ell\neq2,4,8$, see Proposition~\ref{prop:TX}. Thus, in these cases we have the desired right hand side in Theorem~\ref{thm:Dax-augmented}:
\[
    \eta_W\oplus\pi_{d-\ell-1}pr_{\ell+1}\colon\pi_{d-\ell-1}V_{\ell+1}(X)\xrightarrow{\cong} \Z_{\ell+1,d} \oplus \pi_{d-\ell-1}X.
\]

\begin{proof}[Proof of Theorem~\ref{thm:Dax-augmented}]
    It remains to find $\ker(\pi_{d-2\ell-1}\deriv_{u^\e})$, which is the quotient of $\Z[\pi_1X]$ by $\im(\delta_{\Imm^\e})$.

    Assume first $d-\ell$ is odd, so $d-\ell-1$ is even. Since $\im(\delta_{pr_{\ell+1}})=0$ the map $\pi_{d-\ell}pr_{\ell+1}$ is surjective. Looking at the top right of \eqref{diag:ribbons}, it follows that $\im(\delta_{\Imm})\cong\im(\delta_{\Imm^\e})$, so $\pi_{d-2\ell-1}\ev_0$ induces an isomorphism $(\ev_0)_*\colon\ker(\pi_{d-2\ell-1}\deriv_{u^\e})\xrightarrow{\cong}\ker(\pi_{d-2\ell-1}\deriv_{u})$. The latter is isomorphic to $\Z[\pi_1X\sm1]/\relations_{\ell,d}\oplus\md_u (\pi_{d-1}X)$ by Theorem~\ref{thm-intro:Dax}, so we get the desired exact sequence in the theorem in this case. The maps are $\partial\realmap^\e\coloneqq (\ev_0)_*^{-1}\partial\realmap$, and $\eta_W\circ\pi_{d-2\ell-1}\deriv_{u^\e}$ and $p_u\circ\pi_{d-2\ell-1}\ev_0$. 

    Furthermore, in this case $\pi_{d-2\ell-1}\Exp_u$ is injective (since $\pi_{d-\ell-1}i_{\ell+1}$ is), so $\eta_W\circ\deriv_{u^\e}$ is its left splitting. Therefore, we have the claimed isomorphism
    \[
        (\eta_W\circ\deriv_{u^\e})\oplus\pi_{d-2\ell-1}\ev_0\colon
        \;\pi_{d-2\ell-1}\Emb_\partial^\e(\D^1,X)\xrightarrow{\cong}\Z\oplus\pi_{d-2\ell-1}\Emb_\partial(\D^1,X).
    \]
    Now assume $d-\ell$ is even, so $d-\ell-1$ is odd. Since $\im(\delta_{pr_{\ell+1}})=2\Z$ and $\ker(\delta_{pr_{\ell+1}})=\im(\pi_{d-\ell}pr_{\ell+1})$ we have the following horizontal exact sequence
\[\begin{tikzcd}[row sep=1.7em]
    & \Z\arrow[tail]{d} \arrow{drr}{\cong}  
    & \\
    \im(\pi_{d-\ell}pr_{\ell+1})\arrow[tail]{r}\arrow{dr}[swap]{p} 
    & \pi_{d-\ell}V_\ell(X)\arrow[two heads]{rr}[near start]{\delta_{pr_{\ell+1}}}\arrow[two heads]{d}{\pi_{d-\ell}pr_\ell}
    &&  2\Z.  \\
    & \pi_{d-\ell}X & 
\end{tikzcd}
\]
    The vertical sequence is just \eqref{eq:Vl+1X} with the index $\ell+1$ replaced by $\ell$ (note $\Z_{\ell,d}=\Z$ for $d-\ell$ even).
    This implies that $p$ is an isomorphism, and we can define $\md^\e_{u^\e}$ as the composite
\begin{equation}\label{eq:md-e-def}
\begin{tikzcd}[column sep=small]
    \md^\e_{u^\e}\colon\pi_{d-\ell}X\arrow[tail]{r}{p^{-1}} & \im(\pi_{d-\ell}pr_{\ell+1})\subseteq \pi_{d-\ell}V_\ell(X)\cong\pi_{d-2l}(\Imm_\partial(\D^\ell,X),u)\arrow{rr}{\Da\circ\delta_{\Imm}} &&
    \Z[\pi_1X].
\end{tikzcd}
\end{equation}
    Cf.\ the definition of $\md_u$ in~\eqref{eq:md-def}.
    Then by construction we have $\im(\delta_{\Imm^\e})=\im(\md^\e_{u^\e})\subseteq\Z[\pi_1X]$. This gives the second claimed short exact sequence in the theorem, with the maps as before.

    Finally, note that in this case $(\ev_0)_*\colon\ker(\pi_{d-2\ell-1}\deriv_{u^\e})\sra\ker(\pi_{d-2\ell-1}\deriv_{u})$ is not an isomorphism in general, but has for the kernel the cyclic group generated by the class of $1\in\Z[\pi_1X]$ modulo $\md^\e_{u^\e}(\pi_{d-1}X)$. Taking the kernels in the bottom of~\eqref{diag:ribbons} we get the exact sequence $\ker(\ev_0)_*\hra\ker(\pi_{d-2\ell-1}\ev_0)\sra\Z/2$, so the cardinality of $\ker(\pi_{d-2\ell-1}\ev_0)=\Z/\im(\delta_{\ev_0})$ -- which is precisely the number of augmentations of an arc $u\in\pi_{d-2\ell-1}\Emb_\partial(\D^\ell,X)$ -- is equal to two times the mentioned order.
\end{proof}
\begin{remark}\label{rem:Exp-pic}
    The map $\pi_{d-2\ell-1}(\Exp_u)\colon\Z\to\pi_{d-2\ell-1}\Emb_{\partial^\e}^\e(\D^\ell,X)$ is on a generator given by ``integrating'' the $(d-\ell-1)$-family of unit normal vector fields to $u$, given by its meridian $\mu(\S^{d-\ell-1})$ at a point $p=u_+(x)$. See the proof of Proposition~\ref{prop:forg-augm}, Remark~\ref{rem:amb-Exp} and \cite[Fig.5.9]{KT-4dLBT}.
\end{remark}

%%%%%%%%%%%%%%%%
\subsubsection{The 3-dimensional case} \label{subsec:d=3}
We have so far considered $d\geq4$. However, for $d=3$ we still have an exact sequence comparing embedded to immersed arcs, which using Corollary~\ref{cor:immersions} translates to:
\[\begin{tikzcd}
    % \pi_1(\Emb_\partial(\D^1,X),u) \rar{\pi_1\incl} &
    \pi_1(\Imm_\partial(\D^1,X),u) \cong \Z\oplus\pi_2X 
    \rar{\delta_{\Imm}} & \pi_1^{rel}\rar
    & \pi_0\Emb_\partial(\D^1,X) \rar[two heads]{\pi_0\incl} & \pi_0\Imm_\partial(\D^1,X) \cong\pi_1X.
\end{tikzcd}
\]
    where an element of the set $\pi_1^{rel}\coloneqq\pi_1(\Imm_\partial(\D^1,X), \Emb_\partial(\D^1,X);u)$ is represented by a knot together with a path to $u$ through immersed arcs. Moreover, one still has a well-defined surjection $\Da\colon \pi_1^{rel}\sra\Z[\pi_1X]$, with a set-theoretic section $\realmap$, given by doing crossing changes along group elements. In particular, we can define an invariant of knots homotopic to $u$, namely
\[
    \Da_u\colon\KK(X,u)\coloneqq(\pi_0\incl)^{-1}[u]\sra\faktor{\Z[\pi_1X\sm1]}{\md_u(\pi_2X)}.
\]
    In \cite{K-Dax} the first author shows that this is the universal Vassiliev invariant of type $\leq1$ for knots in $X$. 
% One can also define $\realmap$ for $d=3$ as a set-theoretic map, see \cite{K-thesis}.
    
When $\partial u$ has a geometric dual, then $\pi_0\incl\colon\pi_1(X\cup_{\S^2}\D^3)\sra\pi_1X$ has trivial kernel $\KK(X;u)=0$, see Example~\ref{ex-intro:LBT-arcs}. Instead, there is a distinguished class $u_{tw}^G\in\pi_1\Emb_\partial(\D^1,X)$, given by ``swinging the lasso'' around the parallel push-off of the dual $G$ into $X$.

Now consider $\e$-augmented arcs for $d=3$ (equivalently, framed long knots). From~\eqref{diag:ribbons} we have an extension $\Z/\md^\e_{u^\e}(\pi_2X)\hra(\pi_0\ev_0)^{-1}[u]\sra\Z/2$, with $\Z/\md^\e_{u^\e}(\pi_2X)\cong\ker(\pi_1^{rel,\e}/\im(\delta_{\Imm^\e})\sra \pi_1^{rel}/\im(\delta_{\Imm}))$ is generated by the crossing change along $1\in\pi_1X$, and $\Z/2\cong\ker(\pi_1pr_2)$ by \eqref{eq:Vl+1X}. 

Interestingly, there are now only two distinct cases for this extension, depending on whether $\partial u$ has a geometric dual or not: $u$ has respectively either exactly two framings -- that is, $(\pi_0\ev_0)^{-1}[u]\cong\Z/2$, or countably many, $(\pi_0\ev_0)^{-1}[u]\cong\Z$. 
The first case is immediate from $\md^\e_{u^\e}(G)=1$ as in Example~\ref{ex:computing-md}. To see this more explicitly, the mentioned loop $u_{tw}^G$ can be extended to a path of $\e$-augmented arcs, whose start and end framings on~$u$ differ by $2$, so $\delta_{\ev_0}(u_{tw}^G)=2$. This is precisely the well-known light bulb trick for framed knots! To prove that in cases without a dual there is $\Z$ many framings, one approach would be to show that $1\in\Z[\pi_1X]/\md^\e_{u^\e}(\pi_2X)$ has infinite order; see \cite{framings2014} for another proof.

%%%%%%%%%%%%%%%%%%%%%%%%%%%%%%%%%%
\subsection{Homotopy groups of spaces of disks with a dual: proof of Theorem~\ref{thm-intro:combined}} \label{subsec:proof-of-ThmD}

We collect the results obtained so far in order to prove Theorem~\ref{thm-intro:combined}, concerning the space $\Emb_s(\D^k,M)$ of neat embeddings of the $k$-disk in a $d$-manifold $M$ such that $d-k\geq 2$, with the boundary condition $s\colon\S^{k-1}\hra \partial M$, which has a {framed geometric dual} $G\colon\S^{d-k-1}\hra\partial M$.

Firstly, Theorem~\ref{thm-intro:LBT-deloop} gives for all $n\geq0$ explicit ``ambient isotopy'' and ``$\e$-foliation'' isomorphisms
\[\begin{tikzcd}
    \pi_n\foliate^\e_\U\colon\pi_n\Emb_s(\D^k,M)\arrow[shift left=5pt]{r}[swap]{\cong}{} &\pi_{n+1}(\Emb_{u_0^\e}^\e(\D^{k-1},M_G),u_+^\e) \arrow[shift left=5pt]{l}{}\colon\pi_n\amb_\U
  \end{tikzcd}
\]
depending on the choice of a basepoint $\U\in\Emb_s(\D^k,M)$ (recall $M_G\coloneqq M\cup_{\nu G} h^{d-1}$ and $s=u_-\cup u_+$).
Secondly, Theorem~\ref{thm:Dax-augmented} gives for $n\leq d-2\ell-2$ isomorphisms $\pi_{n}(\Emb_{u_0^\e}^\e(\D^\ell,X),u^\e)\cong\pi_{n+1}X$ and an extension on $\pi_{d-2\ell-1}$ by a quotient of $\Z[\pi_1X]$. 
Combining these two results by putting $X\coloneqq M_G$ and $u\coloneqq u_+$ and $\ell=k-1$, we obtain isomorphisms
\[
    \pi_{n}\Emb_s(\D^k,M)\cong\pi_{n+2}M_G,\quad\text{for }n\leq d-2(k-1)-3=d-2k-1,
\]
and the extension:
% \begin{equation*}
%     \begin{tikzcd}[column sep=3.1em,row sep=tiny]
%         \faktor{\Z[\pi_1M_G]}{\langle1,\md(\pi_{d-1}M_G)\rangle} \arrow[tail]{rr}{\pi_{d-4}\amb_\U\circ\partial\realmap^\e_{u_+}} 
%         && \pi_{d-4}\Emb_s(\D^2,M) \arrow[two heads]{rrr}{(\eta_W\circ\deriv_{u^\e})\oplus p_{u_+})\circ\pi_{d-4}\foliate^\e_\U}
%         &&& \Z\oplus\pi_{d-2}M_G,
%     \\
%         \faktor{\Z[\pi_1M_G]}{\md^\e(\pi_{d-1}M_G)} \arrow[tail]{rr}{\pi_{d-4}\amb_\U\circ\partial\realmap^\e_{u_+}}
%         && \pi_{d-4}\Emb_s(\D^2,M) \arrow[two heads]{rrr}{((\eta_W\circ\deriv_{u^\e})\oplus p_{u_+})\circ\pi_{d-4}\foliate^\e_\U}
%         &&& \Z/2 \oplus\pi_{d-2}M_G.
%     \end{tikzcd}
% \end{equation*}
\[\begin{tikzcd}[column sep=large,ampersand replacement=\&]
\begin{drcases}
    k=2, &\faktor{\Z[\pi\sm1]}{\md_u(\pi_{d-1}M_G)}\\
    k\geq3, d-k\text{ even}, & \faktor{\Z[\pi\sm1]}{\langle g+g^{-1}\rangle\oplus\md_u(\pi_{d-k+1}M_G)}\\
    k\geq3, d-k\text{ odd}, & \faktor{\Z[\pi]}{\langle g-g^{-1}\rangle\oplus\md^\e_{u^\e}(\pi_{d-k+1}M_G)}
\end{drcases} \arrow[tail]{r}{\pi_{d-2k}\amb_\U\circ\partial\realmap^\e} 
        \& \pi_{d-2k}(\Emb_s(\D^k,M),\U)
        \arrow[two heads]{d}[description]{(\eta_W\circ\deriv_{u^\e}\oplus p_u)\circ\pi_{d-2k}\foliate^\e_\U}\\
        \& \hspace{-2cm}
\begin{drcases}
     d-k\text{ even}, & \Z\\
     d-k\neq1,3,7\text{ odd}, & \Z/2
\end{drcases}\oplus\pi_{d-k}M_G.\hspace{0.6cm}
     \end{tikzcd}
\]
In Theorem~\ref{thm-intro:combined} we have such an extension (we stated it exists for $d-k\neq1,3,7$ since then we do not have an explicit description of the quotient) but only in terms of the original manifold $M$, so we now remove all appearances of $M_G$. Moreover, our $u_+$ is homotopic into $\partial M$ in which case we simply write
\begin{equation}\label{eq:convention-md}
    \md\coloneqq\md_{u_+},\quad\text{ and }\quad\md^\e\coloneqq\md^\e_{u_+^\e}.
\end{equation}

\begin{lemma}\label{lem:lambda-splits}
    The inclusion $M\subseteq M_G$ induces isomorphisms $\pi_iM\cong\pi_iM_G$ for all $0\leq i\leq d-k-1$, a surjection $\pi_{d-k+1}M\sra\pi_{d-k+1}M_G$, and a split short exact sequence of $\Z[\pi_1M]$-modules
    \[\begin{tikzcd}
        \Z[\pi_1M]\arrow[tail,shift left]{r}{\cdot G} 
        & \pi_{d-k}M\arrow[two heads]{r}\arrow[shift left,dashed]{l}{\lambda^{rel}_M(\U,\bull)} 
        & \pi_{d-k}M_G,
    \end{tikzcd}
    \]
    where $\lambda^{rel}_M\colon\pi_k(M,\partial M)\times\pi_{d-k}(M)\to\Z[\pi_1M]$ is the relative equivariant intersection form.
\end{lemma}
\begin{proof}
    A $(d-k+1)$-handle is homotopy equivalent to a $(d-k+1)$-cell, so we immediately get $\pi_iM\cong\pi_iM_G$ below degree $d-k$. Moreover, the relative homotopy group $\pi_{d-k+1}(M_G,M)$ is the free $\Z[\pi_1M]$-module spanned by $h^{d-k+1}$. Once we show that homomorphism $\lambda^{rel}_M(\U,\bull)$ is a splitting, the surjectivity on $\pi_{d-k+1}$ will follow from the long exact sequence of a pair. Indeed, since $G$ is the geometric dual for $s=\partial\U$, we have $\lambda_{\partial M}(s,G)=1$, so a push-off of $G$ intersects $\U$ in the interior with $\lambda^{rel}_M(\U,G)=1$.
\end{proof}

\begin{lemma}\label{lem:Dax-M-M_G}
    We have $\md(\pi_{d-k+1}M_G)=\md(\pi_{d-k+1}M)$ as subgroups of $\Z[\pi_1M]$. Similarly for $\md^\e$.
\end{lemma}
\begin{proof}
    We need show that $\im(\md_M)=\im(\md_{M_G})$ for the respective $\md$ maps. The diagram
    \[\begin{tikzcd}[column sep=large]
          \pi_{d-k+1}M\arrow[two heads]{r}{p}\arrow[bend right=10pt]{rr}[swap]{\md_{M}} & \pi_{d-k+1}M_G\arrow{r}{\md_{M_G}} & \Z[\pi_1M]
        \end{tikzcd}
    \]
    commutes, since attaching a handle to $\partial M$ does not influence the calculation of $\md_{M}([f])\coloneqq\Da(\wt{F})$. Indeed, if $f\colon\S^{d-k+1}\to M$ is represented by a $(d-2k)$-family $F(\vec{t})\colon\D^{k-1}\imra M$, the same family also computes $\md_{M_G}([f])$. This immediately implies $\im\md_M\subseteq\im\md_{M_G}$. The other inclusion follows since $p$ is surjective (by Lemma~\ref{lem:lambda-splits}): if $r=\md_{M_G}(a)$ for $a\in\pi_{d-k+1}M$, then $r=\md_M(b)$ for $b=p(a)$. The argument is the same for $\md^\e$.
\end{proof}

\begin{proof}[Proof of Theorem~\ref{thm-intro:combined}]
    By the last two lemmas, we may replace $\pi_{d-k}M_G$ by $\pi_{d-k}M/\Z[\pi_1M]\cdot G$ and $\md(\pi_{d-k+1}M_G)$ by $\md(\pi_{d-k+1}M)$, and $\md^\e(\pi_{d-k+1}M_G)$ by $\md^\e(\pi_{d-k+1}M)$. It thus only remains to see that the maps are as claimed in Theorem~\ref{thm-intro:combined}, namely
    \[
        (\eta_W\circ\deriv_{u^\e}\oplus p_{u_+}) \circ\pi_{d-4}\foliate^\e_\U=\eta_{W,\U}\oplus (-\U\cup\bull).
    \]
    For $K\colon\S^{d-2k}\to\Emb_s(\D^k,M)$ the map  $\foliate^\e_\U(K)\colon\S^{d-2k+1}\to\Emb_{u_0^\e}^\e(\D^{k-1},M_G)$ maps $\vec{t}\wedge t\in\S^{d-2k}\wedge\S^1\cong\S^{d-2k+1}$ to $\foliate^\e_\U(K_{\vec{t}})(t)$, the time $t$ of the foliation of the sphere $-\U\cup K_{\vec{t}}$ (use the canonical null homotopy of $-\U\cup\U$ to get the map on the smash product). Then, $p_{u_+}([\foliate^\e_\U(K)])$ is the homotopy class of the map that takes $\vec{t}\in\S^{d-2k+1}$ to $-u_+\cup_\partial\foliate^\e_\U(K_{\vec{t}})(t)\in\Omega M_G$ (based at $u_+(e_{\D^{k-1}})$), see the discussion after Theorem~\ref{thm:Dax-augmented}. It is not hard to see that $-u_+$ is inessential, i.e.\ this is homotopic to $-\U\cup K\colon \S^{d-2k}\to \Map_*(\S^k,M_G)$, so $(p_{u_+} \circ\pi_{d-2k}\foliate^\e_\U)(K)=[-\U\cup K]$ modulo $\Z[\pi_1M]G$.
    
    For $d-k$ even we next identify the composite $\eta_{W,\U}\coloneqq\eta_W\circ\pi_{d-2k+1}\deriv_{u_+^e}\circ\pi_{d-2k}\foliate^\e_\U$ for the splitting $\eta_W\colon\pi_{d-k}V_k(M_G)\to\Z$ constructed in Proposition~\ref{prop:splittingsZ}. Unless $d-k=k=2,4,8$, it was given by $\eta_W=\mathbf{e}^d_k/2$ where $\mathbf{e}^d_k([f])=\langle e(f^*(TM_G)_{d-k}),[\S^{d-k}]\rangle\in2\Z$ (see Definition~\ref{def:Euler}), while for $d-k=k=2,4,8$ we also have the correction term $W[f]$. 
    %If $[f]=\pi_{d-2k+1}\deriv_{u_+^e}\foliate^\e_\U[K]$ then $pr_k[f]$ is just $p_{u_+}([\foliate^\e_\U(K)])=-\U\cup K$. Thus, the correction term is $\lambda_{M_G}(-\U\cup K,-\U\cup K)$.
    % $W_{M_G}(f)/2$ with $W_{M_G}\colon\pi_2M_G\to\Z$ a lift of $w_2(TM_G)$. Using Lemma~\ref{lem:lambda-splits} we can lift this to $W\colon\pi_2M\to\Z$ by defining $W(\Z[\pi_1M]\cdot G)=0$, so that $W_{M_G}([\foliate^\e_\U(K)])=W(-\U\cup K)$. 
    The following lemma then finishes the proof of Theorem~\ref{thm-intro:combined}.
\end{proof}
\begin{remark}
     The homomorphism $\pi_{d-2k}\amb_\U\circ\partial\realmap^\e_{u_+}$ can be made explicit: its value on $g\in\pi_1X$ is the family of $k$-disks in~$M$ obtained by applying to the half-disk $\U$ in $M_G$ the ambient isotopy extended from the family $\partial\realmap^\e_{u_+}(g)$. Alternatively, one can guess a geometric candidate $f_g\in\pi_{d-2k}\Emb_s(\D^2,M)$ and check that its foliation has the correct Dax invariant, that is, $\Da\circ\foliate^\e_\U(f_g)=g$, so $f_g=\pi_{d-2k}\amb_\U\circ\partial\realmap^\e_{u_+}(g)$. The latter approach is carried through for $k=2,d=4$ in \cite{KT-4dLBT}. Moreover, in these cases we also compute there the kernel and cokernel of $-\U\cup\bull\colon\pi_{d-2k}\Emb_s(\D^k,M)\to\pi_{d-k}M$.
\end{remark}
\begin{lemma}\label{lem:e}
    The number $\mathbf{e}^d_k\circ\deriv_{u_+^e}\circ\foliate^\e_\U([K])\in\Z$ is equal to the relative Euler number $e(\nu\wt{K},\nu\wt{\U})$ of the normal bundle of the immersion $\wt{K}\colon \I^{d-2k}\times\D^k \imra \I^{d-2k}\times M$ given by $(\vec{t},x)\mapsto (\vec{t},K_{\vec{t}}(x))$, relative to the immersion $\wt{\U}$ corresponding to the constant family $\U$.
        In particular, the map $\pi_{d-2k}(\Emb_s(\D^k,M),\U)\to\Z$, given by $[K]\mapsto e(\nu\wt{K},\nu\wt{\U})$, is a homomorphism.
\end{lemma}
\begin{proof}
    The normal bundle to $\wt{K}$ consists of vectors $(\vec{0},v_x)$ where $\vec{0}\in T_{\vec{t}}(\I^{d-2k})$ and $v_x$ is a normal direction to $K_{\vec{t}}$ at $x$ in $M$. We need to compute the Euler class of the bundle $f^*(TM_G)_{d-k}$ over $\S^{d-k}$, where $f=\deriv_{u_+^\e}\foliate^\e_\U(K)\colon\S^{d-k}\to V_k(M_G)$. This is obtained by gluing together two maps $\I^{d-k-1}\to \Omega V_k(M_G)$, namely $f_K(\vec{t})=\deriv_{u_+^\e}\foliate^\e(K_{\vec{t}})$ and the constant family $f_{-\U}(\vec{t})=\deriv_{u_+^\e}\foliate^\e(-\U)$. 
    
    First observe we can disregard $u_+^\e$ as before, so that $f_K\colon\I^{d-2k}\times\D^k\to V_k(M_G)$ is at $(\vec{t},x)$ given by the derivative at $x\in\D^k$ of the embedded disk $K_{\vec{t}}$ in $M_G$. Then the bundle $f_K^*(TM_G)_{d-k}$ is by definition given at a point $(\vec{t},x)$ as the subspace of $T_pM_G$, for $p=K_{\vec{t}}(x)$, orthogonal to the derivative $\deriv_xK_{\vec{t}}$, so belongs to the normal bundle of $K_{\vec{t}}$ in $M_G$. The same is true for the constant family $\vec{t}\mapsto-\U$ in place of $K$, and they agree on the boundary $\partial(\I^{d-2k}\times\D^k)$. Moreover, these normal bundles can be taken in $\I^{d-2k}\times M$ instead. Thus, the Euler number of $f^*(TM_G)_{d-k}$ is precisely the relative Euler number of the normal bundles to the immersions $\wt{K}$ and $\wt{\U}$.
\end{proof}

% \begin{figure}[!htbp]
%     \centering
%     \includegraphics[width=0.78\textwidth]{Figures/ambExp1.png}
%     caption{The family $\Exp_u(1)\colon\D^l\times[0,\e]\hra X$ for $l=1$ and $d=4$. The wiggly curves describe a family of push-offs $\Exp_u(1)(\bull,\e)$ of the dashed arc $u=\Exp_u(1)(\bull,0)$, that is contained in the present $u\subseteq\D^3\times\{0\}\subseteq\D^d$. For the two vectors which are parallel to $u$ parts of the push-off are respectively in the past and future. The $2$-disk $\U_{tw}^G(\vec{t})$ is the union of the top right strip, wiggly arcs, and the rest of $\U$ in bottom left.}
%     \label{fig:ambExp(1)}
% \end{figure}
\begin{remark}\label{rem:amb-Exp}
    For $d-k$ even we describe a map $\U_{tw}^G\coloneqq\amb_\U\Exp(1)\colon\S^{d-2k}\to\Emb_s(\D^k,M)$ that splits off the $\Z$ factor, $\eta_{W,\U}(\U_{tw}^G)=1$. Firstly, $\Exp(1)$ is the $(d-2k+1)$-family of augmentations of the $(k-1)$-disk $u_+$ obtained by integrating the normal vector field given by its meridian $\mu(\S^{d-k-1})$, see Remark~\ref{rem:Exp-pic}.
    Applying the ambient isotopy map $\amb_\U$ to this gives a family $\U_{tw}^G(\vec{t})=\amb_\U(\Exp(1)(\vec{t}\wedge-))$, supported in a neighborhood of $\U$; for a fixed $\vec{t}\in\S^{d-2k}$ the ambient isotopy pushes $\U$ around $\mu(\vec{t}\wedge-)$, so $\U_{tw}^G$ is obtained by a ``family interior twist'' to $\U$, plus tubing the unique double point at $p$ into $G$. See \cite[Fig.5.9]{KT-4dLBT} for the $2$-disk $\U_{tw}^G$ when $d=4$ (in other dimensions this is one of the disks in the family $\U_{tw}^G$).
\end{remark}

\appendix

%%%%%%%%%%%%%%%%
\section{On Hurewicz fibrations}\label{app:Hurewicz}
Recall that a map $p\colon E\to B$ is a Hurewicz fibration if any homotopy lifting problem as in
    \begin{equation}\label{eq:htpy-lifting-problem}
    \begin{tikzcd}
        X\times\{0\}\arrow[]{r}{H_0}\arrow[tail]{d}  & E\arrow[]{d}{p}\\
        X\times[0,1]\arrow[]{r}[swap]{h}\arrow[dashed]{ur}{H} & B
    \end{tikzcd}
\end{equation}
has a solution $H$, a so-called ``lift''. We will need some properties of such lifts.
\begin{lemma}\label{lem:uniqueness}
    Up to homotopy rel.\ $X \times \{0\}$ and over $B$, the lift $H$ as in the diagram \eqref{eq:htpy-lifting-problem} is uniquely determined by the initial conditions $(h,H_0)$. As a consequence, its restriction $H_1\colon X \times \{1\}\to E$ is also unique up to homotopy.
\end{lemma}
\begin{proof}
    If $\Lambda^{k-1}\subseteq \partial\Delta^k$ is a $(k-1)$-horn, i.e.\ it consists of all $(k-1)$-faces of the k-simplex, except for one, then a Hurewicz fibration also has the lifting property for pairs $(X \times \Delta^k, X \times \Lambda)$, by the fact that these are acyclic cofibrations in the model structure on topological spaces for which the Hurewicz fibrations are the fibrations. The case $k=1$ is precisely the lifting problem~\eqref{eq:htpy-lifting-problem}, and we use the case $k=2$ to show the uniqueness of the lifted homotopy. 
    Namely, suppose we have two lifts $H$ and $H'$ in diagram \eqref{eq:htpy-lifting-problem}.
    Since they agree on $X \times \{0\}$ we can glue them together at the vertex $v_0\in \Delta^2$ to give a map $X \times \Lambda^1\to E$ that lifts (the restriction of) the map $\tilde h\colon X \times \Delta^2\to B$ that equals $h$ on all rays from $v_0$ to the opposite edge $ \langle v_1,v_2 \rangle\subseteq\Delta^2$. The lift $X \times \Delta^2\to E$ is thus a homotopy from $H$ to $H'$ rel.\ $X \times \{0\}$ and over $B$.
\end{proof}
In particular, consider in \eqref{eq:htpy-lifting-problem} the space $X\coloneqq \Omega B$ of loops based at $b\coloneqq p(e)$ for a basepoint $e\in E$, and the initial conditions $H_0=\const_e\colon\beta\mapsto e$ and $h=\ev\colon(\beta, t)\mapsto\beta(t)$. Any lift $\wh{\ev}\coloneqq H\colon\Omega B\times[0,1]\to E$ at time $t=1$ takes values in the fiber $F\coloneqq p^{-1}(b)$, hence gives a so-called {\em connecting map} $\delta\coloneqq\wh{\ev}_1\colon\Omega B\to F$.
\begin{cor}\label{cor:connecting map}
    Up to homotopy, the map $\delta\colon\Omega B\to F$ is independent of the choice of a lift $\wh{\ev}$ and only depends on the Hurewicz fibration $p$ and the basepoint $e\in E$. Moreover, it is natural:
    given two commuting squares on the right of the diagram
    \[\begin{tikzcd}[column sep=large]
        \Omega B' \arrow{r}{\delta'} \arrow{d}{\Omega(g_B)}
        &  F'\arrow[tail]{r} \arrow{d}{g_{\mid F}}
        & E'\arrow{r}{p'} \arrow{d}{g}
        & B' \arrow{d}{g_B}
        \\
        \Omega B \arrow{r}{\delta} 
        & F\arrow[tail]{r} 
        & E\arrow{r}{p} 
        & B 
    \end{tikzcd}
    \]
    any choice of connecting maps $\delta, \delta'$ makes the square on the left commute up to homotopy.
\end{cor}
\begin{proof}
    The first part follows directly from Lemma~\ref{lem:uniqueness}.
    The naturality follows from it as well, this time applied to lifting $g_B\circ \ev'\colon\Omega B' \times [0,1]\to B$ with the initial condition $g\circ \const_{e'}$.
\end{proof}
If $B$ is well-based the above discussion holds in the based category; this is the case for our spaces of embeddings as they are locally contractible. Then $\Omega B$ is well-based at~$\const_b$ and there are based connecting maps $\delta\colon \Omega B\to F$, inducing the boundary maps in the long exact sequence of homotopy groups for the fibration $p$.
In particular, if $E$ is contractible then any such $\delta$ is a weak homotopy equivalence. We will also need the following strengthening, whose proof we did not find in the literature.
\begin{lemma}\label{lem:Hurewicz}
     If $E$ is contractible, a connecting map $\delta$ is a homotopy equivalence with homotopy inverse $(p_*R)|_{F}\colon F\to \Omega B$ given by $x\mapsto(t\mapsto p\circ R_t(x))$, where $R\colon E\to \Path_eE\coloneqq\{\eta\colon[0,1]\to E \mid  \eta(0)=e\}$ is a contraction of~$E$ with $\ev_1\circ R=\Id_E$.
\end{lemma}
%  Note that if $B$ is well-based $\delta$ is a based homotopy equivalence, but $pR|_F$ is based only if $R$ is.
\begin{proof}
   Consider the following diagram of two fibrations and a lifting problem on top:
\begin{equation}\label{eq:two-fibrations}
    \begin{tikzcd}
    \Omega B \times \{1\}  \arrow{d}{\delta} & 
    \Path_bB \times \{0\}\arrow[tail]{r} \arrow{d}[swap]{\const_e} &
    \Path_bB \times [0,1] \arrow{d}{\ev} \arrow[dashed]{dl}[swap]{\wh{\ev}}
    \\
    F\arrow[tail]{r} \arrow{d}{p_*R|_F} & 
    E\arrow{r}{p} \arrow{d}[swap]{p_*R} & 
    B \arrow[equals]{d}{}
    \\
    \Omega B\arrow[tail]{r} & 
    \Path_bB\arrow{r}{\ev_1} & 
    B.
    \end{tikzcd}
\end{equation}
    As a map between contractible spaces, $p_*R\coloneqq(t\mapsto p\circ R_t)$ is a homotopy equivalence. Being over $B$ ($\ev_1p_*R=p$), this is also a fiber homotopy equivalence \cite[4.H]{Hatcher}. 
    For the same reason, the restriction $q\coloneqq\wh{\ev}_1\colon\Path_bB\to E$ of the lift $\wh{\ev}$ is a fiber homotopy equivalence over $B$. We claim that \emph{there is a homotopy $p_*R \circ q \simeq \Id$ over $B$}, so uniqueness of inverses will imply that $p_*R$ and $q$ are inverse homotopy equivalences, as well as the desired restrictions $p_*R|_F$ to $F\subseteq E$ and $\delta=q|_{\Omega B}$ to $\Omega B\subseteq \Path_bB$.
    
    To find the claimed homotopy we observe two solutions for the outer lifting problem on the right of~\eqref{eq:two-fibrations} (for the fibration $\ev_1$). One lift is clearly $p_*R\circ\wh{\ev}$. We define the second by the formula
    \[(\beta, t)\mapsto\left(s\mapsto\begin{cases}
        p\circ \rho_t(\frac{s}{1-t}), & s\in[0,1-t]\\
        \beta(s-(1-t)), & s\in[1-t,1]
        \end{cases}\right),
    \]
    where $\rho_t\in\Omega E$ for $t\in[0,1]$ is a homotopy from $\rho_0=R(e)$ to $\rho_1=\const_e$, which exists by the argument below. 
    This is indeed another solution, since $\ev_1$ evaluates all paths at $s=1$, so it gives $\beta(t)=\ev(\beta,t)$, while restricting to $t=0$ gives exactly $p\rho_0=pR(e)$. The first lift at $t=1$ is $p_*R\circ q$ whereas the second is $\Id_{\Path_bB}$, so the uniqueness from Lemma~\ref{lem:uniqueness} gives the desired homotopy over $B$ between them.

    To find a homotopy $\rho_t$ from $R(e)\in\Omega E$ to $\const_e$ rel.\ $e$, first observe that these loops are homotopic (since $E$ is contractible to $e$), and then by the usual argument they are also \emph{based} homotopic: free homotopy classes in a path-connected space are in bijection with the conjugacy classes in the fundamental group, and for the unit this consists of a single element.
\end{proof}

\section{On homotopy groups of frame bundles}\label{app:frame-bundles}
In this appendix we collect some information needed for the results in the body of the paper. For the convenience of the reader we include short arguments, some of which may be new.

\subsection{Universal Euler classes of frame bundles}
Fix integers $0\leq \ell\leq d$, a topological space $X$, and a rank $d$ vector bundle $\xi$ on $X$ with inner product. Let $\forg_\ell\colon V_\ell(\xi)\to X$ be the associated $\ell$-frame bundle, given by orthonormal sequences $(v_1,\dots,v_\ell)$ of vectors $v_i\in\xi|_x$, $x\in X$, meaning that $v_i$ are pairwise orthogonal and have unit norm. The pullback bundle $\forg_\ell^*(\xi)$ has $\ell$ canonical sections, and we consider their orthogonal complement $\xi_{d-\ell}$, which is a rank $(d-\ell)$ vector bundle on $V_\ell(\xi)$. Note that $V_0(\xi)=X$ and $\forg_0=\Id$.
  
We will not assume that $\xi$ is orientable in this appendix, although this will be the case in our applications. Recall that one can define the twisted Euler class of $\xi_{d-\ell}$ as a class $e(\xi_{d-\ell})\in H^{d-\ell}(V_\ell(\xi);\Z^w)$,  where the coefficients are twisted by the first Stiefel--Whitney class $w\coloneqq w_1(\xi_{d-\ell})$.  
\begin{defn}\label{def:Euler} 
    Define the map
    $\mathbf{e}^d_\ell\colon\pi_{d-\ell}V_\ell(\xi)\ra\Z$
    by the formula:
\[
    \mathbf{e}^d_\ell(f) \coloneqq \langle e(\xi_{d-\ell}), f_*[\S^{d-\ell}] \rangle = \langle e(f^*(\xi_{d-\ell})), [\S^{d-\ell}] \rangle.
\]
\end{defn}
Note that the last expression shows that $\mathbf{e}^d_\ell$ is a group homomorphism. 

\begin{lemma}\label{lem:chi}
    For $\ell\geq1$ let $pr_\ell\colon V_\ell(\xi)\to V_{\ell-1}(\xi)$ be the fibration that forgets the last vector in an $\ell$-frame, and $i_\ell\colon\S^{d-\ell}\hra V_\ell(\xi)$ its fiber inclusion. 
    Then $\mathbf{e}^d_\ell(i_\ell)=\chi(\S^{d-\ell})$, the Euler characteristic of $\S^{d-\ell}$.
\end{lemma}
\begin{proof}
    This follows since $i_\ell^*(\xi_{d-\ell})\cong T\S^{d-\ell}$. Indeed, the fiber of $\xi_{d-\ell}$ over $(x,v_1,\dots,v_\ell)\in V_\ell(\xi)$ is the orthogonal complement $(v_1,\dots,v_\ell)^\perp\subseteq \xi_x$, so if $v_\ell=u\in\S(U)$ for a fixed $U\coloneqq (v_1,\dots,v_{\ell-1})^\perp\subseteq\xi_x$ then the fiber is $u^\perp\subseteq U$, so exactly the tangent space to the $(d-\ell)$-sphere $\S(U)\subseteq U$ at $u$.
\end{proof}

By induction and the long exact sequence of homotopy groups for the fibration $pr_{\ell+1}\colon V_{\ell+1}(\xi)\to V_\ell(\xi)$ we get isomorphisms
\begin{equation}\label{eq:fr-below-critical}
    \begin{tikzcd}
        \pi_nV_\ell(\xi)\rar{\forg_{\ell}}[swap]{\cong} & 
        \pi_nX
   \end{tikzcd}
\end{equation}
for all $n\leq d-\ell-1$, and an exact sequence in the first interesting case
\begin{equation}\label{eq:pr-seq4}
    \begin{tikzcd}
        \pi_{d-\ell}V_\ell(\xi)\rar{\delta_{pr_{\ell+1}}} & 
        \pi_{d-\ell-1}\S^{d-\ell-1}\rar{i_{\ell+1}} &
        \pi_{d-\ell-1}V_{\ell+1}(\xi)\rar[two heads]{\forg_{\ell+1}} &
        \pi_{d-\ell-1}X.
    \end{tikzcd}
\end{equation}

\begin{lemma}\label{lem:Euler}
    For $0\leq \ell\leq d$ the connecting map $\delta_{pr_{\ell+1}}$ of $pr_{\ell+1}$ can be identified with the homomorphism $\mathbf{e}^d_\ell\colon\pi_{d-\ell}V_\ell(\xi)\to\Z\cong \pi_{d-\ell-1}\S^{d-\ell-1}$.
\end{lemma}
\begin{proof}
    Since the statement is inherited by pullbacks of vector bundles, it suffices to check it for $\xi=\upgamma$, the universal rank $d$ bundle $\upgamma\to BO_d$ over the Grassmannian $BO_d$ of $d$-planes in $\R^\infty$. Then $EO_d\coloneqq V_d(\upgamma)$ is a contractible, free $O_d$-space and we have for $0\leq \ell\leq d$ homotopy equivalences
\[
    V_\ell(\upgamma) \simeq \faktor{EO_d}{O_{d-\ell}} \simeq BO_{d-\ell}.
\]
    The map $pr_{\ell+1}$ is induced by the inclusion $O_{d-\ell-1}\hra O_{d-\ell}$
    % and the isomorphism $\pi_{d-\ell-1}V_\ell(\upgamma)\cong\pi_{d-\ell-2}O_{d-\ell-1}$ takes $f\colon\S^{d-\ell-1}\to V_\ell(\upgamma)$ to the clutching map for the vector bundle $f^*(\forg_\ell^*\upgamma)$.
    and we have a 5-term fibration sequence, where the connecting map in question is induced by the action map $act$ on a basepoint in $\S^{d-\ell-1}$:
    \begin{equation}\label{eq:Euler}
    \begin{tikzcd}
        O_{d-\ell-1} \rar & O_{d-\ell} \arrow{r}{act} & \S^{d-\ell-1} \rar & V_{\ell+1}(\upgamma)\arrow{r}{pr_{\ell+1}} & V_\ell(\upgamma).
    \end{tikzcd}
    \end{equation}
    Recall that the Euler class is the unique (necessary and sufficient) obstruction for finding a nonvanishing section in an rank $n$ vector bundle over $\S^n$, and $\pi_{d-\ell-1}O_{d-\ell}$ are isomorphism classes of such bundles for $n=d-\ell$ (via clutching). Then the claim follows since $\pi_{d-\ell-1}(act)$ is also the unique obstruction for such a section by the exact sequence above. 
\end{proof}
As a consequence, for $0\leq\ell\leq d$ we get an extension of groups
 \begin{equation}\label{eq:pr-seq}
    \begin{tikzcd}
        \faktor{\Z}{\im(\ed_\ell)} \rar[tail]{i_{\ell+1}} &
        \pi_{d-\ell-1}V_{\ell+1}(\xi)\rar[two heads]{\forg_{\ell+1}} &
        \pi_{d-\ell-1}X
    \end{tikzcd}
\end{equation}
and we determine the image of the homomorphism $\ed_\ell$ next.  

If $d-\ell$ is odd then $\mathbf{e}_\ell^d=0$, because it is given by Euler classes of bundles $f^*(\xi_{d-\ell})\to\S^{d-\ell}$ of rank $d-\ell$, for which fiberwise $-\Id$ is an orientation reversing vector bundle isomorphism over $\Id$ on the base. 
    
If $\ell=0$ and $d$ is even then $\im(\ed_\ell)=e(\xi)(\pi_dX)$ is simply the image of the Euler-class $e(\xi)$ evaluated on $\pi_d(X)$ under the Hurewicz map. Then any image is possible, for example, take $X=\S^2$ and $d=2$. Then $e(\xi)[S^2]\in \Z$ can be any integer $n$ by taking $\xi$ to be the $n$-fold (complex) tensor product of the Hopf bundle. As a consequence $\im(\ed_0)=e(\xi)(\pi_2S^2) = n\cdot \Z$. Note that for higher genus surfaces $X$ we have $\pi_2X=0$ and hence this image always vanishes. This example can be generalized as follows.

\begin{lemma}\label{lem:V_1}
Let $X$ be a compact $d$-manifold.
\begin{itemize}
\item 
    If $X$ has nonempty boundary then $\ed_0=0$ and the sequence \eqref{eq:pr-seq} splits for $\ell=0$. In fact, there are isomorphisms 
    $\pi_nV_1(\xi)\cong \pi_n\S^{d-1}\oplus \pi_nX$ for all $n\geq 1.$
\item
    If $X$ is closed then $e(TX)(\pi_dX)$ is nontrivial if and only if $d$ is even and the universal cover of $X$ is a rational homology $d$-sphere. More precisely, $X$ is either a simply connected $\Q$-homology $d$-sphere (orientable) or a $\Q$-homology ball (nonorientable) with fundamental group $\Z/2$.
\end{itemize}
\end{lemma}
\begin{proof}
    If $X$ has nonempty boundary then $H^d(X;\Z^w)=0$ and hence $e(\xi)=0$. This means that we have a continuous section of $V_1(\xi)\to X$. As a consequence, the long exact sequence of the fibration turns into (split) short exact sequences on homotopy groups.
    
    For closed $X$ we consider the case $\xi=TX$. We may assume that $d$ is even, otherwise $\ed_0 =0$. If $\Sigma$ is a simply connected $\Q$-homology sphere then the degree 1 map $\Sigma\to \S^d$ is a rational homotopy equivalence and hence $\pi_d(\Sigma)\otimes\Q \cong\Q$. It follows that
    there is a map $\S^d\to \Sigma$ of nonzero degree and so  $e(T\Sigma)(\pi_d\Sigma)\neq 0$ since the Euler characteristic of $\Sigma$ is 2.  If $X$ is covered by $\Sigma$, it therefore also satisfies $e(TX)(\pi_dX)\neq 0$.
    
    Conversely, if this group is nontrivial then there is a map $\S^d\to X$ of nonzero degree. Let $\Sigma$ be the universal cover of $X$. Then this map lifts to a map $\S^d\to \Sigma$ of nonzero degree (which implies that this is a finite cover). Then $\Sigma$ must be a $\Q$-homology sphere, otherwise there would be a nontrivial rational cup product in its top cohomology which contradicts that all cup products are trivial for $\S^d$. This also implies that the only free action on $\Sigma$ can be by $\Z/2$ and that the quotient is a $\Q$-homology ball.
\end{proof}

Finally, we consider the case $d-\ell$ even with $l\geq1$.
We define the \emph{spherical Stiefel--Whitney class}
\[
    w^s_i(\xi)\colon\pi_iX\to \Z/2
\]
as the evaluation of the $i$-th Stiefel--Whitney class $w_i(\xi)\in H^i(X;\Z/2)$ on spherical classes, namely those in the image of the Hurewicz homomorphism $h\colon\pi_iX\to H_i(X;\Z)$.
\begin{prop}\label{prop:Euler}
    For $X,\xi$ and $1\leq \ell\leq d=\dim(\xi)$ as above, assume that $d-\ell$ is even.     
    Then $\mathbf{e}^d_\ell$ is either onto or has image $2\cdot\Z$. The latter happens if and only if the spherical Stiefel--Whitney class $w^s_{d-\ell}(\xi)\colon \pi_{d-\ell}X\to \Z/2$ vanishes.
    This homomorphism can be nontrivial only for $d-\ell=2$, $4$, or $8$.
\end{prop}
\begin{proof}
    If $d-\ell$ is even and $l\geq1$, then $\mathbf{e}_\ell^d(i_\ell)=\chi(\S^{d-\ell})=2$ by Lemma~\ref{lem:chi}, so $2\cdot\Z$ is contained in the image of $\mathbf{e}^d_\ell$. To decide whether this is the entire image, we only need to understand the Euler class modulo~2. But then it equals the Stiefel--Whitney class $w_{d-\ell}$ which is a \emph{stable} characteristic class and hence we can evaluate it on $\xi$ rather then $\xi_{d-\ell}$. This implies our claim regarding the spherical Stiefel--Whitney class.

    Recall that Adams' solution of the Hopf invariant 1 problem implies that rank $n$ vector bundles over $\S^n$ have even Euler class unless $n=2,4,8$. The universal complex, quaternion and octonian line bundles over the corresponding projective spaces have Euler number $1$, so for $d-\ell=2,4,8$ both cases can arise. 
\end{proof}
For $X=*$, the frame bundle $V_\ell(\xi)$ is just the Stiefel manifold $V_\ell(d)$, whereas for $X=BO_d$ and the universal rank $d$ bundle $\xi\coloneqq\upgamma$ over it we have $V_\ell(\upgamma)\simeq BO_{d-\ell}$ (see the proof of Lemma~\ref{lem:Euler}). Therefore, in these cases we get back the following results, the first of which is due to Stiefel~\cite{Stiefel}, whereas the second can be found in Kervaire's paper~\cite{Kervaire}.

\begin{cor}\label{cor:Stiefel}
    If $1\leq k\leq d$, the groups $\pi_nV_k(d)$ vanish for $n\leq d-k-1$ and there are isomorphisms
\[
    \pi_{d-k}V_k(d)\cong\Z_{k,d}\coloneqq\begin{cases}
        \Z&k=1\text{ or }d-k\text{ even},\\
        \Z/2&\text{otherwise},
    \end{cases}
\]
    generated by the fiber inclusion $i_k\colon\S^{d-k}\to V_k(d)$.
\end{cor}
\begin{proof}
    For $X=*$ we have $\pi_nV_k(d)=0$ for $n\leq d-k-1$ by \eqref{eq:fr-below-critical}, and $\pi_{d-k}V_k(d)\cong\ker\forg_k\cong\Z/\im(\mathbf{e}^d_{k-1})$ by \eqref{eq:pr-seq} for $l+1=k$. We saw after \eqref{eq:pr-seq} that for $d-k$ even this is infinite cyclic, as well as for $k=1$ since $\pi_dX=0$. For $d-k$ odd we trivially have $w^s_{d-\ell}(\xi)=0$, so Proposition~\ref{prop:Euler} implies $\im(\mathbf{e}^d_{k-1})=2\Z$.
\end{proof}

\begin{cor}\label{cor:BO}
    For $m\geq1$ there are isomorphisms $\pi_nBO_{m-1}\cong\pi_nBO_m$ for $n\leq m-2$, and a short exact sequence
\[\begin{tikzcd}[ampersand replacement=\&]
\begin{drcases}
    m\text{ odd}           & \Z\\
    m\neq2,4,8\text{ even} & \Z/2\\
    m=2,4,8                & 0
\end{drcases}
\rar[tail] \& \pi_{m-1}BO_{m-1}\rar[two heads] \&
\begin{cases}
    \Z & m\equiv1,5\pmod8,\\
    \Z/2 & m\equiv2,3\pmod8,\\
    0 & m\equiv0,4,6,7\pmod8.
\end{cases}
\end{tikzcd}
\]
\end{cor}
\begin{proof}
    The kernel was computed in Proposition~\ref{prop:Euler} for $m\coloneqq(d-\ell)$. Just recall that $\pi_mV_\ell(\upgamma)=\pi_mBO_m$ consists of isomorphism classes of all rank $m$ vector bundles over $\S^m$, including those with Euler number 1. 
    To compute the right hand side, note that by \eqref{eq:pr-seq} the group $\pi_{m-1}BO_m \cong \pi_{m-1}BO$ is in the stable range, so the result follows from  Bott periodicity.
\end{proof}

\subsection{Splittings for the first interesting homotopy groups}
For $X,\xi$ as above, let $k\coloneqq\ell+1$ with $2 \leq  k \leq d$. If $d-k$ is even (i.e.\ $d-\ell$ odd) there is a group extension
\begin{equation}\label{eq:pr-seqZ}
    \begin{tikzcd}
        \Z\rar[tail]{i_k} &
        \pi_{d-k}V_k(\xi)\rar[two heads]{\forg_k} &
        \pi_{d-k}X
    \end{tikzcd}
\end{equation}
as we showed below \eqref{eq:pr-seq}. On the other hand, when $d-k$ is odd, $k\geq2$, $d-k\neq1,3,7$, we know from Proposition~\ref{prop:Euler} that our extension is by $\Z/2$:
\begin{equation}\label{eq:pr-seqZ2}
    \begin{tikzcd}
        \Z/2\rar[tail]{i_k} &
        \pi_{d-k}V_k(\xi)\rar[two heads]{\forg_k} &
        \pi_{d-k}X.
    \end{tikzcd}
\end{equation}
We now ask in which circumstances these extensions split and how to construct such splittings.
\begin{prop}\label{prop:splittingsZ}
    If $2\leq k\leq d=\dim(\xi)$ with $d-k$ even, splittings $\eta\colon\pi_{d-k}V_k(\xi)\to\Z$ of the extension~\eqref{eq:pr-seqZ} are in bijection with integer lifts $W\colon\pi_{d-k}X\to\Z$ of the spherical Stiefel--Whitney class
    % $w^s_{d-k}(\xi)\colon\pi_{d-k}X\to\Z/2$, via
\[
    W\ \longleftrightarrow \ \eta_W\coloneqq\frac{1}{2}(\ed_k + W\circ\forg_k).
\]
    Moreover, $w^s_{d-k}(\xi)$ can be nontrivial only for $d-k=2$, $4$, or $8$, so apart from these cases $W=0$ gives the preferred splitting $\eta_0 = \frac{1}{2}\ed_k$.
\end{prop}
\begin{proof}
    To show that the claimed formula for $\eta_W$ is a splitting, first note that the division by 2 makes sense for any integer lift $W$ of $w^s_{d-k}(\xi)$. Both summands are homomorphisms and $\eta_W$ splits the inclusion $i_k$ because
    $\ed_k(i_k)=2$ and $\forg_k(i_k) = 0$.
    Conversely, if $\eta$ is any splitting then $(2\eta - \ed_k)(i_k)=2-2=0$, so $2\eta-\ed_k$ factors through a homomorphism $W_\eta\colon\pi_{d-k}X\to\Z$. This is an integer lift of $w^s_{d-k}(\xi)$.
    
    Moreover, two splittings $\eta$ differ by a unique homomorphism $u\colon\pi_{d-k}X\to \Z$ (composed with $\forg_k$). Given one integer lift $W$ of $w^s_{d-k}(\xi)$, we have a second lift $W+2u$, leading to another splitting $\eta_{W+2u} = \eta_W+u\circ\forg_k$.
    This shows the claimed 1-1 correspondence. The last claim follows from the solution of the Hopf invariant~1 problem as in Proposition~\ref{prop:Euler}.
\end{proof}

\begin{prop}\label{prop:splittingsZ2}
    If $2\leq  k \leq d=\dim(\xi)$ with $d-k\neq1,3,7$ odd, then there is a splitting $\eta\colon \pi_{d-k}V_k(\xi)\to\Z/2$ of \eqref{eq:pr-seqZ2} that only depends on the isomorphism class of the vector bundle $\xi_{d-k}$.
\end{prop}
\begin{proof}
    Note that for $d-k$ odd Euler numbers $\mathbf{e}^d_k=0$ are trivial, so the method of the preceding proof does not construct splittings. However, Corollary~\ref{cor:BO} gives for $d-k\neq1,3,7$ odd a short exact sequence
\[\begin{tikzcd}[ampersand replacement=\&]
    \Z/2\rar[tail] \& \pi_{d-k}BO_{d-k}\rar[two heads] \& \pi_{d-k}BO=
    \begin{cases}
    \Z/2 & d-k\equiv1\pmod8,\\
    0 & d-k\not\equiv1\pmod8.
    \end{cases}
\end{tikzcd}
\]
    Thus, in the case $d-k\not\equiv1\pmod8$, returning to $X$ and comparing to the universal bundle, we see that a splitting $\eta$ arises by the formula $\eta(f)\coloneqq f^*(\xi_{d-k})\in \pi_{d-k}BO_{d-k} \cong \Z/2$.
    On the other hand, if $d-k\equiv 1\pmod 8$ the displayed sequence splits, $\pi_{d-k}BO_{d-k}\cong\Z/2 \times \Z/2$, by Kervaire \cite{Kervaire} (this is the entry $r=-1$ and $m=8s+1=d-k$ in his second table, with $s\geq 1$, which is also our case since $d-k\neq1$). In particular, we get our splitting that again only depends on the isomorphism class of $\xi_{d-k}\to V_k(X)$ (but in these cases we do not know an explicit formula).
\end{proof}

%%%%%%
\subsection{Tangent bundles}

We now improve the above results in the case $\xi=TX$ is the tangent bundle of a manifold. For $d-k$ odd this follows from \eqref{eq:pr-seqZ2}, Proposition~\ref{prop:splittingsZ2} and Lemma~\ref{lem:TX-odd} below. For $d-k$ even this follows from \eqref{eq:pr-seqZ}, Proposition~\ref{prop:splittingsZ} and Lemma~\ref{lem:TX-even}. In Lemma~\ref{lem:V_1} we looked at $k=1$.
\begin{prop}\label{prop:TX}
    For a compact $d$-manifold $X$ let $V_k(X)\coloneqq V_k(TX)$ be its $k$-frame bundle. For any $2\leq k\leq d$ except possibly for $k=d-1\geq 3$ or $k=d-3\geq 5$ or $k=d-7\geq 9$, there is an extension
\[\begin{tikzcd}[ampersand replacement=\&]
    \begin{drcases}
        %k=1, & \faktor{\Z}{e(TX)(\pi_dX)} \\
       d-k\text{ even}, & \Z\\
    d-k\text{ odd}, &\Z/2
    \end{drcases}
    \eqqcolon
    \Z_{k,d}
    \rar[tail]{i_k} \& \pi_{d-k}V_k(X)\rar[two heads]{\forg_k} \&\pi_{d-k}X.
\end{tikzcd}
\]
    These extensions split and for $d-k$ even splittings are given by
    \[
    \eta_W=\frac{1}{2}(\ed_k+W\circ\forg_k),
    \]
    where one can take $W=0$ unless $d-k=2,4$ or $8$ and $d\leq 2k$. If $d-k=2,4$ or $8$ and $d=2k$ an integral lift of $w_{d-k}^s(TX)$ exists and gives a desired $W$.
\end{prop}

\begin{lemma}\label{lem:TX-odd} 
    Assume $k\geq2$ and $d-k$ is odd. The homomorphism $\ed_{k-1}(TX)$ has image $2\cdot\Z$ unless $k=d-1\geq 3$ or $k=d-3\geq 5$ or $k=d-7\geq 9$.
    % Therefore, except in these cases, the split extension \eqref{eq:pr-seqZ2} holds for $\pi_{d-k}V_k(X)$.
\end{lemma}
\begin{proof}
    By Proposition~\ref{prop:Euler} the image of $\ed_\ell(\xi)$ is $2\cdot\Z$ if and only if the spherical Stiefel--Whitney class $w^s_{d-\ell}(\xi)$ vanishes, which is true unless $d-\ell=2,4,8$. Thus, we just need to show that for $\xi=TX$ we additionally must have $d\leq2\ell$ with $\ell=k-1$ for $w^s_{d-\ell}(TX)$ not to vanish (since then $d-k=1,3$ or $7$, and $d\leq 2k-2$ gives $k\geq d-k+2=3,5$ or $9$ respectively).

    If $v_j\in H^j(X;\Z/2)$ are the Wu classes of $X$ and $Sq^i$ denotes Steenrod squares, the Wu formula says
\[
    w_{d-\ell}(TX) = \sum_{i+j=d-\ell} Sq^i(v_j).
\]
    Since the cohomology of $\S^{d-\ell}$ is concentrated in two dimensions, pulling back to it kills all these summands, except possibly $Sq^0(v_{d-\ell})=v_{d-\ell}$.
    For any $a\in H^\ell(X,\partial X;\Z/2)$, by the defining property 
    % of Wu classes we have
\[
    \langle v_{d-\ell}\cup a, [X] \rangle = \langle Sq^{d-\ell}(a), [X] \rangle.
\]
    This vanishes on compact $d$-manifolds $X$ such that $d>2\ell$, since $Sq^i$ is zero on cohomology classes in degrees $<i$. So $w_{d-\ell}(TX)$ vanishes on spherical classes unless $d\leq2\ell$.
\end{proof}

\begin{lemma}\label{lem:TX-even} 
    Assume $k\geq2$ and $d-k$ is even.
    The extension \eqref{eq:pr-seqZ} for $\pi_{d-k}V_k(X)$ splits also in the cases $d-k=2,4$ or $8$ if $d\geq2k$. 
\end{lemma}
\begin{proof}
    The previous proof shows that $w^s_{d-k}(TX)$ vanishes unless $d\leq2k$, so $\eta_W$ with $W=0$ is a splitting unless $d\leq 2k$. Thus, it remains to show that a lift $W$ of $w^s_{d-k}(TX)$ also exists in the middle-dimensional setting $d=2k$ with $k=d-k=2,4$ or $8$. The following argument actually works for any $k$.

    Consider the \emph{twisted intersection form} of $X$, a bilinear pairing $\langle-,-\rangle_X\colon H_k(X;\Z)\times H_k(X;\Z^w)\to \Z$, where the second term denotes coefficients twisted by $w\coloneqq w_1(TX)$. The pairing is given algebraically by Poincar\'e duality and cup products and geometrically by counting transverse intersection points carefully with signs. 
    
    There are Hurewicz homomorphisms $h\colon\pi_kX\to H_k(X;\Z)$ and their twisted partners
\[
    h^w\colon\pi_kX\cong\pi_k\wt X \overset{\wt h}{\ra} H_k(\wt X;\Z) \cong H_k(X;\Z[\pi_1X])\overset{\e^w_*}{\ra} H_k(X;\Z^w),
\]
    where $\e^w\colon \Z[\pi_1X]\to \Z^w$ is the $\pi_1X$-linear map determined by $\e^w(g)\coloneqq w(g)\in\{\pm 1\}$ for all $g\in\pi_1X$. 
    In this middle dimension, any class $a\in \pi_kX$ is represented by a generic immersion $A\colon\S^k\imra X$ and 
\begin{equation}\label{eq:lambda-mu-formula}
   \langle h(a), h^w(a) \rangle_X \eqqcolon \langle a,a\rangle_X = 2\langle A\rangle_X + e(\nu_A),
\end{equation}
    where $\langle A\rangle_X$ is the self-intersection number (the sum of signed double points of A) and $e(\nu_A)$ is the Euler number of the normal bundle $\nu_A\to\S^k$. The formula is proven by intersecting $A$ with a transverse push-off $A^\uparrow$ of $A$, and noting that intersection points in $A\pitchfork A^\uparrow$ are either those where the normal vector vanishes, contributing to $e(\nu_A)$, or occur as a pair of points in a neighborhood of a self-intersection of $A$.

    Using $a^*(TX)\cong T\S^k\oplus\nu_A$ we get
\begin{equation}\label{eq:w-as-inter-form}
    w^s_k(TX)(a) = w_k(a^*(TX)) =w_k(T\S^k)+w_k(\nu_A)  \equiv e(\nu_A) \equiv \langle a, a \rangle_X \pmod 2.
\end{equation}
    By definition, $w^s_k(TX)$ factors through the quotient $\pi_kX/\ker(h)$ and the twisted intersection form can also be restricted to a bilinear pairing on $\pi_kX/\ker(h) \times H_k(X;\Z^w)$. Since it takes values in the torsion-free group $\Z$, the formula~\eqref{eq:w-as-inter-form} implies that $w^s_k(TX)$ actually factors even further, namely through $(\pi_kX/\ker(h))/$torsion. By compactness of $X$ we know that $H_k(X;\Z)$ is finitely generated and so is its subgroup $\im(h)\cong \pi_kX/\ker(h)$. Therefore, we have shown that $w^s_k(TX)$ factors through a finitely generated torsion-free abelian group. As such groups are free, the required integer lift $W$ can be constructed by defining it on the free generators of this group. 
\end{proof}

We note that the 5-manifold $X=SU(3)/SO(3)$ is simply connected, has $\pi_2X\cong\Z/2$ with nontrivial $w^s_2(TX)$. So there is no integer lift $W$ and Lemmas~\ref{lem:TX-odd} and~\ref{lem:TX-even} indeed fail for $d=5,k=4$ and $d=5,k=3$ (so $d<2k$). As a consequence, our assumptions are the best possible for $d=5$. 

% Finally, let us consider $k=1$. If $X$ is a closed, orientable surface of genus $g>0$, the extension $\Z\hra\pi_1V_1(X)\sra\pi_1X$ is classified by $e(TX)=2-2g \in \Z\cong H^2(\pi_1X;\Z)$. Therefore, it splits if and only if $g=1$ (and obviously also for $X=S^2$). However, for the case of manifolds with nonempty boundary, relevant to us, we have the following.

% Finally, since the tangent bundle of a compact, orientable 3-manifold is trivial by the Wu formula, so are the frame bundles $V_2(X)\to X$ and the induced map on $\pi_1$ of a section of this bundle gives a splitting $\eta$ for $(\ell,d)=(1,3)$. 

\phantomsection
\printbibliography[heading=bibintoc]

@misc{KT-4dLBT,
  eprint = {2209.12015},
  archivePrefix={arXiv},
  primaryClass={math.GT},
  author = {Kosanović, Danica and Teichner, Peter},
  title = {A new approach to light bulb tricks: Disks in 4-manifolds},
  publisher = {arXiv},
  year = {2021},
  note = {Accepted in \textit{Duke Math. J.}}
}

@phdthesis{K-thesis,
  author = {Kosanović, Danica},
  title = {A geometric approach to the embedding calculus knot invariants},
  school = {Rheinische Friedrich-Wilhelms-Universität Bonn},
  year = 2020,
%   month = oct,
  url= {https://hdl.handle.net/20.500.11811/8651}
}

@misc{K-Dax,
  eprint = {2111.03041},
  archivePrefix={arXiv},
  primaryClass={math.GT},
  author = {Kosanović, Danica},
  title = {On homotopy groups of spaces of embeddings of an arc or a circle: the Dax invariant},
  publisher = {arXiv},
  year = {2021},
  note = {To appear in \textit{Trans. Amer. Math. Soc.}}
}

@article{Kervaire,
    author = {Michel A. Kervaire},
    title = {{Some nonstable homotopy groups of Lie groups}},
    volume = {4},
    journal = {Illinois Journal of Mathematics},
    number = {2},
    publisher = {Duke University Press},
    pages = {161 -- 169},
    year = {1960},
    doi = {10.1215/ijm/1255455861},
    % URL = {https://doi.org/10.1215/ijm/1255455861}
}

@article{framings2014,
   title={The number of framings of a knot in a 3-manifold},
   volume={23},
   ISSN={1793-6527},
%   url={http://dx.doi.org/10.1142/s0218216514500722},
   DOI={10.1142/s0218216514500722},
   number={13},
   journal={J. Knot Theory Ramifications},
   publisher={World Scientific Pub Co Pte Lt},
   author={Cahn, Patricia and Chernov, Vladimir and Sadykov, Rustam},
   year={2014},
%   month={11},
   pages={1450072}
}

@Book{Farb,
    Author = {Benson {Farb} and Dan {Margalit}},
    Title = {{A primer on mapping class groups}},
    ISBN = {978-0-691-14794-9},
    Pages = {xiv + 492},
    Year = {2011},
    Publisher = {Princeton, NJ: Princeton University Press},
    Language = {English},
    MSC2010 = {57-01 30-01 30F10 30F60 57M35 20F38 32G15 57M07 57N05 20F36 14H15},
    Zbl = {1245.57002}
}

@PHDTHESIS{Stiefel,
	author = {Stiefel, Eduard},
	publisher = {Orell Füssli},
	year = {1935},
% 	language = {de},
	copyright = {In Copyright - Non-Commercial Use Permitted},
	keywords = {TOPOLOGIE DER MANNIGFALTIGKEITEN; TOPOLOGY OF MANIFOLDS},
	size = {51 S.},
	address = {Zürich},
	DOI = {10.3929/ethz-a-000092403},
	title = {Richtungsfelder und Fernparallelismus in n-dimensionalen Mannigfaltigkeiten},
% 	Note = {SA aus: Commentarii mathematici helvetici, vol.8, fasc.4, 1935/36, pp.305-353. Diss. Math. ETH Zürich, Nr. 844, 0000. Ref.: Hopf, H. ; Korref.: Gonseth, F..},
% 	school = {ETH Zurich}
}

@article {Cerf-plongements,
    AUTHOR = {Cerf, Jean},
     TITLE = {Topologie de certains espaces de plongements},
   JOURNAL = {Bull. Soc. Math. France},
  FJOURNAL = {Bulletin de la Soci\'{e}t\'{e} Math\'{e}matique de France},
    VOLUME = {89},
      YEAR = {1961},
     PAGES = {227--380},
     ISSN = {0037-9484},
   MRCLASS = {57.20 (57.50)},
  MRNUMBER = {140120},
MRREVIEWER = {R. S. Palais},
       URL = {www.numdam.org/item?id=BSMF_1961__89__227_0},
}

@article{Cerf-applications,
     author = {Cerf, Jean},
     title = {Th\'eor\`emes de fibration des espaces de plongements. {Applications}},
     journal = {S\'eminaire Henri Cartan},
     note = {talk:8},
     publisher = {Secr\'etariat math\'ematique},
     volume = {15},
     date = {1962/1963},
     zbl = {0121.39903},
     mrnumber = {160230},
     language = {fr},
     url = {http://www.numdam.org/item/SHC_1962-1963__15__A2_0/}
}

@book {Cerf-diffeos,
    AUTHOR = {Cerf, Jean},
     TITLE = {Sur les diff\'{e}omorphismes de la sph\`ere de dimension trois
              {$(\Gamma _{4}=0)$}},
    SERIES = {Lecture Notes in Mathematics, No. 53},
 PUBLISHER = {Springer-Verlag, Berlin-New York},
      YEAR = {1968},
     PAGES = {xii+133},
   MRCLASS = {57.31},
  MRNUMBER = {0229250},
MRREVIEWER = {N. H. Kuiper},
}

@article{Dax,
     author = {Dax, Jean-Pierre},
     title = {\'Etude homotopique des espaces de plongements},
     journal = {Annales scientifiques de l'\'Ecole Normale Sup\'erieure},
     publisher = {Elsevier},
     volume = {4e s{\'e}rie, 5},
     number = {2},
     year = {1972},
     pages = {303-377},
     doi = {10.24033/asens.1230},
     zbl = {0251.58003},
     mrnumber = {47 \#9643},
    %  language = {fr},
    %  url = {http://www.numdam.org/item/ASENS_1972_4_5_2_303_0}
}

@article {Gabai-spheres,
    AUTHOR = {Gabai, David},
     TITLE = {The 4-dimensional light bulb theorem},
   JOURNAL = {J. Amer. Math. Soc.},
  FJOURNAL = {Journal of the American Mathematical Society},
    VOLUME = {33},
      YEAR = {2020},
    NUMBER = {3},
     PAGES = {609--652},
      ISSN = {0894-0347},
   MRCLASS = {57K40 (57N35)},
  MRNUMBER = {4127900},
       DOI = {10.1090/jams/920},
       %URL = {https://doi.org/10.1090/jams/920},
}

@article{ST-LBT,
    author = {Rob Schneiderman and Peter Teichner},
    title = {{Homotopy versus isotopy: Spheres with duals in 4-manifolds}},
    volume = {171},
    fjournal = {Duke Mathematical Journal},
    journal = {Duke Math. J.},
    number = {2},
    publisher = {Duke University Press},
    pages = {273 -- 325},
    keywords = {4-manifolds, homotopy, isotopy, spheres, Whitney moves},
    year = {2022},
    doi = {10.1215/00127094-2021-0016},
    % URL = {https://doi.org/10.1215/00127094-2021-0016}
}

@article {Gabai-disks,
    AUTHOR = {Gabai, David},
     TITLE = {Self-referential discs and the light bulb lemma},
   JOURNAL = {Comment. Math. Helv.},
  FJOURNAL = {Commentarii Mathematici Helvetici. A Journal of the Swiss
              Mathematical Society},
    VOLUME = {96},
      YEAR = {2021},
    NUMBER = {3},
     PAGES = {483--513},
      ISSN = {0010-2571},
   MRCLASS = {57K40 (57N35 57R40 57R52)},
  MRNUMBER = {4344778},
       DOI = {10.4171/cmh/518},
    %   URL = {https://doi.org/10.4171/cmh/518},
    % 2006.15450
}

@misc{Budney-Gabai,
    title   ={Knotted 3-balls in $S^4$},
    author  ={Budney, Ryan and Gabai, David},
    year    ={2019},
    eprint  ={1912.09029},
    archivePrefix   ={arXiv},
    primaryClass    ={math.GT}
}

@incollection{GKW,
	Author = {Goodwillie, Thomas G. and Klein, John R. and Weiss, Michael S.},
	Booktitle = {Surveys on surgery theory, {V}ol. 2},
	Mrclass = {57R40 (57R42 57R65)},
	Mrnumber = {1818775},
	Pages = {221--284},
	Publisher = {Princeton Univ. Press, Princeton, NJ},
	Series = {Ann. of Math. Stud.},
	Title = {Spaces of smooth embeddings, disjunction and surgery},
	Volume = {149},
	Year = {2001}
	}

@article{Smale,
	Author = {Smale, Stephen},
	Doi = {10.2307/1993113},
	Fjournal = {Transactions of the American Mathematical Society},
	Journal = {Trans. Amer. Math. Soc.},
	Mrclass = {55.00},
	Mrnumber = {0094807},
	Pages = {492--512},
	Title = {Regular curves on {R}iemannian manifolds},
	Volume = {87},
	Year = {1958}
	}

@article{Palais,
	Author = {Palais, Richard S.},
	Doi = {10.1007/BF02565942},
	Fjournal = {Commentarii Mathematici Helvetici},
	Journal = {Comment. Math. Helv.},
	Mrclass = {57.20},
	Mrnumber = {123338},
	Mrreviewer = {Morris W. Hirsch},
	Pages = {305--312},
	Title = {Local triviality of the restriction map for embeddings},
	Volume = {34},
	Year = {1960}
	}

@book{Hatcher,
  title={Algebraic Topology},
  author={Hatcher, Allen},
  year={2002},
  publisher={Cambridge University Press}
}

@article {Hatcher-Quinn,
    AUTHOR = {Hatcher, Allen and Quinn, Frank},
     TITLE = {Bordism invariants of intersections of submanifolds},
   JOURNAL = {Trans. Amer. Math. Soc.},
  FJOURNAL = {Transactions of the American Mathematical Society},
    VOLUME = {200},
      YEAR = {1974},
     PAGES = {327--344},
      ISSN = {0002-9947},
   MRCLASS = {57C35 (57D40)},
  MRNUMBER = {353322},
MRREVIEWER = {R. E. Stong},
       DOI = {10.2307/1997261},
    %   URL = {https://doi.org/10.2307/1997261},
}

@article{Hirsch,
 ISSN = {00029947},
 URL = {http://www.jstor.org/stable/1993453},
 author = {Morris W. Hirsch},
 journal = {Transactions of the American Mathematical Society},
 number = {2},
 pages = {242--276},
 publisher = {American Mathematical Society},
 title = {Immersions of Manifolds},
 volume = {93},
 year = {1959}
}

@article{Grant, 
    title={On Self-Intersection Invariants}, 
    volume={55}, 
    DOI={10.1017/S0017089512000481}, 
    number={2}, 
    journal={Glasgow Mathematical Journal}, 
    publisher={Cambridge University Press}, 
    author={Grant, Mark}, 
    year={2013}, 
    pages={259–273}
}

@book {Rolfsen,
    AUTHOR = {Rolfsen, Dale},
     TITLE = {Knots and links},
    SERIES = {Mathematics Lecture Series, No. 7},
 PUBLISHER = {Publish or Perish, Inc., Berkeley, Calif.},
      YEAR = {1976},
     PAGES = {ix+439},
   MRCLASS = {55-01},
  MRNUMBER = {0515288},
}

@misc{Knudsen-Kupers,
      title={Embedding calculus and smooth structures}, 
      author={Ben Knudsen and Alexander Kupers},
      year={2020},
      eprint={2006.03109},
      archivePrefix={arXiv},
      primaryClass={math.AT}
}

@misc{Hatcher-exposition,
      title={An exposition of the Madsen-Weiss theorem}, 
      author={Allen Hatcher},
      year={2011},
      url={http://www.math.cornell.edu/~hatcher/Papers/MW.pdf}
}

@article{Gramain,
     author = {Gramain, Andr\'e},
     title = {Le type d'homotopie du groupe des diff\'eomorphismes d'une surface compacte},
     journal = {Annales scientifiques de l'\'Ecole Normale Sup\'erieure},
     pages = {53--66},
     publisher = {Elsevier},
     volume = {4e s{\'e}rie, 6},
     number = {1},
     year = {1973},
     doi = {10.24033/asens.1242},
     zbl = {0265.58002},
     mrnumber = {48 #5116},
     language = {fr},
     % url = {http://www.numdam.org/articles/10.24033/asens.1242/}
}

@article {Hatcher-SmaleConj,
    AUTHOR = {Hatcher, Allen E.},
     TITLE = {A proof of the {S}male conjecture, {${\rm Diff}(S^{3})\simeq
              {\rm O}(4)$}},
   JOURNAL = {Ann. of Math. (2)},
  FJOURNAL = {Annals of Mathematics. Second Series},
    VOLUME = {117},
      YEAR = {1983},
    NUMBER = {3},
     PAGES = {553--607},
      ISSN = {0003-486X},
   MRCLASS = {57M99 (57S05)},
  MRNUMBER = {701256},
MRREVIEWER = {R. C. Kirby},
       DOI = {10.2307/2007035},
       % URL = {https://doi.org/10.2307/2007035},
}

@article {Litherland,
    AUTHOR = {Litherland, Rick},
     TITLE = {A generalization of the lightbulb theorem and {PL}
              {$I$}-equivalence of links},
   JOURNAL = {Proc. Amer. Math. Soc.},
  FJOURNAL = {Proceedings of the American Mathematical Society},
    VOLUME = {98},
      YEAR = {1986},
    NUMBER = {2},
     PAGES = {353--358},
      ISSN = {0002-9939},
   MRCLASS = {57Q45},
  MRNUMBER = {854046},
MRREVIEWER = {V. G. Turaev},
       DOI = {10.2307/2045711},
       % URL = {https://doi.org/10.2307/2045711},
}

@article {Goodwillie-thesis,
    AUTHOR = {Goodwillie, Thomas Gehret},
     TITLE = {A multiple disjunction lemma for smooth concordance
              embeddings},
   JOURNAL = {Mem. Amer. Math. Soc.},
  FJOURNAL = {Memoirs of the American Mathematical Society},
    VOLUME = {86},
      YEAR = {1990},
    NUMBER = {431},
     PAGES = {viii+317},
      ISSN = {0065-9266},
   MRCLASS = {57R40 (57R45)},
  MRNUMBER = {1015675},
MRREVIEWER = {Kiyoshi Igusa},
       DOI = {10.1090/memo/0431},
       % URL = {https://doi.org/10.1090/memo/0431},
}

@article {Haefliger-plong,
    AUTHOR = {Haefliger, Andr\'{e}},
     TITLE = {Plongements diff\'{e}rentiables de vari\'{e}t\'{e}s dans
              vari\'{e}t\'{e}s},
   JOURNAL = {Comment. Math. Helv.},
  FJOURNAL = {Commentarii Mathematici Helvetici},
    VOLUME = {36},
      YEAR = {1961},
     PAGES = {47--82},
      %ISSN = {0010-2571,1420-8946},
   MRCLASS = {57.20},
  MRNUMBER = {145538},
MRREVIEWER = {Morris\ W.\ Hirsch},
       DOI = {10.1007/BF02566892},
       % URL = {https://doi.org/10.1007/BF02566892},
}

\vspace{5pt}
\hrule

\end{document}